\documentclass[10pt]{article}
\usepackage[singlespacing]{setspace}
\usepackage{amsmath, amssymb, amscd, amsthm, amsfonts}
\usepackage{mathtools}
\usepackage[mathscr]{euscript}
\usepackage{mathrsfs}
\usepackage{enumitem}
\usepackage{graphicx}
\usepackage{tocloft}
\usepackage{caption}
\usepackage{subcaption}
\RequirePackage{doi}
\usepackage{hyperref}
\hypersetup{
colorlinks = true,
urlcolor   = blue,
citecolor  = blue,
linkcolor  = blue,
}
\usepackage{xcolor}
\usepackage{scalerel,stackengine}
\usepackage[colorinlistoftodos]{todonotes}
\usepackage{authblk}
\usepackage{fancyhdr}

\newlength\FHoffset
\fancyheadoffset{\FHoffset}

\newlength\FHleft
\newlength\FHright

\renewcommand{\headrulewidth}{1.0pt} 
\newbox\FHline
\setbox\FHline=\hbox{\hsize=\paperwidth%
  \hspace*{\FHleft}%
  \rule{\dimexpr\headwidth-\FHleft-\FHright\relax}{\headrulewidth}\hspace*{\FHright}%
}

\title{\textbf{Three dimensional branching pipe flows for optimal scalar transport between walls}}
\author[]{\large \textbf{Anuj Kumar}\footnote{Department of Applied Mathematics, University of California Santa Cruz, CA 95064. \textit{Email:}  \href{mailto:akumar43@ucsc.edu}{akumar43@ucsc.edu}. }}
\date{}

\newtheoremstyle{mystyle}
  {}
  {}
  {\itshape}
  {}
  {\bfseries}
  {.}
  { }
  {\thmname{#1}\thmnumber{ #2}\thmnote{ (#3)}}

\theoremstyle{mystyle}
\newtheorem{theorem}{Theorem}[section]
\newtheorem{proposition}[theorem]{Proposition}
\newtheorem{lemma}{Lemma}[section]
\newtheorem{corollary}[theorem]{Corollary}
\newtheorem{conjecture}[theorem]{Conjecture}
\newtheorem{definition}{Definition}[section]

\theoremstyle{definition}
\newtheorem{remark}{Remark}[section]

\newcommand\norm[1]{\left\lVert#1\right\rVert}

\newcommand{\bs}[1]{\boldsymbol{#1}}
\newcommand{\wh}[1]{\widehat{#1}}
\newcommand{\wt}[1]{\widetilde{#1}}
\newcommand{\ol}[1]{\overline{#1}}

\stackMath
\newcommand\reallywidecheck[1]{%
\savestack{\tmpbox}{\stretchto{%
  \scaleto{%
    \scalerel*[\widthof{\ensuremath{#1}}]{\kern-.6pt\bigwedge\kern-.6pt}%
    {\rule[-\textheight/2]{1ex}{\textheight}}
  }{\textheight}%
}{0.5ex}}%
\stackon[1pt]{#1}{\scalebox{-1}{\tmpbox}}%
}

 \DeclareMathOperator{\diverge}{div}
 \DeclareMathOperator{\dist}{dist}
\DeclareMathOperator\supp{supp}

\def\Xint#1{\mathchoice
{\XXint\displaystyle\textstyle{#1}}%
{\XXint\textstyle\scriptstyle{#1}}%
{\XXint\scriptstyle\scriptscriptstyle{#1}}%
{\XXint\scriptscriptstyle\scriptscriptstyle{#1}}%
\!\int}
\def\XXint#1#2#3{{\setbox0=\hbox{$#1{#2#3}{\int}$ }
\vcenter{\hbox{$#2#3$ }}\kern-.6\wd0}}

\def\dashint{\Xint-}

\numberwithin{equation}{section}

\begin{document}

\maketitle

\vspace{-0.5cm}
\begin{abstract}
We consider the problem of ``wall-to-wall optimal transport'' in which we attempt to maximize the transport of a passive temperature field between hot and cold plates. Specifically, we optimize the choice of the divergence-free velocity field in the advection-diffusion equation subject to an enstrophy constraint (which can be understood as a constraint on the power required to generate the flow). Previous work established an a priori upper bound on the transport, scaling as the 1/3-power of the flow's enstrophy. Recently,  Tobasco \& Doering
(\href{https://doi.org/10.1103/PhysRevLett.118.264502}{\textit{Phys. Rev. Lett.} vol.118, 2017, p.264502}) and Doering \& Tobasco (\href{https://doi.org/10.1002/cpa.21832}{\textit{Comm. Pure Appl. Math.} vol.72, 2019, p.2385--2448}) constructed self-similar two-dimensional steady branching flows saturating this bound up to a logarithmic correction. This logarithmic correction appears to arise due to a topological obstruction inherent to two-dimensional steady branching flows. We present a construction of three-dimensional ``branching pipe flows" that eliminates the possibility of this logarithmic correction and therefore identifies the optimal scaling as a clean 1/3-power law. Our flows resemble previous numerical studies of the three-dimensional wall-to-wall problem by Motoki, Kawahara \& Shimizu (\href{https://doi.org/10.1017/jfm.2018.557}{\textit{J. Fluid Mech.} vol.851, 2018, p.R4}). We also discuss the implications of our result to the heat transfer problem in Rayleigh--B\'enard convection and the problem of anomalous dissipation in a passive scalar.
\end{abstract}

{
  \hypersetup{linkcolor=black}
  \tableofcontents
}

\section{Introduction}\label{section-introduction}
\subsection{Motivation}
An important subdiscipline of thermal engineering is devoted to the design of heat exchangers, ventilation systems, air-conditioning systems, refrigeration systems, boilers, and chemical reactors \cite{arora2000refrigeration, jakobsen2008chemical, thulukkanam2013heat, alam2018comprehensive}. A fundamental challenge in this field is how to transport heat from a hot surface to a cold surface by moving the fluid using actuators such as fans or pumps, which can advect heat at a quicker rate than pure conduction.  For most practical purposes, one would of course like to do so in the most economical way, minimizing the power supplied to the actuators. In the design of the systems described above, we would therefore like to know the answers to the following questions:
\begin{enumerate}[label=\textbf{(\Alph*)}]
\item What is the optimal heat transfer rate as a function of the power supplied?
\item What is the corresponding placement of fans/pumps which maximizes the heat transfer for a given amount of power supplied?
\end{enumerate}

To model the problem mathematically, we use the forced Navier--Stokes equation to describe the flow of an incompressible fluid:
\begin{eqnarray}
\partial_t \bs{u} + \bs{u} \cdot \nabla \bs{u} = - \nabla p + \nu \Delta \bs{u} + \bs{f} \quad \text{in } \Omega, \nonumber
\label{forced NS eqn}
\end{eqnarray}
where $\Omega$ is a bounded domain with smooth boundaries and $\nu$ is the viscosity of the fluid. We assume that the fluid satisfies a no-slip boundary condition on the surface, i.e., $\left. \bs{u}\right|_{\partial \Omega} = \bs{0}$. In this mathematical formulation of the problem, the question of interest now involves finding  the optimal design for the force $\bs{f}$ that maximizes the heat transfer with a restricted mean power supply $\mathscr{P}^\ast$. Denoting the volume average and the long-time volume average, respectively, as 
\begin{eqnarray}
\dashint_{\Omega} [\;\cdot\;] \; \textrm{d}\bs{x} = \frac{1}{|\Omega|}\int_{\Omega} [\;\cdot\;] \; \textrm{d}\bs{x} \qquad \text{and} \qquad \langle[\;\cdot\;]\rangle = \limsup_{\tau \to \infty} \frac{1}{\tau} \int_{0}^\tau \dashint_{\Omega} [\;\cdot\;] \; \textrm{d}\bs{x} \textrm{d}t, \nonumber
\end{eqnarray}
we can express the constraint on the mean power supply as $\langle \bs{f} \cdot \bs{u} \rangle \leq \mathscr{P}^\ast$.

 Assuming the velocity field stays smooth and the kinetic energy of flow stays bounded in time, then the long-time spatial average of the energy equation leads to $$\langle \bs{f} \cdot \bs{u} \rangle = \nu \langle |\nabla \bs{u}|^2 \rangle.$$ Physically, this means that the work done by the force $\bs{f}$ to move the fluid is eventually dissipated viscously. It also shows that instead fixing the power supply, one can equivalently impose a constraint on the enstrophy of the flow, i.e., $\langle |\nabla \bs{u}|^2 \rangle \leq \nu^{-1} \mathscr{P}^\ast$. 
 
 The advantage of formulating the constraint in terms of the enstrophy is that we can from here on ignore the momentum equation entirely. We can simply ask, what is the flow $\bs{u}$ that maximizes the heat transfer, for a given bound on the enstrophy ($\langle |\nabla \bs{u}|^2 \rangle \leq \nu^{-1} \mathscr{P}^\ast$). Once that flow $\bs{u}$ is found, the corresponding forcing $\bs{f}$ can then be computed from (\ref{forced NS eqn}).  Whether the optimal flow $\bs{u}$ obtained in this manner is dynamically stable is a separate question that we will not address in this paper.
 
 Beyond the primary engineering motivation, the optimal heat transport problem considered in this paper is also inspired by two problems: (1) anomalous dissipation in a passive scalar, (2) Rayleigh--B\'enard convection. These problems, and their relationship with the optimal transport problem investigated here, will be discussed in section \ref{Discussion}. 

\subsection{Problem setup}
Although the problem discussed above is very general,  we now focus on a special case in the simplest possible geometry namely the transport of a passive scalar $T$ (which we refer to as temperature) by a flow field $\bs{u}$ between two parallel walls held at different constant values of $T$. We assume that the flow field $\bs{u}$ is incompressible and satisfies no-slip boundary conditions at the walls, which in the wall-normal coordinates are located at $z = -H/2$ and $z = H/2$, where $H$ denotes the distance between the walls. The temperature field evolves according to the advection-diffusion equation 
\begin{eqnarray}
\partial_t T + \bs{u} \cdot \nabla T - \kappa \Delta T = 0, 
\label{Introduction: time dependent advection-diffusion equation}
\end{eqnarray}
 and satisfies Dirichlet boundary conditions
\begin{eqnarray}
T = T_B \; \text{ at } \; z = -H/2, \quad T = T_T \; \text{ at } \; z = H/2. 
\label{Introduction: temp bcs}
\end{eqnarray}
In (\ref{Introduction: time dependent advection-diffusion equation}), $\kappa$ is the thermal diffusivity, and without loss of generality, we choose $T_B > T_T$ in (\ref{Introduction: temp bcs}). For simplicity, we consider the horizontal directions $x$ and $y$ to be periodic with length $l_x$ and $l_y$. The domain of interest is thus given by $\Omega \coloneqq \mathbb{T}_{l_x} \times \mathbb{T}_{l_y} \times (-H/2, H/2).$

For a given flow field $\bs{u}$, we define the corresponding rate of heat transfer as 
\begin{eqnarray}
Q(\bs{u}) \coloneqq \left\langle u_z T - \kappa \frac{\partial T}{\partial z} \right\rangle = \langle u_z T \rangle + \frac{\kappa (T_B - T_T)}{H}. \nonumber
\label{Introduction: Heat transfer unsteady}
\end{eqnarray}
By performing the long-time horizontal average of equation (\ref{Introduction: time dependent advection-diffusion equation}), one can show that $Q(\bs{u})$ is equal to the heat flux at the top or the bottom boundary, hence the definition. Furthermore, by multiplying (\ref{Introduction: time dependent advection-diffusion equation}) with $T$ and performing the long-time volume average, one can alternatively express the rate of heat transfer $Q(\bs{u})$ as
\begin{eqnarray}
Q(\bs{u}) = \frac{\kappa H}{T_B - T_T} \langle |\nabla T|^2 \rangle. \nonumber
\label{Introduction: Q for time indep eq}
\end{eqnarray}

The question of optimal heat transport described in the previous subsection seeks to find the maximum possible value of the heat transfer over all incompressible flow fields satisfying the no-slip boundary condition and the enstrophy constraint:
\begin{eqnarray}
Q_{\max}(\mathscr{P}^\ast) \coloneqq \sup_{\substack{\bs{u}(t, \bs{x}) \\ \nabla \cdot \bs{u} = 0, \; \left. \bs{u} \right|_{\partial \Omega} = \bs{0} \\ \langle |\nabla \bs{u}|^2 \rangle \leq \nu^{-1} \mathscr{P}^\ast}} Q(\bs{u}). \nonumber
\end{eqnarray}

Before we proceed further, we nondimensionalize the problem by making the following transformations, respectively, for the position, time, velocity field, temperature and the heat transfer:
\begin{eqnarray}
\bs{x} \to H \bs{x}, \quad t \to \frac{H^2}{\kappa} t, \quad \bs{u} \to \frac{\kappa}{H} \bs{u}, \quad T \to (T_B - T_T) T + T_T, \quad Q \to \frac{\kappa (T_B - T_T)}{H}Q.
\label{nondim scales}
\end{eqnarray}
We continue to denote the nondimensional horizontal periodic lengths with $l_x$ and $l_y$. After the rescaling (\ref{nondim scales}), a single nondimensional parameter remains, namely the nondimensional power given by 
$$\mathscr{P} = \mathscr{P}^\ast \frac{H^4}{\nu \kappa^2}$$
which can be increased by either increasing the dimensional power supply $\mathscr{P}^\ast$ and the domain size $H$ or by decreasing the thermal diffusivity $\kappa$ and viscosity $\nu$ of the fluid. After the nondimensionalization the enstrophy constraint becomes $\langle |\nabla \bs{u}|^2 \rangle \leq \mathscr{P}$.

As the problem of optimal heat transport can be considered both in two and three dimensions, let us define
\addtocounter{equation}{1}
\begin{align}
\Omega^{\, 2D} \coloneqq \mathbb{T}_{l_x} \times (-1/2, 1/2), \qquad \qquad \Omega^{\, 3D} \coloneqq \mathbb{T}_{l_x} \times \mathbb{T}_{l_x} \times (-1/2, 1/2).
\tag{\theequation a-b}
\label{2D 3D domain}
\end{align}
In the introduction, $\Omega$ is used to mean either $\Omega^{2D}$ or $\Omega^{3D}$ except in places where the distinction is required, in which case we will make the reference explicit. In rest of the paper $\Omega$ will only mean $\Omega^{3D}$. Without the loss of generality, we assume that the aspect ratio of the domain satisfies $l_x \leq l_y$. Next, we explicitly formulate the steady, followed by the unsteady version of the optimal heat transport problem.

\subsubsection{Steady case}
In the steady case, we seek
\begin{eqnarray}
Q^s_{\max}(\mathscr{P}) = \sup_{\substack{\bs{u} \in L^\infty(\Omega) \\ \nabla \cdot \bs{u} = 0, \; \left. \bs{u} \right|_{\partial \Omega} = \bs{0} \\ \dashint_{\Omega} |\nabla \bs{u}|^2 \, {\rm d} \bs{x} \leq \mathscr{P}}} Q(\bs{u}) \qquad \qquad \qquad \text{where} \qquad Q(\bs{u}) = \dashint_{\Omega} |\nabla T|^2 \, {\rm d} \bs{x}, 
\label{Qmax steady advec-diff}
\end{eqnarray}
and $T$ solves the steady advection-diffusion equation with Dirichlet boundary conditions
\begin{eqnarray}
\begin{rcases}
\qquad \qquad \bs{u} \cdot \nabla T - \Delta T = 0,  \\
T = 1 \; \text{ at } \; z = -1/2, \quad T = 0 \; \text{ at } \; z = 1/2. \quad 
\end{rcases}
\label{steady advec-diff}
\end{eqnarray}

\subsubsection{Unsteady case}
In the unsteady case, we seek
\begin{eqnarray}
Q^u_{\max}(\mathscr{P}) =\sup_{\substack{\bs{u} \in L^\infty([0, \infty) \times \Omega) \\ \nabla \cdot \bs{u} = 0, \; \left. \bs{u} \right|_{\partial \Omega} = \bs{0} \\ \langle |\nabla \bs{u}|^2 \rangle \leq \mathscr{P}}} Q(\bs{u}) \qquad \qquad \qquad \text{where} \qquad Q(\bs{u}) = \langle |\nabla T|^2 \rangle, 
\label{Qmax unsteady advec-diff}
\end{eqnarray}
and $T$ solves the unsteady advection-diffusion equation with Dirichlet boundary conditions
\begin{eqnarray}
\begin{rcases}
\qquad \qquad \qquad \partial_t T + \bs{u} \cdot \nabla T - \Delta T = 0,  \\
\qquad \qquad \qquad T = T_0 \in L^2(\Omega) \; \text{ at } \; t = 0, \\
T = 1 \; \text{ at } \; z = -1/2, \quad T = 0 \; \text{ at } \; z = 1/2 \quad \forall \; t \in (0, \infty). \quad 
\end{rcases}
\label{unsteady advec-diff}
\end{eqnarray}

\begin{remark}
It is clear that for every $\mathscr{P}$, we have the inequality $Q^s_{\max} \leq Q^u_{\max}$. Therefore, any upper bound on $Q^u_{\max}$ provides an upper bound on $Q^s_{\max}$. Similarly, any lower bound on $Q^s_{\max}$ is a lower bound on $Q^u_{\max}$ as well.
\label{rel Qs Qu}
\end{remark}
\begin{remark}
The values of $Q_{\max}^s$ and $Q_{\max}^u$ for the three-dimensional problem are larger than their corresponding values for the two-dimensional problem. This is because any two-dimensional solution of the advection diffusion equation is also a solution in three-dimensions by an extension that is invariant in the third direction.
\end{remark}
\begin{remark}
 In the unsteady case, the quantity $Q(\bs{u})$ does not depend on the initial condition $T_0$ as long as this initial condition belongs to $L^2(\Omega)$. Physically, this means that the dependence of the solution $T$ on the initial data is lost at long times because of the presence of diffusion.
\end{remark}
    
For both the steady and unsteady cases, in their two- and three-dimensional versions, the questions of prime importance are:
\begin{enumerate}[label=\textbf{(\Alph*)}]
\item How the maximum heat fluxes $Q_{\max}^s$ and $Q_{\max}^u$ scale  as a function of the input power $\mathscr{P}$ for asymptotically large values of $\mathscr{P}$? 
\item What does the structure of the flow fields that transfer heat ``most efficiently'' look like?
\end{enumerate}

In this paper, we investigate these questions for the steady case only. The unsteady case is also of great importance and will be considered in a future study. 

\subsection{Previous work and the present results}
The problem of optimal heat transport between parallel walls, as described above, was first introduced in the work of Hassanzadeh, Chini and Doering \cite{hassanzadeh2014wall} whose motivation was to improve previously known upper bounds on heat transfer in porous medium convection \cite{doering98porus} and Rayleigh--B\'enard convection \cite{doering1996variational, Plasting03Couette, Whitehead11Ultimate, Wen15timestepping}. Using numerical and matched asymptotic techniques, they studied the problem in two dimensions, applying either the energy or enstrophy constraints on the velocity field which satisfies stress-free boundary conditions.

Their initial investigation has since inspired several studies of optimal heat transport between differentially heated plates \cite{tobasco2017optimal, motoki2018optimal, doering2019optimal, souza2020wall} and slightly different problem of optimal cooling of a fluid subjected to a given volumetric heating \cite{marcotte2018optimal, iyer2021bounds, tobasco2021optimal}. Of all these studies, the three of particular interest to the current paper are \cite{tobasco2017optimal}, \cite{motoki2018optimal} and \cite{doering2019optimal}. They all investigate the same problem considered in this paper, i.e., optimal heat transport between parallel boundaries by incompressible flows satisfying no-slip boundary conditions with an enstrophy constraint. Doering \& Tobasco (\cite{doering2019optimal}) derived an upper bound on the maximum possible heat transfer, and showed that any flow satisfying the required constraint cannot transport heat faster than a rate proportional to the enstrophy to the power of $1/3$, i.e.,
\begin{eqnarray}
Q^u_{\max}(\mathscr{P}) \leq C^\prime \mathscr{P}^{1/3} \qquad \text{for} \quad \mathscr{P} \geq C^{\prime \prime},
\label{upper bound Qu}
\end{eqnarray}
where $C^\prime$ is a universal constant but $C^{\prime \prime}$ depends on the aspect ratio. This upper bound is valid both in two and three dimensions and applies to $Q^s_{\max}(\mathscr{P})$ as well (see Remark \ref{rel Qs Qu}).
The same bound had been proved before in the context of Rayleigh--B\'enard convection in at least three different ways: using the variational principle of Howard (\cite{MR159530, busse1969howard}), the background method of Doering \& Constantin (\cite{ doering1996variational, Plasting03Couette}) and more recently by Seis (\cite{ seis2015bound}).

Complimentary to their upper bound, \cite{tobasco2017optimal} and \cite{doering2019optimal} constructed two-dimensional steady branching flows (in which the flow structures have increasingly fine scales as one approaches the boundary, see figure \ref{conve rolls 2D steady branching b}) and showed that the upper bound could be attained up to an unknown logarithmic correction. More specifically, they showed 
\begin{eqnarray}
\frac{\mathscr{P}^{1/3}}{\log^{4/3} \mathscr{P}} \lesssim Q^s_{\max}(\mathscr{P}). \nonumber
\label{Strategy: DT bound lower}
\end{eqnarray}
Soon after the work of \cite{tobasco2017optimal}, Motoki, Kawahara and Shimizu  \cite{motoki2018optimal}, through a numerical optimization procedure, discovered complicated but rather beautiful three-dimensional steady branching flows (depending on $\mathscr{P}$) that appears to display a heat transfer rate $Q^s_{\max}(\mathscr{P}) \sim \mathscr{P}^{1/3}$.

In this paper, we rigorously show that $\mathscr{P}^{1/3} \lesssim Q^s_{\max}(\mathscr{P})$ by constructing three-dimensional steady branching pipe flows. Our main result is:
\begin{theorem}[Steady three-dimensional case]
\label{Main theorem steady case} Let $\Omega$ be $\Omega^{\, 3D}$ as defined in (\ref{2D 3D domain}). Then there exists two positive constants $\mathscr{P}_0$ and $C$ such that $Q^s_{\max}$, as defined in (\ref{Qmax steady advec-diff}), obeys the following lower bound: $$C \mathscr{P}^{1/3} \leq Q^s_{\max}(\mathscr{P})$$ for
$\mathscr{P}_0 \leq \mathscr{P}.$ The constants  $\mathscr{P}_0$ and $C$ depends on $l_x$ as follows:
\begin{eqnarray}
\mathscr{P}_0(l_x) = \frac{1 + l_x^2}{l_x^2} \mathscr{P}^\prime_0, \quad C(l_x) = \frac{l_x^{8/3}}{1+l_x^{8/3}} C^\prime, \nonumber
\end{eqnarray}
where $\mathscr{P}_0^\prime, C^\prime > 0$ are two universal constants.
\end{theorem}

\begin{remark}
Combining the result of Theorem \ref{Main theorem steady case} with upper bound (\ref{upper bound Qu}) and Remark \ref{rel Qs Qu}, we fully characterize the exact behavior of maximum heat transfer in three dimensions. In particular, we have $Q^s_{\max} \sim \mathscr{P}^{1/3}$ (as a result $Q^u_{\max} \sim \mathscr{P}^{1/3}$) in three dimensions. Whether $Q^s_{\max} \sim \mathscr{P}^{1/3}$ in two dimensions as well is an open problem (see Conjecture \ref{weak conj}).
\end{remark}

\begin{remark}
It clear that if $l_x \geq 1$ then $\mathscr{P}_0$ and $C$ are bounded from below by two positive constants independent of $l_x$. Therefore, assuming $l_x \geq 1$, i.e., for sufficiently wide domains, we can also restate the above theorem where $\mathscr{P}_0$ and $C$ are two positive constants independent of any parameter.
\end{remark}

\begin{remark}
 We consider here a periodic setting in the $x$ and $y$ directions. As the flows that we construct to prove the theorem have a compact support in space, the theorem remains true if $\Omega$ is a closed box of size $l_x$ and $l_y$ with insulating and no-slip side walls. 
\end{remark}

\subsection{Overview and philosophy of the proof}
\subsubsection{The variational principle}
The proof of Theorem \ref{Main theorem steady case} relies on a variational principle for the heat transfer derived by \cite{doering2019optimal} which was inspired by the work in homogenization theory such as of \cite{Marco91effective, Fannjiang94enhanced} about estimating the effective diffusivity in a random or periodic array of vortices. To state the result, we start by defining two admissible sets:
\begin{subequations}
\begin{eqnarray}
&& \mathcal{A}^s \coloneqq L^\infty(\Omega; \mathbb{R}^3) \cap H_0^1(\Omega; \mathbb{R}^3), \label{admissible vel steady}  \\
&& \mathcal{X}^s \coloneqq H^1_0(\Omega). \label{admissible xi steady}
\end{eqnarray} 
\end{subequations}
For steady velocity fields, the variational principle associated with the maximization of heat transfer can be stated as
\begin{proposition}[\cite{doering2019optimal}]
\label{The variational principle: main corollary}
For $Q^s_{\max}$ given by (\ref{Qmax steady advec-diff}), we have
\begin{eqnarray}
\label{Q max P steady}
Q^s_{\max}(\mathscr{P}) - 1 =  \sup_{\substack{\bs{u} \in \mathcal{A}^s  \\ \nabla \cdot \bs{u} = 0 }} \sup_{\substack{\xi \in \mathcal{X}^s \\ \xi \not\equiv 0}} \frac{\left(\dashint_{\Omega} u_z \xi \, {\rm d} \bs{x} \right)^2}{\dashint_{\Omega} |\nabla \Delta^{-1} (\bs{u} \cdot \nabla \xi)|^2 \, {\rm d} \bs{x} + \frac{1}{\mathscr{P}} \dashint_{\Omega} |\nabla \bs{u}|^2 \, {\rm d} \bs{x} \dashint_{\Omega} |\nabla \xi|^2 \, {\rm d} \bs{x}}.
\label{steady var prin}
\end{eqnarray}
\end{proposition}
In (\ref{steady var prin}), $\Delta^{-1}$ denotes the inverse Laplacian operator in $\Omega$ corresponding to the homogeneous Dirichlet boundary conditions. For completeness, we provide a derivation of this variation principle in the Appendix \ref{Derivation var prin}. The derivation is taken from (\cite{doering2019optimal}).

From the variational principles (\ref{steady var prin}), we see that any choice of admissible velocity field $\bs{u}$ and scalar field $\xi$ provides a lower bound on the heat transfer. Our goal, therefore, is to find a ``good" flow field $\bs{u}$ (depending on $\mathscr{P}$), and a corresponding $\xi$, for which the dependence on $\mathscr{P}$ of the lower bound obtained matches that of the theoretical upper bound (\ref{upper bound Qu}), namely, $\mathscr{P}^{1/3}$.

We closely analyze each term involved in the variational principle (\ref{steady var prin}). We label them
\addtocounter{equation}{1}
\begin{align}
I = \left( \dashint_{\Omega} u_z \xi \, {\rm d} \bs{x} \right)^2, \qquad II = \dashint_{\Omega} |\nabla \Delta^{-1} (\bs{u} \cdot \nabla \xi)|^2 \, {\rm d} \bs{x}, \qquad III = \frac{1}{\mathscr{P}} \dashint_{\Omega} |\nabla \bs{u}|^2 \, {\rm d} \bs{x} \dashint_{\Omega} |\nabla \xi|^2 \, {\rm d} \bs{x},
\tag{\theequation a-c}
\label{terms var prin}
\end{align}
and identify them as the transport term ($I$), the nonlocal term ($II$) and the dissipation term ($III$), respectively.

 \begin{figure}
\centering
\begin{tabular}{lc}
\begin{subfigure}{0.5\textwidth}
\centering
 \includegraphics[scale = 0.6]{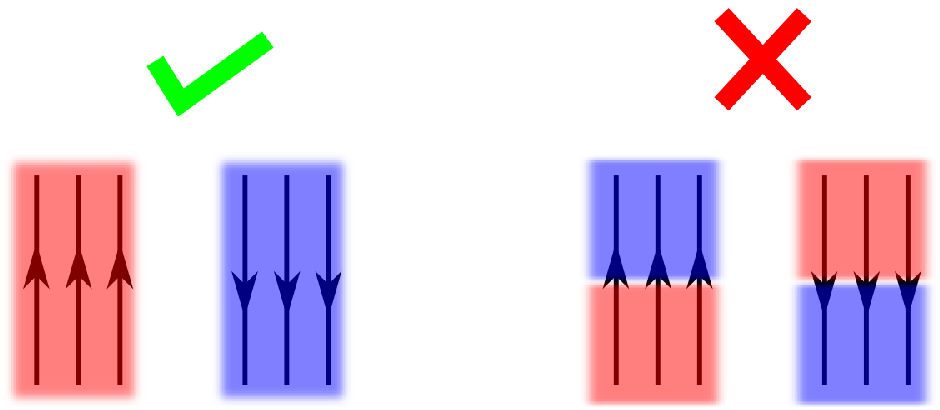}
\caption{}
\label{strategies I II a}
\end{subfigure} &
\begin{subfigure}{0.5\textwidth}
\centering
 \includegraphics[scale = 0.6]{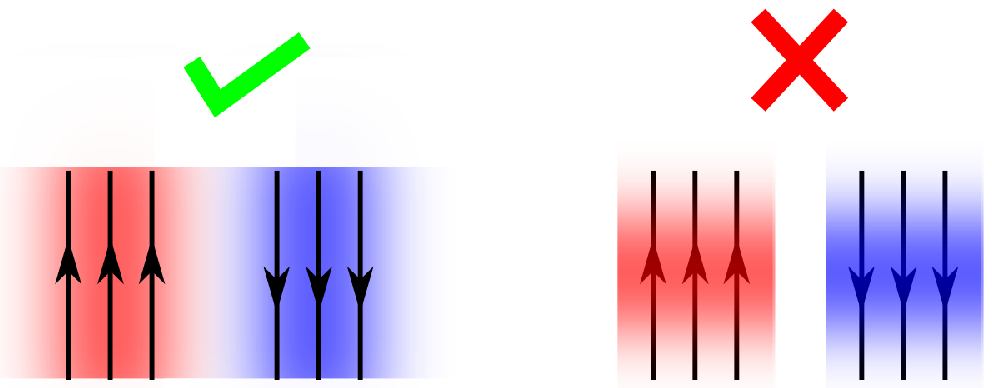}
\caption{}
\label{strategies I II b}
\end{subfigure}
\end{tabular}
\caption{Panel (a) illustrates good and bad strategies to maximize term $I$ defined in (\ref{terms var prin}). In the good scenario $\xi$ is positive (indicated by red color) where the flow is moving upward (positive $z$-direction) and is negative (blue color) where the flow is moving downward. Therefore, $u_z$ and $\xi$ are positively correlated. This is not the case in the bad scenario. Panel (b) illustrates  good and bad strategies to minimize term $II$. In the good scenario $\nabla \xi$ is perpendicular to $\bs{u}$, hence $\bs{u} \cdot \nabla \xi \equiv 0$, i.e., $\xi$ is constant along the streamlines and therefore the term $II$ zero. In the bad scenario $\nabla \xi$ is parallel to $\bs{u}$, so the term $II$ is nonzero.}
 \label{strategies I II}
\end{figure}

In order to obtain a good lower bound, we would ideally like to choose $\bs{u}$ and $\xi$ to maximize the right-hand side of (\ref{steady var prin}) as much as possible. This, in turn, means we should aim to maximize $I$ and minimize $II$ and $III$. 

To maximize $I$, we should choose a $\bs{u}$ such that its $z$-component is ``positively correlated'' with the $\xi$ field. Figure \ref{strategies I II a} shows examples of a good and bad scenario. To minimize $II$, we should make a choice such that $\bs{u}$ is perpendicular to $\nabla \xi$ in most of the domain, which can alternatively be stated as $\xi$ should be constant along the streamlines of the flow $\bs{u}$. Figure \ref{strategies I II b} shows examples of a good and bad scenario.

Our aim at this point is to provide heuristic but compelling arguments why trial velocity profiles such as (i) standard convection rolls and (ii) the two-dimensional steady branching flows considered by \cite{tobasco2017optimal, doering2019optimal} are not sufficient to prove Theorem \ref{Main theorem steady case}. By diligently inspecting the limitations of these trial flow fields, we are then naturally led to propose three-dimensional branching pipe flows as a remedy.


\subsubsection{Convection rolls}
The first choice of a trial velocity profile $\bs{u}$ that comes to mind is the one associated with planar convection rolls, as this is one of the simplest incompressible flow fields capable of transporting heat by advection. Figure \ref{conve rolls 2D steady branching a} shows the streamlines of typical convection rolls. In the bulk region, far from the horizontal walls, the flow either moves up or down. To maintain the incompressibility constraint, the flow must turn around in a boundary layer near the walls. We then select a $\xi$ field accordingly, in an attempt to maximize $I$ and minimize $II$ (see figure \ref{strategies I II}).

The advantage of this configuration is that it is possible to restrict the region where $\bs{u} \cdot \nabla \xi$ is non-zero (which eventually contributes toward $II$) to the boundary layer only. However, this choice turns out to be particularly bad with respect to term $III$. Indeed, assuming the flow velocity $\bs{u} \sim 1$ in the bulk region, then we must have a ``large'' fluid velocity $\bs{u} \sim \ell/\delta$ in the boundary layer because of the incompressibility condition, where $\ell$ is the aspect ratio of a single convection roll. Consequently, $\nabla \bs{u} \sim \ell/\delta^2$, which essentially becomes ``very large" for small boundary layer thickness $\delta$. Performing a formal scaling analysis of each individual terms in the variational principle (\ref{steady var prin}) leads to
\begin{eqnarray}
Q^{s}_{\max} \gtrsim \frac{1}{\delta + \frac{1}{\mathscr{P}} \left(\frac{\ell^2}{\delta^4} + \frac{1}{\ell^4}\right)}. \nonumber
\end{eqnarray}
The right-hand side is optimized by choosing $\delta \sim \mathscr{P}^{-3/11}$ and $\ell \sim \mathscr{P}^{-2/11}$, which leads to $$Q^{s}_{\max} \gtrsim \mathscr{P}^{3/11}.$$

This scaling recovers the result of \cite{souza2020wall}, who rigorously showed that $Q^{s}(\bs{u}) \sim \mathscr{P}^{3/11}$ for a particular choice of convective rolls, as well as the results of \cite{MR159530, doering1996variational} who found the same scaling in the context of Rayleigh--B\'enard convection. The exponent $3/11$ is clearly less than $1/3$ suggesting that the convection rolls may not be the most efficient way of transporting heat  at high $\mathscr{P}$. 
 \begin{figure}
\centering
\begin{tabular}{lc}
\begin{subfigure}{0.5\textwidth}
\centering
\includegraphics[scale = 0.55]{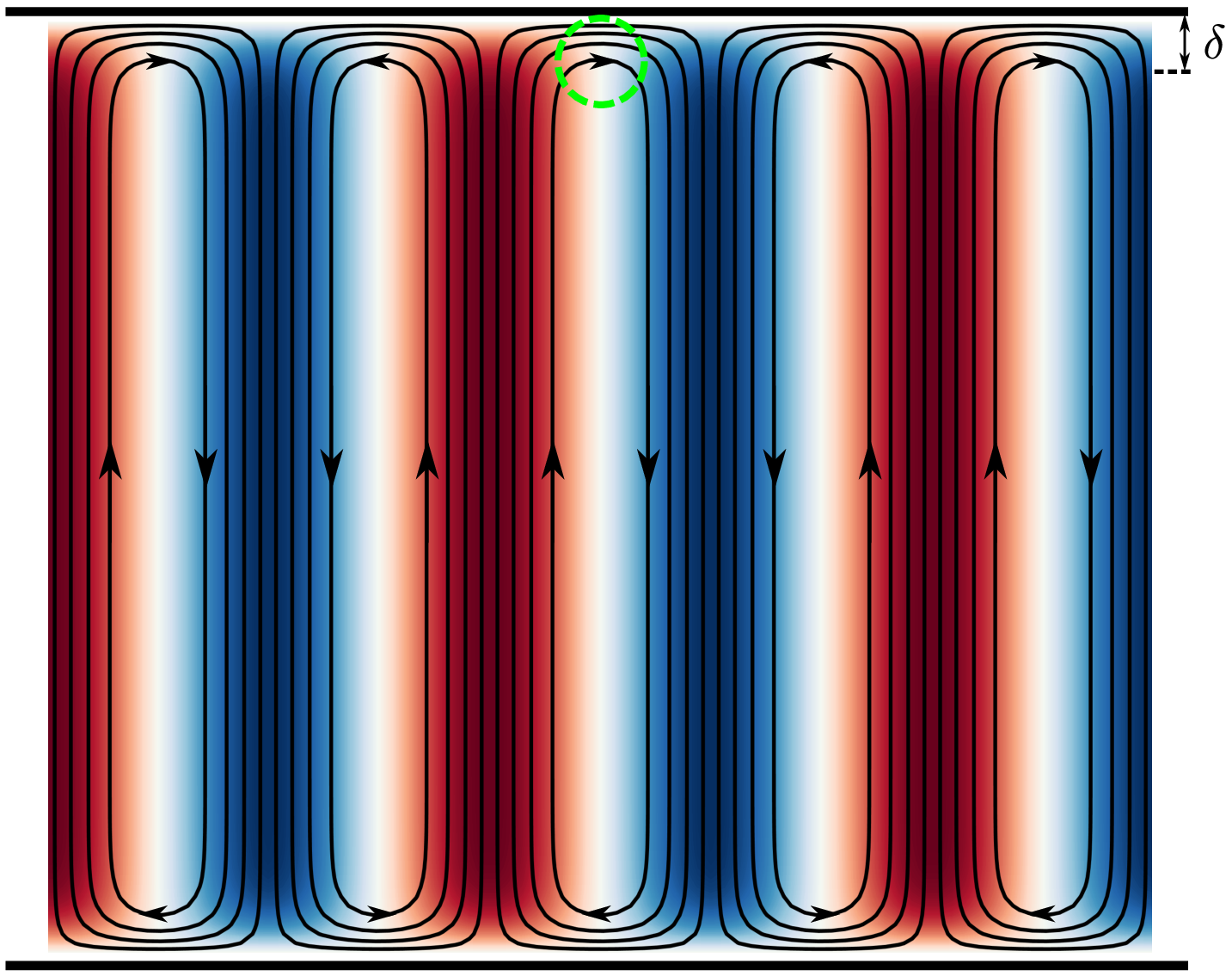}
\caption{}
\label{conve rolls 2D steady branching a}
\end{subfigure} &
\begin{subfigure}{0.5\textwidth}
\centering
 \includegraphics[scale = 0.55]{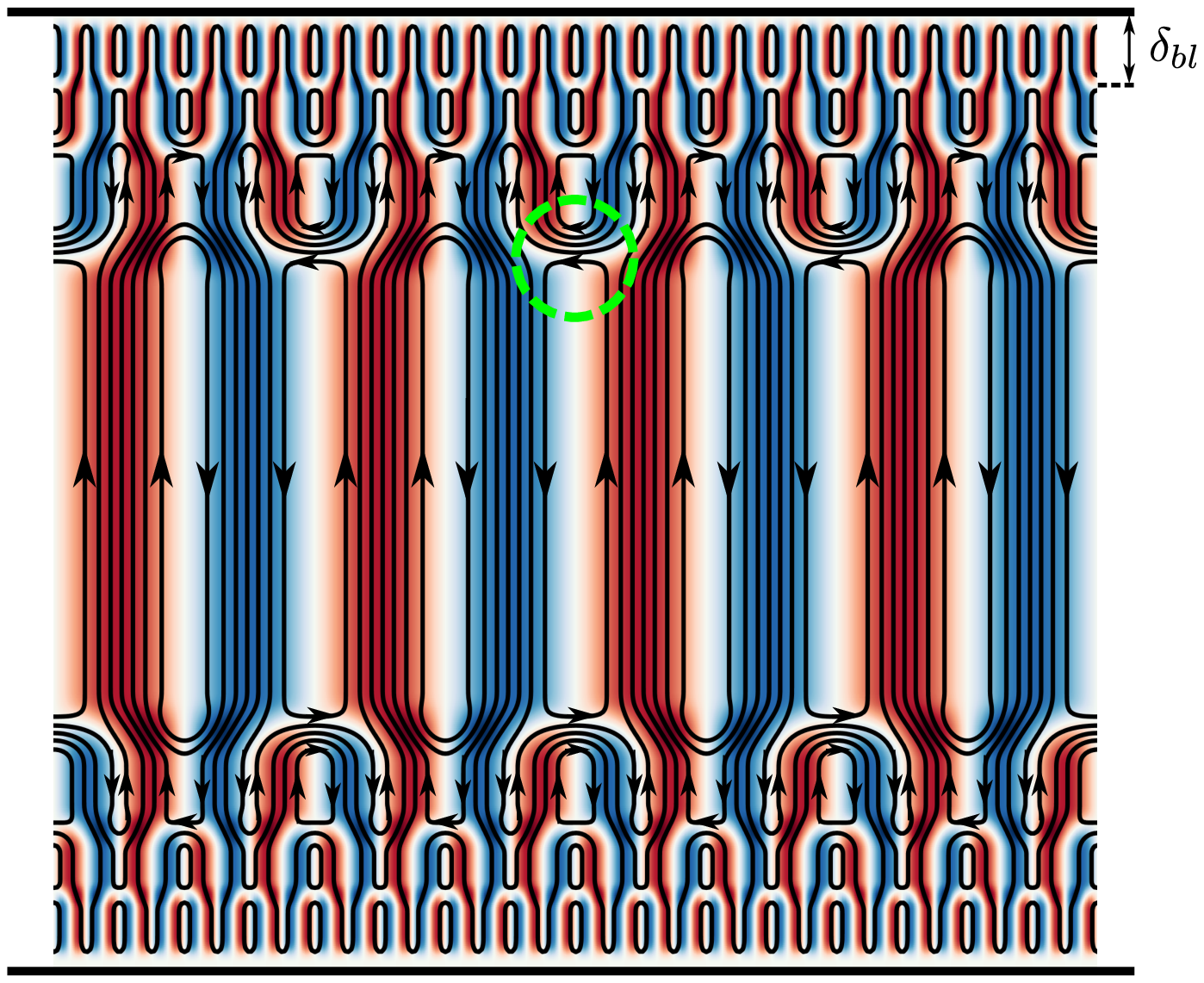}
\caption{}
\label{conve rolls 2D steady branching b}
\end{subfigure}
\end{tabular}
\caption{Panel (a) shows the streamlines of a set of typical convective rolls. Panel (b) shows the streamlines of a steady two dimensional branching flow. In both figures, the streamlines have been overlayed with a $\xi$ field according to the good scenario described in figure \ref{strategies I II a} (i.e. $\xi$ is positive whenever $u_z$ is positive, and $\xi$ is negative whenever $u_z$ is negative). The red color indicates a positive value of $\xi$, whereas the blue color indicates a negative value. The dashed circles in both the figures show a region where $\bs{u} \cdot \nabla \xi$ is nonzero.}
 \label{conve rolls 2D steady branching}
\end{figure}

\subsubsection{Two-dimensional steady branching flows}
One way to improve the heat transport is to consider a flow field with a branching structure, i.e., where the scale of flow structures becomes smaller (possibly in a self-similar manner) as one approaches the walls, an idea that goes back to Busse (\cite{busse1969howard}). The branching ends after a finite number of steps, which depends on $\mathscr{P}$. The idea behind the branching is to continue dividing the flow into ``multiple channels" as it moves towards the wall, which helps maintain the typical magnitude of the velocity field to be order unity throughout the domain. Assuming $\delta_{bl}$ is the vertical thickness of the last branching level (namely, the boundary layer), then for the branching flows we have $\nabla \bs{u} \sim \delta_{bl}^{-1}$ in the boundary layer, which is significantly smaller than $\nabla \bs{u} \sim \delta_{bl}^{-2}$ in the case of convection rolls. 

A replica of the two-dimensional steady branching flow structure constructed by \cite{tobasco2017optimal, doering2019optimal} is shown in figure \ref{conve rolls 2D steady branching b} and we have overlaid the streamlines with a $\xi$ field according to the good scenario shown in figure \ref{strategies I II a}.
Branching in two dimensions requires some part of the flow to fold back at every branching level. Although this solves the problem regarding the term $III$, it creates a different topological issue. From figure \ref{conve rolls 2D steady branching b} it becomes clear that $\nabla \xi$ is parallel to $\bs{u}$ not just in the boundary layer but also in the bulk at every branching level. A typical region is shown in dashed circle in figure \ref{conve rolls 2D steady branching b}. The result is that $\bs{u} \cdot \nabla \xi$ is nonzero (which ultimately contributes towards the term $II$) in a significant portion of the domain compared with the case of convection rolls where this term was nonzero only in the boundary layer. Furthermore, there does not appear to be a way around this topological obstruction by simply choosing a different $\xi$ field. This is because the streamlines of the flow $\bs{u}$ continually fold back throughout the branching structure, from the bulk to the boundary layer, (see figure \ref{conve rolls 2D steady branching b}) and leaving only very few streamlines to continue towards the boundary layer. So it appears that if we follow the good strategy in figure \ref{strategies I II a} (which is to choose $\xi$ positive where $u_z$ is positive and vice-versa) then it is impossible to maintain $\xi$ constant along streamlines (even in the bulk), hence we pay towards term $II$. If we avoid the good strategy in figure \ref{strategies I II a}, then it is not possible to make the term $I$ ``large.'' However, it turns out the situation is still much better than the convection rolls and a formal scaling analysis (see \cite{doering2019optimal}) shows
\begin{eqnarray}
Q^s_{\max} \gtrsim \frac{1}{\ell_{bl} + \int_{\frac{1}{2}}^{1-\delta_{bl}} (\ell^\prime)^2 \, {\rm d} z  + \frac{1}{\mathscr{P}} \left(\frac{1}{\ell^2_{bulk}} + \int_{\frac{1}{2}}^{1-\delta_{bl}} \frac{1}{\ell^2} \, {\rm d} z + \frac{1}{\ell_{bl}} \right)^2}, \nonumber
\end{eqnarray}
where $\ell_{bulk}$ and $\ell_{bl}$ denotes the horizontal aspect ratio of a typical roll in the bulk region and in the boundary layer region, respectively, while the function $\ell(z)$ denotes how the aspect ratio changes as a function of the $z$ coordinate. After optimizing the unknown parameters ($\delta_{bl}, l_{bulk}, l_{bl}$ and $l(z)$), one can only show $$Q^s_{\max} \gtrsim \frac{\mathscr{P}^{1/3}}{\log^{4/3} \mathscr{P}},$$
which is result of \cite{tobasco2017optimal, doering1996variational}

Whether there exists a two-dimensional steady flow that can overcome the topological obstruction elaborated above to show $\mathscr{P}^{1/3} \lesssim Q^s_{\max}$ is an open question. Based on the heuristic reasons given previously, we believe that there are no such flows and therefore conjecture the following.
\begin{conjecture}[Weak]
\label{weak conj}
Let $\Omega$ be $\Omega^{2D}$ given by (\ref{2D 3D domain}). Then the heat transfer defined in (\ref{Qmax steady advec-diff}) obeys $$Q^s_{\max} = o(\mathscr{P}^{1/3})$$ for large $\mathscr{P}$, where `$o$' denotes the little-o.
\end{conjecture}
The weak conjecture states that $Q^s_{\max}$ is asymptotically smaller than $\mathscr{P}^{1/3}$ but does not identify the correct asymptotic scaling of $Q^s_{\max}$ at large $\mathscr{P}$. It is reasonable to assume that the lower bound estimate of \cite{tobasco2017optimal, doering2019optimal} could be sharp. We therefore conjecture:
\begin{conjecture}[Strong]
\label{strong conj}
Let $\Omega$ be $\Omega^{2D}$ given by (\ref{2D 3D domain}). Then the heat transfer defined in (\ref{Qmax steady advec-diff}) obeys $$Q^s_{\max} \sim \frac{\mathscr{P}^{1/3}}{\log^{4/3} \mathscr{P}},$$ for large $\mathscr{P}$.
\end{conjecture}
We strongly believe that the weak conjecture is true but not so much that the strong conjecture is also true.

 \begin{figure}
\centering
\begin{subfigure}{1.0\textwidth}
\centering
 \includegraphics[scale = 0.25]{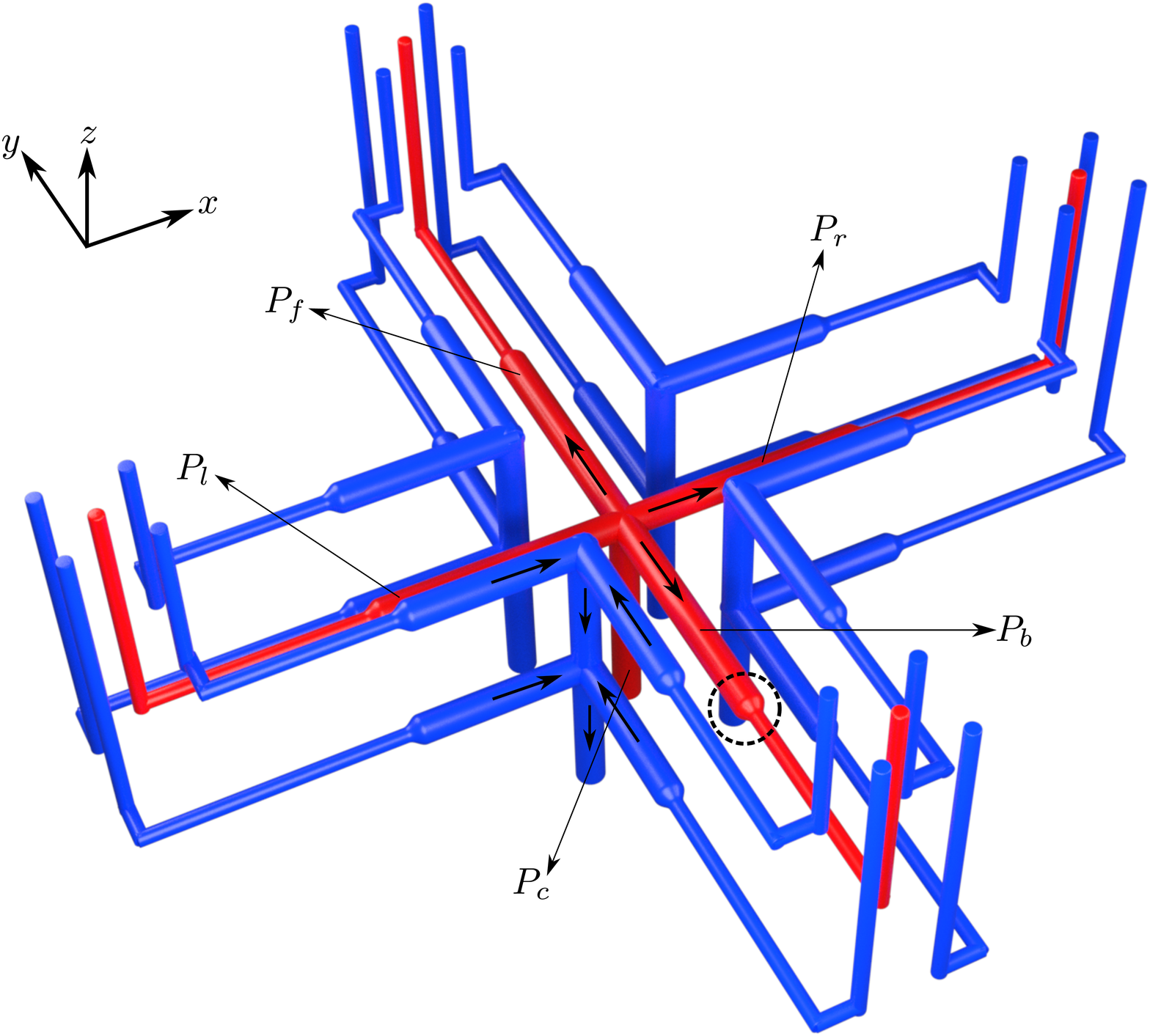}
 \caption{}
 \label{3D pipe branching a}
\end{subfigure}
\begin{tabular}{lc}
\begin{subfigure}{0.5\textwidth}
\centering
 \includegraphics[scale = 0.1]{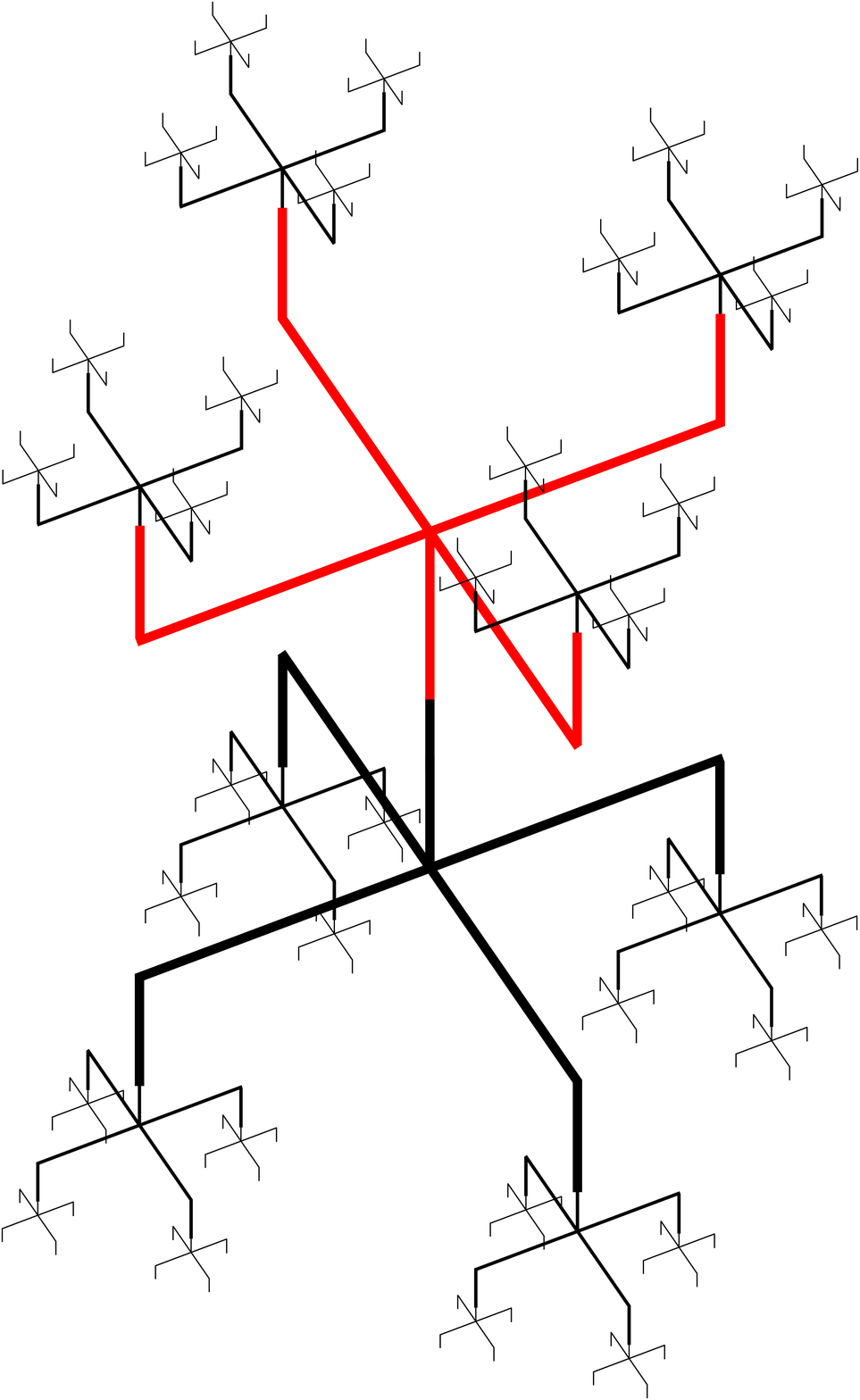}
\caption{}
\label{3D pipe branching b}
\end{subfigure} &
\begin{subfigure}{0.5\textwidth}
\centering
 \includegraphics[scale = 0.1]{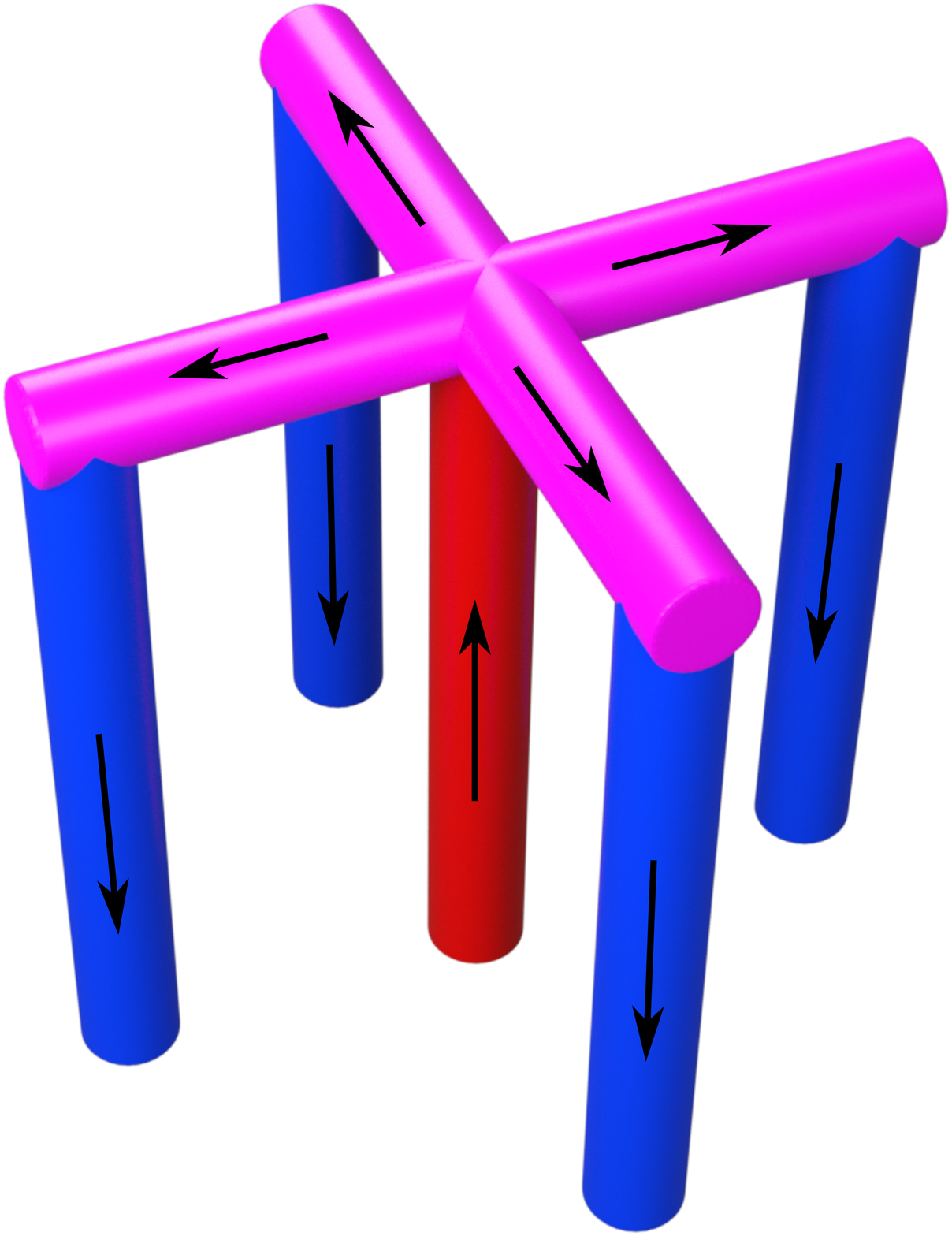}
\caption{}
\label{3D pipe branching c}
\end{subfigure}
\end{tabular}
\caption{Illustration of the branching pipe flow. Panel (a): the parent construct $\overline{\bs{u}}$. It consists of red and blue pipes which are the part of pipelines $\bs{P}_{up}$ and $\bs{P}_{down}$, respectively. In panel (a), arrows are used in some pipes to show the direction of the flow. The reducer region of a pipe is also shown using a dashed circle. Panel (b): the branching skeleton. To build the main copy $\ol{\bs{u}}_N$ away from the boundary layer, we place the appropriately dilated version of the parent construct $\overline{\bs{u}}$ along the skeleton up to $N$ levels. Panel (c): the parent construct $\overline{\bs{u}}_b$, which we use in the boundary layer. In the construct $\overline{\bs{u}}_b$, the flow from red pipes turn back to blue pipes (shown in the pink color). A 2D cartoon of the resultant pipe flow is shown in figure \ref{2D cartoon}.}
 \label{3D pipe branching}
\end{figure}



\subsubsection{Three-dimensional steady branching pipe flows}
The novelty of this work comes from the realization that the topological obstruction discussed above, can be overcome by taking advantage of the third dimension. Indeed, in three dimensions, it is possible to construct flow channels  with a branching structure that continues all the way to the wall without the need for the flow to fold back as in the two-dimensional case. Therefore, in three dimensions, it is possible to construct a flow field $\bs{u}$ and a scalar field $\xi$ that have a branching structure while maintaining $\bs{u} \cdot \nabla \xi = 0$ everywhere except in the boundary layer (which overcomes the difficulty faced in two-dimensional steady branching flows). The construction in this paper is self-similar and the resultant flow field $\bs{u}$ looks like branching pipe flow. The parent construct used in the self-similar construction is shown in figure \ref{3D pipe branching a}. It consists of two different type of pipes, one in which the flow moves up (shown in red, we choose $\xi$ positive in this region) and one in which the flow moves down (shown in blue, we choose $\xi$ negative in this region). By placing appropriately scaled copies of this parent construct along the tree structure shown in figure \ref{3D pipe branching b}, we obtain the desired flow field. The self-similar construction does not continue forever but truncates after a finite number of levels depending on the value of $\mathscr{P}$. After a fixed number of levels, the flow finally folds back in the boundary layer, according to the construct shown in figure \ref{3D pipe branching c}.
This is the region where the hot and cold pipelines finally merge and $\bs{u} \cdot \nabla \xi$ is nonzero. A two-dimensional cartoon of this three-dimensional branching pipe structure is shown in figure \ref{2D cartoon}. This cartoon also emphasizes the topological obstruction in two dimensions which informally can be expressed as ``it is not possible to build two branching channels, one hot (in which the flow moves up) and other one cold (in which the flow moves down), in two-dimensions without having them intersect.''

Using this branching pipe flow a formal scaling analysis of the heat transfer yields
\begin{eqnarray}
Q^s_{\max}  \gtrsim \frac{1}{2^{-N} + \frac{4^N}{\mathscr{P}}}, \nonumber
\end{eqnarray}
where $N$ denotes the number of branching levels. After choosing $N = \lceil \log_2 \mathscr{P}^{1/3} \rceil$, we find that $Q^s_{\max} \gtrsim \mathscr{P}^{1/3}$. Theorem \ref{Main theorem steady case} is the rigorous result of this statement, which will be proved in the subsequent sections. The construction carried out in this paper can be summarized in three steps.

\begin{itemize}
\item \underline{Step I: Creating the parent constructs (the building blocks)} \\ The velocity fields from this step form the basis for the self-similar construction in the second step. In this step, we construct (i) $\ol{\bs{u}}$ (figure \ref{3D pipe branching a}), which is used to build branching flow away from the boundary layer. (ii) $\ol{\bs{u}}_b$ (figure \ref{3D pipe branching c}), which is used in the boundary layer to truncate the branching structure. 

\item \underline{Step II: Construction of the main copy (a single tree)} \\ In this step, we assemble the appropriately dilated copies of the parent constructs from the previous step to build the flow field $\ol{\bs{u}}_N$ (a 2D cartoon is shown in figure \ref{2D cartoon}). Here, $N$ denotes the number of branching levels which depends on $\mathscr{P}$. We also refer to this main copy as a single tree.

\item \underline{Step III: Construction of the final flow field (a forest)} \\ The flow field constructed in the last step is enough to capture the correct dependence of $Q^s_{\max}$ on $\mathscr{P}$. However, to capture the correct dependence on the domain aspect ratio $l_x$ and $l_y$, we build the final flow field $\bs{u}$ by placing several copies of the tree side-by-side to fill the whole domain, which then looks like a forest.
\end{itemize}

\begin{figure}
\centering
 \includegraphics[scale = 0.06]{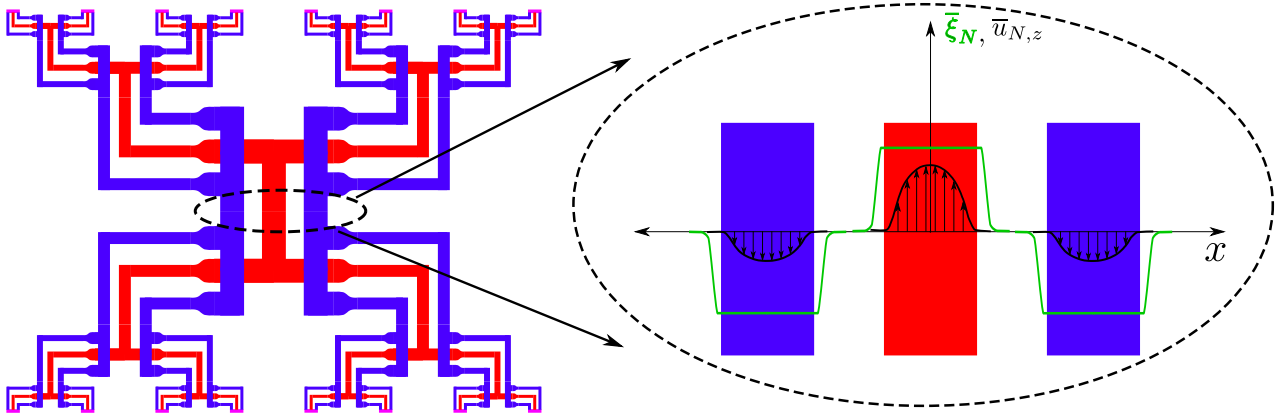}
\caption{shows a 2D cartoon of the main copy $\ol{\bs{u}}_N$. The pipeline $\bs{P}_{up}$ is shown in red color and the pipeline $\bs{P}_{down}$ is shown in blue color. In the blow-up figure of a section of the pipeline, the graph of $\ol{\xi}_N$ is also shown. Notice is that $\ol{\xi}_N$ is constant in the support of $\ol{\bs{u}}_N$.}
 \label{2D cartoon}
\end{figure}

\subsection{Organization of the paper}
The rest of paper is organized as follows. In section \ref{Notation}, we introduce a few notations and preliminaries that will be frequently used throughout the paper. In section \ref{Strategy and proof}, we perform Step III of the construction and prove the main theorem. We provide a detailed sketch of the parent constructs in section \ref{construction of parent copy}. We then carry out Step I and Step II. We provide a proof of Proposition \ref{Poisson's equation: inverse Laplace torus to D} (essential for the analysis of the nonlocal term defined in (\ref{terms var prin})) in section \ref{Section: An estimate solution of Poisson}. We close by discussing implications of our results in section \ref{Discussion}. A few of the more cumbersome but trivial calculations required to finish the proofs are carried out in appendices. 

\begin{center}
\subsection*{Acknowledgement}
\end{center}
A.K. thanks P. Garaud for a careful read of the paper and providing comments. A.K. also thanks I. Tobasco for providing comments, an invitation to visit University of Illinois Chicago and for several useful discussions.

\section{Notation and preliminaries} \label{Notation}
The three domains we will be frequently using in this paper are: $\mathbb{R}^3$, $\Omega$ and $D$, where
\addtocounter{equation}{1}
\begin{align}
\Omega \coloneqq \mathbb{T}_{l_x} \times \mathbb{T}_{l_y} \times (-1/2, 1/2), \qquad \qquad D \coloneqq \mathbb{R}^2 \times (-1/2, 1/2).
\tag{\theequation a-b}
\label{Omega D domain}
\end{align}
Here, for some $l > 0$, $\mathbb{T}_{l} \coloneqq (\mathbb{R}/l\mathbb{Z})$ and $\mathbb{T}_{l}$ is identified with $[-l/2, l/2)$ in the usual way. In the rest of this section, $V$ will denote either of these three domains: $\mathbb{R}^3$, $\Omega$ and $D$, whereas $\wt{V}$ will denote either $\mathbb{R}^3$ or $D$. Let $\bs{x}, \bs{x}^\prime \in \wt{V}$, for which we denote
\addtocounter{equation}{1}
\begin{align}
\bs{x}_{\parallel} \coloneqq (x, y, 0), \qquad \text{ and } \qquad |\bs{x} - \bs{x}^\prime|_{\parallel} \coloneqq  |\bs{x}_{\parallel} - \bs{x}^\prime_{\parallel}|,
\tag{\theequation a-b}
\label{parallel coordinate collapse notation}
\end{align}
where $|\, \cdot \,|$ denotes the Euclidean distance. Let $S \subseteq V$, we will use $\bs{1}_S$ to denote the indicator function corresponding to the set $S$.

We define the support of a scalar or a vector-valued function $f$ on $V$ as
\begin{eqnarray}
\supp f \coloneqq \ol{\{\bs{x} \in V \, | \, f(\bs{x}) \neq 0\}},
\end{eqnarray}
and the support only in the $z$ variable as
\begin{eqnarray}
\supp_z f \coloneqq \ol{\{z \in \mathbb{R} \, | \, (x, y, z) \in V, f(x, y, z) \neq 0\}}.
\end{eqnarray}

For a given $\bs{p} \in \mathbb{R}^3$, we define a translation map $T^{\bs{p}}: \mathbb{R}^3 \to \mathbb{R}^3$ as $T^{\bs{p}}(\bs{x}) = \bs{x} + \bs{p}$. The inverse map is therefore denoted as $T^{-\bs{p}}$.
Then, if $f$ is a scalar function or a vector-valued function on $\mathbb{R}^3$, we define the corresponding translated function $T^{\bs{p}} f$ as
\begin{eqnarray}
(T^{\bs{p}} f)(\bs{x}) = f(T^{-\bs{p}}(\bs{x})), \quad \text{where} \quad \bs{x} \in \mathbb{R}^3.
\end{eqnarray}

Similarly, for a given $\theta \in [0, 2 \pi]$, we define a rotation map $\rho_\theta : \wt{V} \to \wt{V}$, which performs a counterclockwise rotation in the $xy$-plane by an angle $\theta$. We denote the inverse map by $\rho_{-\theta}$. Then, if $\zeta$ is a scalar function on $\wt{V}$, we define the corresponding rotated scalar function $\rho_\theta \zeta$ on $\wt{V}$ as
\begin{eqnarray}
(\rho_\theta \zeta)(\bs{x}) = \zeta(\rho_{-\theta}(\bs{x})), \quad \text{where} \quad \bs{x} \in \wt{V}.
\end{eqnarray}
Furthermore, if $\bs{v}$ is a vector-valued function defined on $\wt{V}$, we define the corresponding rotated vector-valued function $\rho_\theta \bs{v}$ on $\wt{V}$ as
\begin{eqnarray}
(\rho_\theta \bs{v})(\bs{x}) = \rho_\theta\left(\bs{v}(\rho_{-\theta}(\bs{x}))\right), \quad \text{where} \quad \bs{x} \in \wt{V}.
\end{eqnarray}

Let's denote the $\sigma-$algebra of Borel sets by $\mathcal{B}(\mathbb{R}^3)$. Given a Radon measure $\mu: \mathcal{B}(\mathbb{R}^3) \to \mathbb{R}$ and a vector field $\bs{u} \in L^1_{loc}(\mathbb{R}^3; \mathbb{R}^3, \mu)$, the set function  $\nu: \mathcal{B}(\mathbb{R}^3) \to \mathbb{R}^3$
\begin{eqnarray}
\nu \coloneqq (\nu_x, \nu_y, \nu_z) \coloneqq (u_x \mu, u_y \mu, u_z \mu)
\label{Notation: vect-val-meas}
\end{eqnarray}
is called a vector-valued Radon measure. Alternate shorthand notation is $\nu = \bs{u} \mu$. The Riesz’s theorem ensures that the space of vector-valued Radon measure $\mathcal{M}$ is dual to the space of compactly supported continuous vector fields $C_c(\mathbb{R}^3; \mathbb{R}^3)$ \cite{MR1645086, Maggi12gmt}.

Now given a function $f \in C_c(\mathbb{R}^3)$ and $\nu \in \mathcal{M}$ as defined in (\ref{Notation: vect-val-meas}), the integration of $f$ with respect to the measure $\nu$ is a vector in $\mathbb{R}^3$ and is given by
\begin{eqnarray}
\int_{\mathbb{R}^3} f \, {\rm d} \nu = \left(\int_{\mathbb{R}^3} f u_x \, {\rm d} \mu, \int_{\mathbb{R}^3} f u_y \, {\rm d} \mu, \int_{\mathbb{R}^3} f u_z \, {\rm d} \mu\right),
\end{eqnarray}
and the convolution is given by
\begin{eqnarray}
(f * \nu)(\bs{x}) = \int_{\mathbb{R}^3} f(\bs{x} - \bs{x}^\prime) \, {\rm d} \nu(\bs{x}^\prime).
\end{eqnarray}


\section{Step III of the construction: Proof of Theorem \ref{Main theorem steady case}} \label{Strategy and proof}

In this section we begin by performing Step III of the construction. We assume the existence of main copies $\ol{\bs{u}}_N$ and $\ol{\xi}_N$ with properties stated in the proposition below. Then we place several of these copies together in $\Omega$, to build the flow field $\bs{u}$ and scalar field $\xi$, which we then use in the variational principle (\ref{steady var prin}) to prove Theorem \ref{Main theorem steady case}.
\begin{proposition}
\label{Structure of the near optimal flow: flow fields between stripes}
For every positive integer $N$, there exist $\overline{\bs{u}}_N \in C^\infty_{c}(D; \mathbb{R}^3)$ and $\overline{\xi}_{N} \in C^\infty_c(D)$ such that 
\\
(i) $\nabla \cdot \overline{\bs{u}}_N \equiv 0$, \\
(ii) $\supp \overline{\bs{u}}_N \cup \supp \overline{\xi}_{N} \Subset (-1/2, 1/2) \times (-1/2, 1/2) \times (-1/2, 1/2)$, \\
(iii) $\supp_z (\overline{\bs{u}}_N \cdot \nabla \overline{\xi}_{N}) \Subset (1/2 - c_1 2^{-N}, 1/2 - c_2 2^{-N}) \cup (-1/2 + c_2 2^{-N}, -1/2 + c_1 2^{-N})$, \\
(iv) $\norm{\overline{\bs{u}}_N \cdot \nabla \overline{\xi}_{N}}_{L^\infty(D)} \lesssim 2^N$, \\
(v) $\int_{D} |\nabla \overline{\bs{u}}_N|^2 \, {\rm d} \bs{x} + \int_{D} |\nabla \overline{\xi}_{N}|^2 \, {\rm d} \bs{x} \lesssim 2^N$, \\
(vi) $\int_{D} \overline{u}_{N, z} \, \overline{\xi}_{N} \, {\rm d} \bs{x} \geq c_3 > 0$, \\
Here, $0 < c_2 < c_1 < 1$ and $c_3$ are constants independent of $N$ and $\overline{u}_{N, z}$ is the $z$-component of $\overline{\bs{u}}_N$.
\end{proposition}
\begin{proof}[Proof of Theorem \ref{Main theorem steady case}]
We construct $\bs{u}$ (and $\xi$) by appropriately placing the several horizontally scaled copies of $\overline{\bs{u}}_N$ (and $\overline{\xi}_{N}$) from Proposition \ref{Structure of the near optimal flow: flow fields between stripes} side-by-side (see below for details). Specifically, let $n_x$ and $n_y$ be two positive integers, then we place $n_x n_y$ copies of $\overline{\bs{u}}_N$ (and $\overline{\xi}_{N}$) in a  two-dimensional rectangular horizontal array. Then from the conditions on $\overline{\bs{u}}_N$ and $\overline{\xi}_{N}$ given in Proposition \ref{Structure of the near optimal flow: flow fields between stripes}, we obtain estimates on various terms in the expression (\ref{steady var prin}) and show that the desired lower bound on $Q^s_{\max}$, stated in Theorem \ref{Main theorem steady case}, can be obtained. 

More specifically, given $n_x, n_y \in \mathbb{N}$, we define two lengths $d_x$ and $d_y$ as follows:
$$d_x = \frac{l_x}{n_x} \quad \text{and} \quad d_y = \frac{l_y}{n_y}.$$
Next, we define $\bs{u} : D \to \mathbb{R}^3$ and $\xi : D \to \mathbb{R}$ as 
\begin{subequations}
\begin{eqnarray}
\bs{u}\left(x d_x - \frac{l_x}{2} + \frac{2 i - 1}{2} d_x, y d_y - \frac{l_y}{2} + \frac{2 j - 1}{2} d_y, z\right) \coloneqq \overline{\bs{u}}_N(x, y, z) \nonumber \\
\xi\left(x d_x - \frac{l_x}{2} + \frac{2 i - 1}{2} d_x, y d_y - \frac{l_y}{2} + \frac{2 j - 1}{2} d_y, z\right) \coloneqq \overline{\xi}_N(x, y, z) \nonumber
\end{eqnarray}
\end{subequations}
for all $i, j \in \mathbb{Z}$ and $(x, y, z) \in (-1/2, 1/2) \times (-1/2, 1/2) \times (-1/2, 1/2)$, otherwise, $\bs{u} \coloneqq \bs{0}$ and $\xi \coloneqq 0$. It is clear that $\bs{u}$ and $\xi$ are $l_x-l_y-$periodic functions. It is the identification of these $l_x-l_y-$periodic functions with functions on $\Omega = \mathbb{T}_{l_x} \times \mathbb{T}_{l_y} \times (-1/2, 1/2)$, which we continue to denote as $\bs{u}$ and $\xi$, that we use throughout.

By construction, $\bs{u} \in C_c^\infty(\Omega; \mathbb{R}^3)$ and $\xi \in C^\infty_c(\Omega)$ and therefore belongs to the admissible sets $\mathcal{A}^s$ and $\mathcal{X}^s$ as defined in (\ref{admissible vel steady}) and (\ref{admissible xi steady}), respectively. Now one can estimates important terms in the variational formula (\ref{steady var prin}). Let's start with the following:
\begin{eqnarray}
\dashint_{\Omega} |\nabla \bs{u}|^2 \, {\rm d} \bs{x} = \frac{1}{l_x l_y} \int_{\Omega} |\nabla \bs{u}|^2 \, {\rm d} \bs{x} = \frac{n_x n_y}{l_x l_y} \int_{-\frac{l_x}{2}}^{-\frac{l_x}{2} + d_x} \int_{-\frac{l_y}{2}}^{-\frac{l_y}{2} + d_y} \int_{-\frac{1}{2}}^{\frac{1}{2}} |\nabla \bs{u}|^2 \, {\rm d} \bs{x} \nonumber \\
\lesssim \frac{n_x n_y d_x d_y}{l_x l_y \min\{d_x^2, d_y^2, 1\}} \int_{D} |\nabla \overline{\bs{u}}_N|^2 \, {\rm d} \bs{x} \lesssim \frac{n_x n_y d_x d_y}{l_x l_y \min\{d_x^2, d_y^2, 1\}} 2^N.
\label{Strategy and proof of the main theorem: estimate L 2 u}
\end{eqnarray}
Similarly, we have
\begin{eqnarray}
\dashint_{\Omega} |\nabla \xi|^2 \, {\rm d} \bs{x} 
\lesssim \frac{n_x n_y d_x d_y}{l_x l_y \min\{d_x^2, d_y^2, 1\}} \int_{D} |\nabla \overline{\xi}_N|^2 \, {\rm d} \bs{x} \lesssim \frac{n_x n_y d_x d_y}{l_x l_y \min\{d_x^2, d_y^2, 1\}} 2^N.
\label{Strategy and proof of the main theorem: estimate L 2 xi}
\end{eqnarray}
In a same way, one can also show
\begin{eqnarray}
\dashint_{\Omega} u_z \xi \, {\rm d} \bs{x} = \frac{n_x n_y d_x d_y}{l_x l_y} \int_{D} \overline{u}_{N, z} \, \overline{\xi}_{N} \, {\rm d} \bs{x}   \geq \frac{n_x n_y d_x d_y c_3}{l_x l_y},
\label{Strategy and proof of the main theorem: estimate w xi}
\end{eqnarray}
Finally, we have
\begin{eqnarray}
\norm{\bs{u} \cdot \nabla \xi}_{L^\infty(\Omega)} \lesssim \frac{2^N}{\min\{d_x, d_y, 1\}},
\label{Strategy and proof of the main theorem: estimate nonlocal}
\end{eqnarray}
with
\begin{eqnarray}
\supp_z (\bs{u} \cdot \nabla \xi) \Subset (1/2 - c_1 2^{-N}, 1/2 - c_2 2^{-N}) \cup (-1/2 + c_2 2^{-N}, -1/2 + c_1 2^{-N}).
\end{eqnarray}
Provided $N \geq 3$, we obtain
\begin{eqnarray}
\dashint_{\Omega} |\nabla \Delta^{-1} \diverge (\bs{u} \xi)|^2 \, {\rm d} \bs{x} \lesssim \frac{1}{2^N \min\{d_x^2, d_y^2, 1\}},
\label{Strategy and proof of the main theorem: estimate nonlocal req. poisson}
\end{eqnarray}
using the following proposition.
\begin{proposition}
\label{Poisson's equation: inverse Laplace torus to D}
Let $f \in L^\infty(\Omega)$ such that $\supp_z f \subseteq (1/2 - c_1 \varepsilon, 1/2 - c_2 \varepsilon) \cup (-1/2 + c_2 \varepsilon, -1/2 + c_1 \varepsilon)$, where $0 < c_2 < c_1 < 1$ and $\varepsilon < 1/4$ are three constants, then we have
\begin{eqnarray}
\dashint_{\Omega} |\nabla \Delta^{-1} f|^2 \, {\rm d} \bs{x}  \lesssim \varepsilon^3 \norm{f}_{L^\infty(\Omega)}^2.
\label{The Poisson's equation: f supp inverse Laplace torus to D}
\end{eqnarray}
\end{proposition}
Proof of the Proposition \ref{Poisson's equation: inverse Laplace torus to D} is provided in section \ref{Section: An estimate solution of Poisson}.

At this point, we prescribe $n_x$ and $n_y$. As stated in the introduction, we have chosen $l_x \leq l_y$ without the loss of generality. We divide the proof of the theorem into two parts: (i) when $l_x \geq 1$, (ii) when $l_x < 1$.

In the first case ($l_x \geq 1$), we choose
\begin{eqnarray}
n_x = \left\lceil l_x \right\rceil \quad \text{and} \quad n_y = \left\lceil l_y \right\rceil,
\end{eqnarray} 
where $\left\lceil \; \cdot \; \right\rceil$ is the ceiling function. Then from the definitions of $d_x$ and $d_y$, we have
\begin{eqnarray}
\frac{1}{2} \leq d_x, d_y \leq 1.
\end{eqnarray}
Noting this and using the estimates (\ref{Strategy and proof of the main theorem: estimate L 2 u}), (\ref{Strategy and proof of the main theorem: estimate L 2 xi}), (\ref{Strategy and proof of the main theorem: estimate w xi}) and (\ref{Strategy and proof of the main theorem: estimate nonlocal}) in (\ref{steady var prin}) gives
\begin{eqnarray}
\frac{1}{c_5 \frac{1}{2^N} + c_6 \frac{4^N}{\mathscr{P}}} \lesssim Q_{\max}(\mathscr{P}),
\end{eqnarray}
where $c_5$ and $c_6$ are two constants independent of any parameter. Choosing the value of $N$ as
\begin{eqnarray}
N = \left\lceil \frac{1}{3} \log_2 \frac{c_5 \mathscr{P}}{2 c_6} \right\rceil,
\end{eqnarray}
we can show
\begin{eqnarray}
\mathscr{P}^{1/3} \lesssim Q_{\max}(\mathscr{P})
\end{eqnarray}
provided
\begin{eqnarray}
\mathscr{P} \gtrsim \frac{2 c_6}{c_5}.
\end{eqnarray}

In the second case, when $l_x < 1$, we choose
\begin{eqnarray}
n_x =  1  \quad \text{and} \quad n_y = \left\lceil \frac{l_y}{l_x} \right\rceil,
\end{eqnarray}
then we have
\begin{eqnarray}
d_x = l_x \quad \text{and} \quad   \frac{l_x}{2} \leq d_y \leq l_x.
\end{eqnarray}
The estimates (\ref{Strategy and proof of the main theorem: estimate L 2 u}), (\ref{Strategy and proof of the main theorem: estimate L 2 xi}), (\ref{Strategy and proof of the main theorem: estimate w xi}) and (\ref{Strategy and proof of the main theorem: estimate nonlocal}) then imply
\begin{eqnarray}
\frac{1}{c_7 \frac{1}{2^N l_x^2} + c_8 \frac{4^N}{\mathscr{P} l_x^4}} \lesssim  Q_{\max}(\mathscr{P}),
\end{eqnarray}
for some positive constants $c_7$ and $c_8$ independent of any parameter. Now choosing the following value of $N$
\begin{eqnarray}
N = \left\lceil \frac{1}{3} \log_2 \frac{c_7 \mathscr{P} l_x^2}{2 c_8} \right\rceil,
\end{eqnarray}
we obtain
\begin{eqnarray}
\mathscr{P}^{1/3} l_x^{8/3} \lesssim Q_{\max}(\mathscr{P}),
\end{eqnarray}
provided
\begin{eqnarray}
\mathscr{P} \gtrsim \frac{2 c_8}{c_7 l_x^2},
\end{eqnarray}
which then completes the proof of the Theorem \ref{Main theorem steady case}.
\end{proof}

\section{Construction of three-dimensional branching pipe flow: Step I and Step II}
\label{construction of parent copy}
The goal of this section is to perform Step I, which is to build the parent constructs $\overline{\bs{u}}$, $\overline{\bs{u}}_b$, $\overline{\xi}$ and $\overline{\xi}_b$, followed by Step II, which is to create the main copies $\overline{\bs{u}}_N$ and $\overline{\xi}_N$. We start by giving a sketch of the parent copies and how to assemble their dilated versions to create the main copies, which is then followed by the actual construction in Step I and Step II.

\subsection{A detailed sketch of the construction}
\label{a detailed sketch}
As the support of the velocity field $\ol{\bs{u}}_N$ looks like a pipe network (see figure \ref{2D cartoon}) and the flow field itself is similar to flow in pipes, we use words such as pipe, pipe network or pipeline for ease of exposition below.

The main copy $\ol{\bs{u}}_N$ consists of two ``pipelines'': one in which the flow goes upward (the positive $z$-direction)  in a branching fashion, $\bs{P}_{up}$ (shown in red in figure \ref{2D cartoon}) and one in which the flow goes downward (the negative $z$-direction), again in a branching fashion, $\bs{P}_{down}$ (shown in blue). The part of the pipelines $\bs{P}_{up}$ and $\bs{P}_{down}$ that resides in the parent construct $\ol{\bs{u}}$ is also shown in red and blue, in figure \ref{3D pipe branching}.

The volume flow rate through both of these pipelines is the same. In what follows, we describe the pipeline design only for $z \geq 0$ and simply use mirror symmetry to construct the pipeline for $z \leq 0$. In the parent construct,  $\bs{P}_{up}$ starts from a center pipe, denoted by $P_c$ in figure \ref{3D pipe branching a}. The center pipe goes up vertically, until a first junction at $z = 1/ 8$, where it splits
 into four pipes going right (positive $x$-direction) $P_r$, left (negative $x$-direction) $P_l$, front (positive $y$-direction) $P_f$ and back (negative $y$-direction) $P_b$. In plumbing terms, the junction of these pipes would be known as a 5-way cross. The horizontal extent of these pipes is $1/4$. Near the junction, each of the four horizontal pipes has a radius equal to that of $P_c$. Therefore, because of incompressibility condition, the speed of the flow goes down by a factor of four as the flow enters from $P_c$ to $P_r$, $P_l$, $P_f$ and $P_b$. However, away from the junction (midway), a constriction is added to reduce the radii of these four pipes by a factor of half after which the speed of the flow regains its original value (again because of incompressibility). In plumbing terms, the region where the radius of the pipe decreases is known as a reducer. Finally, these horizontal pipes bend upward up to a level $z = 1/4$. With this construction, the pipeline $\bs{P}_{up}$ near $z = 1/4$ consists of four pipes whose radius is half that  of the pipe $P_c$ near $z = 0$ but all of them with same magnitude of velocity. We can then continue the pipeline from $z = 1/4$ to $z = 1/4 + 1/8$ by adding four half-sized copies of the original one.  In a similar way, the pipeline can be further continued up to any number of levels $N$.

The pipeline in which the flow goes down consists of four pipes surrounding $P_c$, each with radii equal to that of $P_c$. The speed of the flow in one of these four pipes is $1/4$  the speed of the flow in $P_c$, ensuring that the total volume flow going upward and downward are the same. The  flow in these four pipes come from the horizontally placed pipes that are similarly surrounding the pipes $P_r$, $P_l$, $P_f$ and $P_b$ as shown in figure \ref{3D pipe branching a}. The radii of these horizontal pipes, similar to the case of the previous pipeline, changes by a factor of two to ensure that the flow velocity in the vertical pipes that they connect remains the same. Finally, before bending in the upward direction, the horizontal pipes, in this pipeline, close their distance to the horizontal pipes from pipeline $\bs{P}_{up}$ to make sure that we can glue a self-similar parent copy of half-the-size to continue the pipeline.

The self-similar continuation of both pipelines truncates after a fixed number of levels $N$ (depending on $\mathscr{P}$). In the last level (closest to the wall), the two pipelines merge, i.e., the flow from the pipeline $\bs{P}_{up}$ goes to the pipeline $\bs{P}_{down}$. This done by gluing an appropriately scaled parent construct $\overline{\bs{u}}_b$ as shown in figure \ref{3D pipe branching c}.

Once we have the main copy $\ol{\bs{u}}_N$ ready, we can select $\ol{\xi}_N$. We choose $\ol{\xi}_N$ (everywhere except in the boundary layers where the pipelines truncate) to be such that its value is a positive constant $\xi_0$ in the region where the pipeline $\bs{P}_{up}$ lies and is $-\xi_0$ in the region where the pipeline $\bs{P}_{down}$ lies and decays to zero rapidly away from these pipelines (see figure \ref{2D cartoon}). There are two advantages with this choice:
\begin{itemize}
\item[(i)] The quantity $\ol{\bs{u}}_N \cdot \nabla \ol{\xi}_N$ is identically zero except in the last level of construction where the branching structure truncates. Therefore, it is possible to restrict the support of $\ol{\bs{u}}_N \cdot \nabla \ol{\xi}_N$ to a thin horizontal layer close to the wall, which helps in obtaining a good estimate on the nonlocal term in (\ref{terms var prin}).
\item[(ii)] The transport term simplifies as follows. 
\begin{eqnarray}
\int \ol{u}_{N, z} \; \ol{\xi}_N \, {\rm d}{\bs{x}} \approx \xi_0 \int_{\bs{P}_{up}} \ol{u}_{N, z} \, {\rm d}\bs{x} - \xi_0 \int_{\bs{P}_{down}} \ol{u}_{N, z} \, {\rm d}\bs{x} \approx 2 \xi_0 V_0, \nonumber
\end{eqnarray}
where $V_0 > 0$ is the total flow (constant volume flux through any horizontal section) going upward in pipeline $\bs{P}_{up}$ or downward in pipeline $\bs{P}_{down}$. There will be minor corrections in the region where the pipelines truncate, which is why we use the approximate symbol.
\end{itemize}

In summary, we built two pipelines with a self-similar ``tree-like'' branching structure. The first one, $\bs{P}_{up}$, is ``hot'' (as $\ol{\xi}_N$ is positive in that region) in which the flow goes up and the second one, $\bs{P}_{down}$, is ``cold'' (as $\ol{\xi}_N$ is negative) and surrounds (without touching) the hot pipeline $\bs{P}_{up}$. This type of ``disentanglement'' of the hot pipeline from the cold one is possible in three dimensions but not in two dimensions and is the main reason behind the proof of Theorem \ref{Main theorem steady case}.

\subsection{Step I: Construction of the parent copies}
The purpose of this subsection is to build the parent constructs: $\ol{\bs{u}}$, $\ol{\bs{u}}_b$ and the trial $\xi$-field: $\ol{\xi}$, $\ol{\xi}_b$. Let us define a few parameters that will be frequently used in this section:
\begin{eqnarray}
\gamma = \frac{1}{500}, \qquad \lambda = \frac{1}{100}, \qquad \delta = \frac{1}{20}.
\label{Construction: choices that we made in our lives}
\end{eqnarray}
These parameters can roughly be understood as follows. The parameter $\gamma$ can be thought of as the radius of pipes in which the flow field is supported, whereas $\lambda$ is the radius of pipes in which the $\xi-$field is supported and  $\delta$ denotes the distance between pipeline $\bs{P}_{up}$ and $\bs{P}_{down}$ in the parent copy $\overline{\bs{u}}$.

\subsubsection{The flow field $\ol{\bs{u}}$ and $\ol{\bs{u}}_b$} \label{The parent flow field}
To construct the flow field, the basic idea is to define an appropriate vector-valued Radon measure supported on a set. This set is a collection of line segments and rays, which, in a loose sense, form the skeleton of the pipelines whose sketch is described in subsection \ref{a detailed sketch}. Most of the desired flow field will then be created by regularizing the Radon measure using a convolution with a mollifier, except in the reducer region of the pipelines. The flow field in the reducer region will be designed separately with the help of an axisymmetric streamfuction. 

We start by defining a few important points in $\mathbb{R}^3$, which will be helpful in creating the ``skeleton'' of the pipelines. We define
\begin{eqnarray}
\bs{p}_1 \coloneqq (0, 0, 0), \quad \bs{p}_2 \coloneqq (0, 0, 1/8), \quad \bs{p}_3 \coloneqq (1/4, 0, 1/8), \quad \bs{p}_4 \coloneqq (1/4, 0, 1/4), \nonumber
\end{eqnarray}
and
\begin{eqnarray}
\bs{q}^{i, j}_1 \coloneqq \left(\delta, \, j\delta, \, 0\right), \quad \bs{q}^{i, j}_2 \coloneqq \left(\delta, \, j\delta, \, 1/8-i\delta\right),  \quad \bs{q}^{i, j}_3 \coloneqq \left(\frac{1}{4} + \frac{i\delta}{2}, \, j\delta, \, 1/8-i\delta\right), \nonumber \\  \quad \bs{q}^{i, j}_4 \coloneqq \left(\frac{1}{4} + \frac{i\delta}{2}, \, \frac{j\delta}{2}, \, 1/8-i\delta\right), \quad \bs{q}^{i, j}_5 \coloneqq  \left(\frac{1}{4} + \frac{i\delta}{2}, \, \frac{j\delta}{2}, \, \frac{1}{4}\right), \nonumber
\end{eqnarray}
where $ i, j \in \mathbb{Z}.$
Next, we define a family of points, obtained by horizontal rotation of the points defined above. Let $\theta \in [0, 2\pi]$, we define
\begin{eqnarray}
\bs{p}_{k, \theta} \coloneqq \rho_{\theta}(\bs{p}_k) \qquad \text{for} \quad k \in \{1, 2, 3, 4\}; \qquad \bs{q}^{i, j}_{k, \theta} \coloneqq \rho_{\theta}(\bs{q}^{i, j}_k) \qquad \text{for} \quad k \in \{1, 2, 3, 4, 5\}.
\label{Construction: vip points}
\end{eqnarray}
We recall that the transformation $\rho_\theta$, defined in section \ref{Notation}, is a counterclockwise horizontal rotation by an angle $\theta$.

We define two sets:
\begin{eqnarray}
J \coloneqq \{-1, 1\} \qquad \text{and} \qquad \Theta \coloneqq \left\{0, \frac{\pi}{2}, \pi, \frac{3 \pi}{2} \right\}. \nonumber
\end{eqnarray}

Before defining the appropriate vector-valued Radon measures, we set a few notations. Given two points  $\bs{a}_1, \bs{a}_2 \in \mathbb{R}^3$, where $\bs{a}_1 \neq \bs{a}_2$, we denote the line segment whose end points are $\bs{a}_1$ and $\bs{a}_2$ as
\begin{eqnarray}
\overline{\bs{a}_1 \bs{a}_2} \coloneqq \{(1-t) \bs{a}_1 + t \bs{a}_2 | \; t \in [0, 1]\},
\label{Notation: line seg}
\end{eqnarray}
whereas to denote the ray that starts at $\bs{a}_1$ and goes all the way up to infinity, passing through the point $\bs{a}_2$ as 
\begin{eqnarray}
\overrightarrow{\bs{a}_1 \bs{a}_2} \coloneqq \{(1-t) \bs{a}_1 + t \bs{a}_2 | \; t \in [0, \infty) \}.
\label{Notation: ray}
\end{eqnarray}
For a given $S \subseteq \mathbb{R}^3$ and $\varepsilon > 0$, we denote  the $\varepsilon-$neighborhood of the set $S$ by
\begin{eqnarray}
S^{\varepsilon} \coloneqq \{\bs{x} \in \mathbb{R}^3 \, | \, \dist(\bs{x}, S) \leq \varepsilon \}.
\label{epsilon neighborhood}
\end{eqnarray}
Finally, $\mathcal{H}^1$ denotes the Hausdorff measure of dimension one.

\begin{table}[t]
\rule{\textwidth}{0.4pt} \\[-10pt]
\rule{\textwidth}{0.4pt} 
\vspace{-0.5cm}
\begin{subequations}
\begin{align}
&\ell_0 \coloneqq \overrightarrow{\bs{p}_2 \bs{p}_1}         &  &\bs{e}_{0} \coloneqq \bs{e}_z             &  &\nu_{0} \coloneqq \bs{e}_{0} \mathcal{H}^1 \llcorner \ell_{0}\\
&\ell_{1, \theta} \coloneqq \overline{\bs{p}_{2, \theta} \, \bs{p}_{3, \theta}}        &  &\bs{e}_{1, \theta} \coloneqq \frac{1}{4}\frac{\bs{p}_{3, \theta} - \bs{p}_{2, \theta}}{|\bs{p}_{3, \theta} - \bs{p}_{2, \theta}|}   &  &\nu_{1} \coloneqq \sum_{\theta \in \Theta} \bs{e}_{1, \theta} \mathcal{H}^1 \llcorner \ell_{1, \theta}\\
&\ell_{2, \theta} \coloneqq \overrightarrow{\bs{p}_{3, \theta} \, \bs{p}_{4, \theta}} &  &\bs{e}_{2, \theta} \coloneqq \frac{\bs{e}_z}{4}           &  &\nu_{2} \coloneqq  \sum_{\theta \in \Theta} \bs{e}_{2, \theta} \mathcal{H}^1 \llcorner \ell_{2, \theta}\\
&\ell_{3, \theta} \coloneqq \overrightarrow{\bs{q}_{2, \theta}^{1, 1} \, \bs{q}_{1, \theta}^{0, 1} \,} & &\bs{e}_{3, \theta} \coloneqq -\frac{\bs{e}_z}{4} & &\nu_{3} \coloneqq \sum_{\theta \in \Theta} \bs{e}_{3, \theta}\mathcal{H}^1 \llcorner \ell_{3, \theta}\\
&\ell_{4, \theta} \coloneqq \overline{\bs{q}_{2, \theta}^{-1, 1} \, \bs{q}_{2, \theta}^{1, 1} \,} & &\bs{e}_{4, \theta} \coloneqq -\frac{\bs{e}_z}{8} & &\nu_{4} \coloneqq \sum_{\theta \in \Theta} \bs{e}_{4, \theta}\mathcal{H}^1 \llcorner \ell_{4, \theta}\\
&\ell_{5, \theta}^{\, i, j} \coloneqq \overline{\bs{q}_{2, \theta}^{i, j} \, \bs{q}_{3, \theta}^{i, j}} & &\bs{e}_{5, \theta}^{\, i, j} \coloneqq \frac{1}{16}\frac{\bs{q}_{2, \theta}^{i, j} - \bs{q}_{3, \theta}^{i, j}}{|\bs{q}_{2, \theta}^{i, j} - \bs{q}_{3, \theta}^{i, j}|} & &\nu_{5} \coloneqq \sum_{\theta \in \Theta} \sum_{i, j \in J} \bs{e}_{5, \theta}^{\, i, j}\mathcal{H}^1 \llcorner \, \ell_{5, \theta}^{\, i, j}\\
&\ell_{6, \theta}^{\, i, j} \coloneqq \overline{\bs{q}_{3, \theta}^{i, j} \, \bs{q}_{4, \theta}^{i, j}} & &\bs{e}_{6, \theta}^{\, i, j} \coloneqq \frac{1}{16}\frac{\bs{q}_{3, \theta}^{i, j} - \bs{q}_{4, \theta}^{i, j}}{|\bs{q}_{3, \theta}^{i, j} - \bs{q}_{4, \theta}^{i, j}|} & &\nu_{6} \coloneqq \sum_{\theta \in \Theta} \sum_{i, j \in J} \bs{e}_{6, \theta}^{\, i, j}\mathcal{H}^1 \llcorner \, \ell_{6, \theta}^{\, i, j}\\
&\ell_{7, \theta}^{\, i, j} \coloneqq \overrightarrow{\bs{q}_{4, \theta}^{i, j} \, \bs{q}_{5, \theta}^{i, j} \,} & &\bs{e}_{7, \theta}^{\, i, j} \coloneqq -\frac{\bs{e}_z}{16} & &\nu_{7} \coloneqq \sum_{\theta \in \Theta} \sum_{i, j \in J} \bs{e}_{7, \theta}^{\, i, j}\mathcal{H}^1 \llcorner \, \ell_{7, \theta}^{\, i, j}\\
&\ell_{8, \theta} \coloneqq \overrightarrow{\bs{q}_{2, \theta}^{0, 1} \, \bs{q}_{1, \theta}^{0, 1} \;} & &\bs{e}_{8, \theta} \coloneqq -\frac{\bs{e}_z}{4} & &\nu_{8} \coloneqq  \sum_{\theta \in \Theta} \bs{e}_{8, \theta} \mathcal{H}^1 \llcorner \ell_{8, \theta}\\
&\ell_{9, \theta} \coloneqq \ol{\bs{p}_{2} \, \bs{q}_{2, \theta}^{0, 1} \,} & &\bs{e}_{9, \theta} \coloneqq \frac{1}{4} \frac{\bs{q}_{2, \theta}^{0, 1} - \bs{p}_{2}}{|\bs{q}_{2, \theta}^{0, 1} - \bs{p}_{2}|} & &\nu_{9} \coloneqq \sum_{\theta \in \Theta} \bs{e}_{9, \theta}\mathcal{H}^1 \llcorner \, \ell_{9, \theta}
\end{align}
\label{Construction: vip definitions}
\end{subequations} \\[-15pt]
\rule{\textwidth}{0.4pt} \\[-10pt]
\rule{\textwidth}{0.4pt} 
\caption{A few useful definitions: line segments or rays (column one), vectors in $\mathbb{R}^3$ (column two) and vector-valued measures (column three).}
\label{Table: ell nu}
\end{table}

Using Table \ref{Table: ell nu}, we  now define a few vector-valued measures as
\begin{subequations}
\begin{eqnarray}
&& \nu^u \coloneqq \nu_{0} + \nu_{1} + \nu_{2}, \\
&& \nu^d \coloneqq \nu_{3} + \nu_{4} + \nu_{5} + \nu_{6} + \nu_{7}, \\
&& \nu_b \coloneqq \nu_{0} +  \nu_{8} + \nu_{9}.
\end{eqnarray}
\label{Construction: a few imp. measures}
\end{subequations}
The measure $\nu^u$ will be used in constructing the upward moving part of the flow field $\ol{\bs{u}}$ and $\nu^d$ will be used for constructing the downward moving part of the flow field, whereas, $\nu_b$ will be useful in constructing the flow field $\ol{\bs{u}}_b$. We also define a few useful sets as
\begin{subequations}
\begin{eqnarray}
&& S^{u} \coloneqq \ell_0 \cup \bigcup_{\theta \in \Theta} \ell_{1, \theta} \cup \ell_{2, \theta}, \\
&& S^{d} \coloneqq \bigcup_{\theta \in \Theta} \left(\ell_{3, \theta} \cup \ell_{4, \theta} \cup \bigcup_{i, j \in J} \ell_{5, \theta}^{i, j} \cup \ell_{6, \theta}^{i, j} \cup \ell_{7, \theta}^{i, j} \right), \\
&& S \coloneqq S^{u} \cup S^{d}, \\
&& S_{b} \coloneqq \ell_0 \cup \bigcup_{\theta \in \Theta} \ell_{8, \theta} \cup \ell_{9, \theta}.
\end{eqnarray}
\label{Construction: a few imp. setts}
\end{subequations}

To regularize the measures, we define a family of mollifiers. Let $\varphi \in C^\infty_c(\mathbb{R}^3)$ be any radial bump function whose support lies in $|\bs{x}| \leq 1$, such as
\begin{eqnarray}
\varphi(\bs{x}) \coloneqq \ol{\varphi}(|\bs{x}|),
\end{eqnarray}
where $\ol{\varphi}:\mathbb{R} \to \mathbb{R}$ is defined as
\begin{eqnarray}
\ol{\varphi}(r) \coloneqq 
\begin{cases}
c \exp \left(\frac{1}{r^2 - 1}\right) \quad \text{if} \quad |r| < 1, \\
0 \quad \text{if} \quad |r| \geq 1,
\end{cases}
\label{Construction: over bar phi}
\end{eqnarray} 
and $c$ is chosen such that $\int_{\mathbb{R}^3} \varphi(\bs{x}) \, \text{d}  \bs{x} = 1$.
For any $\varepsilon > 0$, we then define a standard mollifier as
\begin{eqnarray}
\varphi_\varepsilon(\bs{x}) \coloneqq \frac{1}{\varepsilon^3}\varphi\left(\frac{\bs{x}}{\varepsilon}\right).
\label{Construction: phi epsilon}
\end{eqnarray}
We use this definition of mollifier and the measures (\ref{Construction: a few imp. measures}{\color{blue}a-d}) to define the velocity fields
\addtocounter{equation}{1}
\begin{align}
\begin{rcases}
\overline{\bs{u}}^u_1 \coloneqq \varphi_{\gamma} * \nu^u, & \qquad \overline{\bs{u}}^u_2 \coloneqq \varphi_{\frac{\gamma}{2}} * \nu^u, \nonumber \\
\overline{\bs{u}}^d_1 \coloneqq \varphi_{\gamma} * \nu^d, & \qquad \overline{\bs{u}}^d_2 \coloneqq \varphi_{\frac{\gamma}{2}} * \nu^d, \nonumber \\
\ol{\bs{u}}_b \coloneqq \varphi_{\gamma} * \nu_b. &
\end{rcases}
\tag{\theequation a-e}
\label{Construction: a few imp. vel. fields}
\end{align}
From the definition of $\varphi_\varepsilon$ in (\ref{Construction: phi epsilon}) and the definition of the velocity fields (\ref{Construction: a few imp. vel. fields}), we see that
\addtocounter{equation}{1}
\begin{align}
\begin{rcases}
    \supp \ol{\bs{u}}_1^{\, u} \cup \supp \ol{\bs{u}}_2^{\, u} \subseteq S^{\, u, \gamma},\qquad &  \supp \ol{\bs{u}}_1^{\, d} \cup \supp \ol{\bs{u}}_2^{\, d} \subseteq S^{\, d, \gamma},  \nonumber \\
    \supp \ol{\bs{u}}_{b} \subseteq S_{b}^{ \gamma}.  
\end{rcases} \nonumber
\tag{\theequation a-c}
\end{align}
Here, we added $\gamma$ in the superscripts to mean $\gamma$-neighborhood of the sets (see definition (\ref{epsilon neighborhood})). Also, from the definition (\ref{Construction: a few imp. vel. fields}), we see that all the velocity fields belong to $L^\infty(\mathbb{R}^3; \mathbb{R}^3)$.

Our next task is to show that the velocity fields as defined in (\ref{Construction: a few imp. vel. fields}) belong to $C^\infty(\mathbb{R}^3; \mathbb{R}^3)$ and are divergence free. We start with the following definition.
\begin{definition}[Kirchhoff's junction]
Let $\widehat{\bs{p}} \in \mathbb{R}^3$ be a point and $\wh{\bs{e}}_j \in \mathbb{R}^3$, for $j = 1$ to $n \in \mathbb{N}$, be different non-zero vectors. Also, let $o_j \in \{-1, 1\}$, for $j \in \{1, \dots n\}$, be $n$ numbers. We say $\widehat{\bs{p}}$ together with the set of pairs $\wh{\bs{e}}_j$ and $o_j$ forms a  Kirchhoff's junction if
\begin{eqnarray}
\sum_{j = 1}^{n} o_j |\wh{\bs{e}}_j| = 0.
\end{eqnarray}
\end{definition}
For every Kirchhoff's junction defined above, we can associate a vector-valued Radon measure $\wh{\nu}$. First define $n$ rays emanating from $\widehat{\bs{p}}$ as $\wh{\ell}_j \coloneqq \{\bs{y}_j(t)\, | \, t \in [0, \infty)\},$
where $\bs{y}_j: \mathbb{R} \to \mathbb{R}^3$ are curves which in the parametric form are given by 
$ \bs{y}_j(t) \coloneqq \wh{\bs{p}} + t \wh{\bs{e}}_j/|\wh{\bs{e}}_j|,$ for $t \in [0, \infty).$
Consider the vector-valued Radon measures supported on these rays as
$ \wh{\nu}_j \coloneqq o_j \widehat{\bs{e}}_j \mathcal{H}^1 \llcorner \widehat{\ell}_j.$
Using these measures, we define a measure corresponding to the Kirchhoff's junction as 
\begin{eqnarray}
\wh{\nu} \coloneqq \sum_{j=1}^{n}\wh{\nu}_j.
\label{Construction of the flow: nu}
\end{eqnarray}
Next, we state an important lemma.
\begin{lemma}
\label{Construction of the flow: junction lemma}
Let $\psi \in C^\infty_{c}(\mathbb{R}^3)$ be a radially symmetric mollifier such that the support of $\psi$ lies in $|\bs{x}| \leq \varepsilon$, for some $\varepsilon > 0$. Assume that $\widehat{\bs{p}} \in \mathbb{R}^3$ and a set of $n$ pairs, $\wh{\bs{e}}_j \in \mathbb{R}^3$ and $o_j \in \{-1, 1\}$, for $j = 1$ to $n \in \mathbb{N}$, forms a Kirchhoff's junction. Let $\wh{\nu}$ be the associated vector-valued Radon measure to this junction. Then the velocity field given by
$\widehat{\bs{u}}\coloneqq \psi * \widehat{\nu}$
belongs to $C^\infty(\mathbb{R}^3; \mathbb{R}^3)$ and is divergence-free.
\end{lemma}
\begin{proof}[Proof of Lemma \ref{Construction of the flow: junction lemma}]
By differentiating under the integral sign in the expression of $\widehat{\bs{u}}$, we immediately see that $\widehat{\bs{u}} \in C^\infty(\mathbb{R}^3; \mathbb{R}^3)$. Next, for any $\bs{x}_0 \in \mathbb{R}^3$, the following calculation holds
\begin{eqnarray}
&& (\nabla \cdot \wh{\bs{u}})(\bs{x}_0) = \sum_{j = 1}^n o_j \int_{\mathbb{R}^3} \widehat{\bs{e}}_j \cdot \nabla \psi(\bs{x}_0 - \bs{y})  \, {\rm d} \mathcal{H}^1 \llcorner \widehat{\ell}_j(\bs{y})   = \sum_{j = 1}^n o_j \int_{0}^\infty \widehat{\bs{e}}_j \cdot \nabla \psi(\bs{x}_0 - \bs{y}_j(t_j))  \, {\rm d} t_j \nonumber \\ && = - \sum_{j = 1}^n o_j |\wh{\bs{e}}_j| \int_{0}^\infty \frac{\partial \psi(\bs{x}_0 - \bs{y}_j(t_j))}{\partial t_j}  \, {\rm d} t_j  = - \sum_{j = 1}^n o_j |\wh{\bs{e}}_j| \left. \psi(\bs{x}_0 - \bs{y}_j(t_j)) \right|_0^{\infty} = \psi(\wh{\bs{p}})\sum_{j=1}^{n} o_j |\wh{\bs{e}}_j|. \nonumber
\end{eqnarray}
Finally, using the assumption of the Kirchhoff's junction, implies $\nabla \cdot \wh{\bs{u}} \equiv 0$.
\end{proof}

\begin{corollary}
\label{Construction of the flow: junction corollary}
Let $\psi \in C^\infty_{c}(\mathbb{R}^3)$ be a radially symmetric mollifier such that the support of $\psi$ lies in $|\bs{x}| \leq \varepsilon$, for some $\varepsilon > 0$. Let $\widehat{\bs{p}}_i \in \mathbb{R}^3$, where $i \in \{1, \dots m\}$, be $m$ points which are part of $m$ different Kirchhoff's junctions and let $\wh{\nu}_i$ be the vector-valued Radon measures associated to each of the Kirchhoff's junction. Then for the vector-valued Radon measure defined as $\reallywidecheck{\nu} \coloneqq \sum_{i=1}^m \wh{\nu}_i$, the velocity field given by $\reallywidecheck{\bs{u}} \coloneqq \psi * \reallywidecheck{\nu}$
belongs to $C^\infty(\mathbb{R}^3; \mathbb{R}^3)$ and is divergence-free.
\end{corollary}

\begin{lemma} \label{Construction of the flow: measures on rays}
Let $\wh{\bs{p}}_1$ and $\wh{\bs{p}}_2$ be two different points in $\mathbb{R}^3$. Let $\wh{\ell}_{12} = \ol{\wh{\bs{p}}_1 \wh{\bs{p}}_2}$ and $\wh{\bs{e}}_{12} = c (\wh{\bs{p}}_2 - \wh{\bs{p}}_1)$ for some $c > 0$. Then a vector-valued measure defined as $\wh{\nu}_{12} \coloneqq \wh{\bs{e}}_{12} \mathcal{H}^1 \llcorner \wh{\ell}_{12}$ can also be written as
$ \wh{\nu}_{12} = \wh{\nu}_{in} + \wh{\nu}_{out},$
where $\wh{\nu}_{in} = (-\wh{\bs{e}}_{12}) \mathcal{H}^1 \llcorner \wh{\ell}_{in}$, $\wh{\nu}_{out} = \wh{\bs{e}}_{12} \mathcal{H}^1 \llcorner \wh{\ell}_{out}$,
$\wh{\ell}_{in} = \{\wh{\bs{p}}_2 + t \bs{e}_{12} \, | \, t \in [0, \infty) \},$ and $\wh{\ell}_{out} = \{\wh{\bs{p}}_1 + t \bs{e}_{12} \, | \, t \in [0, \infty) \}.$
\end{lemma}
\begin{proof}[Proof of \ref{Construction of the flow: measures on rays}]
We can write $
\wh{\nu}_{out} = \wh{\nu}_{out} \llcorner \wh{\ell}_{12} + \wh{\nu}_{out} \llcorner \wh{\ell}_{in}.$
Now $\wh{\nu}_{out} \llcorner \wh{\ell}_{12}$ coincides with $\wh{\nu}_{12}$ in $\mathbb{R}^3$, whereas $\wh{\nu}_{in} + \wh{\nu}_{out} \llcorner \wh{\ell}_{in}$ is a zero measure, which then finishes the proof.
\end{proof}

A tedious verification shows that using Lemma \ref{Construction of the flow: measures on rays}, the vector-valued measures (\ref{Construction: a few imp. measures}{\color{blue}a-c}) can be written as a sum of vector-valued measures associated with different Kirchhoff's junctions. Therefore, the velocity fields as defined in (\ref{Construction: a few imp. vel. fields}) belong to $C^\infty(\mathbb{R}^3; \mathbb{R}^3)$ and are divergence free. Here, we write down the Kirchhoff's junctions such that the sum of associated vector-valued measures is $\nu^u$:
\begin{center}
\begin{tabular}{ c c c }
 Junction No. & The point $\wh{\bs{p}}$ & The set of pairs of $\wh{\bs{e}}_j$ and $o_j$ \\ 
 1 & $\bs{p}_2$ & $\left\{\left(\bs{e}_z, 1\right), \left(\frac{\bs{e}_x}{4}, -1\right), \left(\frac{\bs{e}_y}{4}, -1\right), \left(-\frac{\bs{e}_x}{4}, -1\right), \left(-\frac{\bs{e}_y}{4}, -1\right)\right\}$ \\ [5pt] 
 2 & $\bs{p}_{3, 0}$ & $\left\{\left(-\frac{\bs{e}_x}{4}, 1\right), \left(\frac{\bs{e}_z}{4}, -1\right)\right\}$\\ [5pt]   
 3 & $\bs{p}_{3, \frac{\pi}{2}}$ & $\left\{\left(-\frac{\bs{e}_y}{4}, 1\right), \left(\frac{\bs{e}_z}{4}, -1\right)\right\}$\\ [5pt]  
 4 & $\bs{p}_{3, \pi}$ & $\left\{\left(\frac{\bs{e}_x}{4}, 1\right), \left(\frac{\bs{e}_z}{4}, -1\right)\right\}$\\ [5pt]   
 5 & $\bs{p}_{3, \frac{3 \pi}{2}}$ & $\left\{\left(\frac{\bs{e}_y}{4}, 1\right), \left(\frac{\bs{e}_z}{4}, -1\right)\right\}$
\end{tabular}
\end{center}
It can be shown that a similar decomposition exists for the other three measures defined in (\ref{Construction: a few imp. measures}).

\paragraph{Patching up $\ol{\bs{u}}_1$ and $\ol{\bs{u}}_2$: Construction in the reducer region} \noindent
\\
\\
To design the velocity field $\ol{\bs{u}}$, we need to patch the velocity fields $\ol{\bs{u}}_1$ and $\ol{\bs{u}}_2$  by defining an appropriate velocity field in the reducer region. Therefore, at this point, we shift our focus to designing velocity field in the reducer region. 

We start by considering a simple example of one reducer, where we design such a velocity field. Let's define a function $m: \mathbb{R} \to \mathbb{R}$ as
\begin{eqnarray}
m(r) \coloneqq \frac{1}{\gamma^3} \int_{-\infty}^{\infty} \ol{\varphi} \left(\frac{\sqrt{x^{\prime 2} + r^2}}{\gamma}\right) \, {\rm d} x^\prime,
\label{Construction: the function h}
\end{eqnarray}
where $\ol{\varphi}$ is defined in (\ref{Construction: over bar phi}). In what follows, we will use $$\varrho \text{ as a placeholder for } \sqrt{y^2+z^2}$$ in rest of the section. With these definitions in hand, we define two velocity fields $\ol{\bs{u}}_s, \ol{\bs{u}}_e:\mathbb{R}^3 \to \mathbb{R}^3$ as
\begin{eqnarray}
\ol{\bs{u}}_s(\bs{x}) \coloneqq (\ol{u}_{x, s}, \ol{u}_{y, s}, \ol{u}_{z, s}) \coloneqq (m(\varrho), 0, 0), \quad \ol{\bs{u}}_e(\bs{x}) \coloneqq (\ol{u}_{x, e}, \ol{u}_{y, e}, \ol{u}_{z, e}) \coloneqq (4 m(2 \varrho), 0, 0) \quad \text{for} \quad \bs{x} \in \mathbb{R}^3. \nonumber \\
\label{Construction of the flow: red start end vel}
\end{eqnarray}
As $\ol{\varphi}$ has a compact support, therefore, $\ol{\bs{u}}_s, \ol{\bs{u}}_e \in L^\infty(\mathbb{R}^3, \mathbb{R}^3)$. The arguments given in Appendix \ref{Appendix: A few basic lemmas} show that both of these velocity fields also belong to $C^\infty(\mathbb{R}^3, \mathbb{R}^3)$. Furthermore, it is clear that both the velocity fields are divergence free. Finally, one can verify that the volume flux through any  plane parallel to the $yz$-plane is same for both the velocity fields.

The task at hand is to come up with a divergence free velocity field $\ol{\bs{u}}_c$ such that it coincides with $\bs{u}_s$ in the region $x \leq 0$ and it coincides with $\bs{u}_e$ in the region $ \gamma \leq x$, and it belongs to $L^\infty(\mathbb{R}^3; \mathbb{R}^3) \cap C^\infty(\mathbb{R}^3, \mathbb{R}^3)$. To ensure the required velocity field is divergence free, we work with streamfunctions.  The strategy is to define the velocity field in the reducer region ($0 < x < \gamma$) based on a streamfunction which smoothly matches with streamfunction corresponding to the velocity field $\bs{u}_s$ for $x \leq 0$ and with streamfunction corresponding to the velocity field $\bs{u}_e$ for $x \geq \gamma$. To pursue this idea, we define two functions $\Psi_s, \Psi_e : \mathbb{R} \to \mathbb{R}$ as
\begin{eqnarray}
\Psi_s(r) \coloneqq \int_{0}^{|r|} r^\prime m(r^\prime) \, {\rm d} r^\prime, \qquad \Psi_e(r) \coloneqq 4 \int_{0}^{|r|} r^\prime m(2 r^\prime) \, {\rm d} r^\prime. \nonumber
\label{Construction: Psis Psie}
\end{eqnarray}
Next, we define $\Psi_{c}:\mathbb{R}^2 \to \mathbb{R}$ as
\begin{eqnarray}
\Psi_{c}(x, r) \coloneqq (1 - \eta_\gamma(x)) \Psi_s(r) + \eta_\gamma(x) \Psi_e(r), \nonumber
\end{eqnarray}
where $\eta_\varepsilon = \eta(x/\varepsilon)$ and $\eta$ is a smooth cut-off function such that $\eta \equiv 0$ for $x \leq 0$ and $\eta \equiv 1$ for $x \geq 1$.

The function $\Psi_{c}(x, r)$ may be understood as the axisymmetric streamfunction of the desired velocity field. With the help of $\Psi_{c}(x, r)$, we are ready to define the components of the velocity field that we wish to construct as
\begin{eqnarray}
&& \overline{u}_{x, c}(\bs{x}) \coloneqq (1 - \eta_\gamma(x)) m(\varrho) + 4 m(2\varrho) \eta_\gamma(x) \nonumber \\
&& \overline{u}_{y, c}(\bs{x}) \coloneqq
\begin{cases}
 y \frac{d \eta_\gamma}{d x} \frac{1}{\varrho^2} (\Psi_s(\varrho) - \Psi_e(\varrho)) \quad \text{if} \quad \varrho \neq 0,  \\
0 \quad \text{if} \quad y,z = 0.
\end{cases} \nonumber \\
&& \overline{u}_{z, c}(\bs{x}) \coloneqq
\begin{cases}
 z \frac{d \eta_\gamma}{d x} \frac{1}{\varrho^2} (\Psi_s(\varrho) - \Psi_e(\varrho)) \quad \text{if} \quad \varrho \neq 0, \\
0 \quad \text{if} \quad y,z = 0.
\end{cases} \nonumber
\label{Construction: red VEL flied}
\end{eqnarray}
where $\bs{x} \in \mathbb{R}^3$. The velocity field is then given by
\begin{eqnarray}
\ol{\bs{u}}_c \coloneqq (\overline{u}_{x, c}, \overline{u}_{y, c}, \overline{u}_{z, c}).
\label{Construction of the flow: Red comp vel field}
\end{eqnarray}
With this definition, we state the following lemma.
\begin{lemma} \label{Construction of the flow: reducer lemma}
Let the velocity field $\ol{\bs{u}}_c$ be as defined in (\ref{Construction of the flow: Red comp vel field}). Then it coincides with $\ol{\bs{u}}_s$ when $x \leq 0$ and with $\ol{\bs{u}}_e$ when $\gamma \leq x$. Furthermore, $\ol{\bs{u}}_c \in L^\infty(\mathbb{R}^3; \mathbb{R}^3) \cap C^\infty(\mathbb{R}^3; \mathbb{R}^3)$ and is divergence free with
\begin{eqnarray}
\supp \ol{\bs{u}}_c \subseteq \{(x, y, z) \, | \, y^2 + z^2 \leq \gamma^2\}.
\label{Construction of the flow: reducer lemma: supp}
\end{eqnarray} 
\end{lemma}
\begin{proof}
With the definition of function $m$ (\ref{Construction: the function h}) and noting that $\Psi_s(\varrho) = \Psi_e(\varrho)$ when $\varrho > \gamma$, we obtain (\ref{Construction of the flow: reducer lemma: supp}).
It is clear by construction that $\ol{\bs{u}}_c$ coincides with $\ol{\bs{u}}_s$ when $x \leq 0$ and with $\ol{\bs{u}}_e$ when $\gamma \leq x$ and therefore it is also infinite differentiable in these regions. To see the infinite differentiability in the region $0 < x < \gamma$, we use Lemma \ref{Construction of the flow: reducer lemma} given in Appendix \ref{Appendix: A few basic lemmas}. Finally, as the velocity field $\ol{\bs{u}}_c$ is defined based on a streamfunction, it is necessarily divergence-free.
\end{proof}

With this lemma in hand, we are ready to patch $\ol{\bs{u}}_1$ and $\ol{\bs{u}}_2$  using $\ol{\bs{u}}_c$ to construct the velocity field $\ol{\bs{u}}$. For this purpose, we define a few points in $\mathbb{R}^3$ as
\begin{eqnarray}
\bs{s}_0 = (1/8, 0, 0), \qquad \bs{s}_1^{i, j} = (1/8, j \delta, -i \delta) \quad \text{for} \quad i, j \in J.
\label{Construction: def of ss}
\end{eqnarray}
We also define two velocity field as
\begin{eqnarray}
\ol{\bs{u}}_{red}^{\, u} \coloneqq \frac{1}{4} T^{\bs{s}_0} \ol{\bs{u}}_c, \qquad  \qquad \ol{\bs{u}}_{red}^{\, d} \coloneqq - \frac{1}{16}\sum_{i,j \in J} T^{\bs{s}^{i, j}_1} \ol{\bs{u}}_c. \nonumber
\end{eqnarray}
As a result of Lemma \ref{Construction of the flow: reducer lemma}, the velocity fields
\begin{eqnarray}
\ol{\bs{u}}^{\, u} \coloneqq \ol{\bs{u}}_1^{\, u} \bs{1}_{\{|x|, |y| \leq 1/8\}}  + \ol{\bs{u}}_2^{\, u} \bs{1}_{\{|x| \geq 1/8+\gamma\} \cup \{|y| \geq 1/8+\gamma\}} + \sum_{\theta \in \Theta} \rho_{\theta} \; \ol{\bs{u}}_{red}^{\, u} \bs{1}_{\{1/8 < x < 1/8 + \gamma\}}, \nonumber \\
\ol{\bs{u}}^{\, d} \coloneqq \ol{\bs{u}}_1^{\, d} \bs{1}_{\{|x|, |y| \leq 1/8\}}  + \ol{\bs{u}}_2^{\, d} \bs{1}_{\{|x| \geq 1/8+\gamma\} \cup \{|y| \geq 1/8+\gamma\}} + \sum_{\theta \in \Theta} \rho_{\theta} \; \ol{\bs{u}}_{red}^{\, d} \bs{1}_{\{1/8 < x < 1/8 + \gamma\}}, \nonumber
\label{Construction: vel ub ud}
\end{eqnarray}
are uniformly bounded, infinitely differentiable, divergence free with $
\supp \ol{\bs{u}}^{\, u} \subseteq S^{\, u, \gamma}$ and $\supp \ol{\bs{u}}^{\, d} \subseteq S^{\, d, \gamma}$.
Finally, we arrive at the definition of the parent construct
\begin{eqnarray}
\ol{\bs{u}} \coloneqq \ol{\bs{u}}^u + \ol{\bs{u}}^b.
\label{Construction: vel v}
\end{eqnarray}
Summarizing the properties of the parent constructs $\ol{\bs{u}}$ and $\ol{\bs{u}}_b$ (from (\ref{Construction: a few imp. vel. fields})), we have $\ol{\bs{u}}, \ol{\bs{u}}_b \in C^\infty(\mathbb{R}^3; \mathbb{R}^3) \cap L^\infty(\mathbb{R}^3; \mathbb{R}^3)$, both obeying $\nabla \cdot \ol{\bs{u}} \equiv 0$ and $\nabla \cdot \ol{\bs{u}}_b \equiv 0$ with
\begin{eqnarray}
\supp \ol{\bs{u}} \subseteq S^{\, \gamma}, \qquad \supp \ol{\bs{u}}_b \subseteq S^{\, \gamma}_b.
\label{Construction: supp vel v}
\end{eqnarray}

Next, define a few points in $\mathbb{R}^3$ as
\begin{eqnarray}
\bs{\tau}_{\theta} \coloneqq \left(\frac{\cos \theta}{4}, \frac{\sin \theta}{4}, -\frac{1}{4}\right) \quad \text{for} \quad \theta \in \Theta. \nonumber
\end{eqnarray}
We now gather an important property of the parent constructs $\ol{\bs{u}}$ and $\ol{\bs{u}}_b$, which is that the velocity fields defined as 
\begin{subequations}
\begin{eqnarray}
&& \wt{\bs{u}} \coloneqq \ol{\bs{u}}(\bs{x}) \bs{1}_{\{z < 1/4\}} + \left(\sum_{\theta \in \Theta} T^{\bs{\tau}_\theta} \ol{\bs{u}}(2 \bs{x}) \right) \bs{1}_{\{z \geq 1/4\}}, \\
&& \wt{\bs{u}}_b  \coloneqq \ol{\bs{u}}(\bs{x}) \bs{1}_{\{z < 1/4\}} + \left(\sum_{\theta \in \Theta} T^{\bs{\tau}_\theta} \ol{\bs{u}}_b(2 \bs{x}) \right) \bs{1}_{\{z \geq 1/4\}}, \\
&& \wt{\bs{u}}_r \coloneqq \overline{\bs{u}}(\bs{x}) \, \bs{1}_{z \geq 0} + \left(-\overline{u}_x(x, y, -z), -\overline{u}_y(x, y, -z), \overline{u}_z(x, y, -z)\right) \, \bs{1}_{z < 0},
\end{eqnarray}
\label{imp. outcome velocity}
\end{subequations}
all belong to $C^\infty(\mathbb{R}^3; \mathbb{R}^3)$. For example, let's look at $\wt{\bs{u}}$. The infinite differentiability away from $z = 1/4$ is clear by definition. However, the velocity fields $\ol{\bs{u}}$ and $$\sum_{\theta \in \Theta} T^{\bs{\tau}_\theta} \ol{\bs{u}}(2 \bs{x})$$ are identical when $7/32 < z < 9/32$, which can be shown by writing down their explicit expressions in this region. Therefore, $\wt{\bs{u}}$ is infinitely differentiable at $z = 1/4$ as well. Similar arguments apply for $\wt{\bs{u}}_b$ and $\wt{\bs{u}}_r$.

\subsubsection{The scalar fields $\ol{\xi}$ and $\ol{\xi}_b$} \label{The parent temperature field}
The construction of $\ol{\xi}$ and $\ol{\xi}_b$ is relatively simple but somewhat different from that of $\ol{\bs{u}}$ and $\ol{\bs{u}}_b$. Recall from the detailed sketch given in section \ref{a detailed sketch}, we want to create a scalar field $\xi$ that is constant in the support of the velocity field $\bs{u}$ (with the exception of the boundary layer). To that end, we first define a few sets consisting of line segments and rays, which in a sense form the skeleton. This step is the same as in the previous subsection. We then consider a $\lambda$-neighborhood of this skeleton for sufficiently small positive $\lambda$.  Next, we mollify the indicator function of this $\lambda$-neighborhood set. If the mollification parameter, which we choose to be $\gamma$, is small compared with $\lambda$, we will have designed a smooth function supported in tubes of radius $\lambda + \gamma$ and constant in tubes of radius $\lambda - \gamma$ enveloping the skeleton. This strategy works everywhere except in the reducer region, where we design the scalar field using a cut-off function similar to the case of the velocity field. We now begin our construction.

In addition to (\ref{Construction: vip points}), we define a few extra points in $\mathbb{R}^3$ as
\begin{eqnarray}
\bs{p}_5 \coloneqq (0, 0, h), \qquad \bs{q}_6 \coloneqq (\delta, \delta, h), \nonumber
\end{eqnarray}
and a few rays 
\begin{eqnarray}
\ell_{10} \coloneqq \overrightarrow{\bs{p}_{5} \, \bs{p}_{1} \,},  \qquad \ell_{11, \theta} \coloneqq \overrightarrow{\bs{q}_{6, \theta} \, \bs{q}_{1, \theta}^{0,1} \;}, \nonumber
\end{eqnarray}
where $\bs{q}_{6, \theta} = \rho_{\theta}(\bs{q}_6)$. Here, we choose
$$h \coloneqq \frac{1}{16}.$$
To complement (\ref{Construction: a few imp. setts}), we also define
\begin{eqnarray}
\wt{S}^u_b \coloneqq \ell_{10}, \qquad \wt{S}^d_b \coloneqq \bigcup_{\theta \in \Theta} \ell_{11, \theta}. \nonumber
\end{eqnarray}
Remember, we chose $\lambda = 1/100$ and $\gamma = 1/500$ (see (\ref{Construction: choices that we made in our lives})). Now  using the definition of $\varepsilon$-neighborhood of a set (\ref{epsilon neighborhood}), we define the following scalar fields:
\addtocounter{equation}{1}
\begin{align}
\begin{rcases}
\ol{\xi}_1^u \coloneqq \varphi_{\gamma} * \bs{1}_{S^{u, \lambda}} \; ,  \quad
& \ol{\xi}_1^d \coloneqq -\varphi_{\gamma} * \bs{1}_{S^{d, \lambda}} \; ,  \quad \nonumber \\
\ol{\xi}_2^u \coloneqq \varphi_{\gamma/2} * \bs{1}_{S^{u, \lambda/2}} \; , \quad  & \ol{\xi}_2^d \coloneqq -\varphi_{\gamma/2} * \bs{1}_{S^{d, \lambda/2}} \; ,   \nonumber \\
\ol{\xi}_b^u \coloneqq \varphi_{\gamma} * \bs{1}_{\wt{S}^{u, \lambda}_b} \; ,  \quad & \ol{\xi}_b^d \coloneqq -\varphi_{\gamma} * \bs{1}_{\wt{S}^{d, \lambda}_b} \; ,
\end{rcases}
\tag{\theequation a-f}
\label{Construction: a few temp fields}
\end{align}
which belong to $C^\infty(\mathbb{R}^3)$ as a result of the following lemma.
\begin{lemma}
\label{Construction: smth moll indi opn set}
Let $f : \mathbb{R}^3 \to \mathbb{R}$ be a locally integrable function and let $\psi \in C^\infty_{c}(\mathbb{R}^3)$ be a radially symmetric mollifier such that the support of $\psi$ lies in $|\bs{x}| \leq \epsilon$. Then the function given by $g = \psi * f$
belongs to $C^\infty(\mathbb{R}^3)$.
\end{lemma}
\begin{proof}
By differentiating under the integral sign in the expression of $g$, one can finish the proof.
\end{proof}
From the definitions (\ref{Construction: a few temp fields}), we notice
\addtocounter{equation}{1}
\begin{align}
\begin{rcases}
\supp \ol{\xi}_1^{\, u} \cup \supp \ol{\xi}_2^{\, u} \subseteq S^{u, \, \lambda + \gamma}, \qquad & \supp \ol{\xi}_1^{\, d} \cup \supp \ol{\xi}_2^{\, d} \subseteq S^{d, \, \lambda + \gamma},  \nonumber \\ 
\supp \ol{\xi}^u_{b} \subseteq \wt{S}^{u, \lambda + \gamma}_{b}, \qquad & \supp \ol{\xi}^d_{b} \subseteq \wt{S}^{d, \lambda + \gamma}_{b}.
\end{rcases} \nonumber
\tag{\theequation a-d}
\label{Construction: support of temp field}
\end{align}
Moreover,
\addtocounter{equation}{1}
\begin{align}
\begin{rcases}
\ol{\xi}_1^{\, u}(\bs{x}) = \ol{\xi}_2^{\, u}(\bs{x}) = 1 \quad \text{when} \quad \bs{x} \in S^{u, \, \frac{\lambda - \gamma}{2}}, \qquad & \ol{\xi}_1^{\, d}(\bs{x}) = \ol{\xi}_2^{\, d}(\bs{x}) = -1 \quad \text{when} \quad \bs{x} \in S^{d, \, \frac{\lambda - \gamma}{2}}  \nonumber \\
\ol{\xi}_b^u(\bs{x}) = 1 \quad \text{when} \quad \bs{x} \in \wt{S}^{u, \, \lambda - \gamma}_b, \qquad &\ol{\xi}_b^d(\bs{x}) = -1 \quad \text{when} \quad \bs{x} \in \wt{S}^{d, \, \lambda - \gamma}_b. \qquad \qquad
\end{rcases}
\tag{\theequation a-d}
\label{Construction: region const temp field}
\end{align}
Next, let's define two scalar fields for the construction in the reducer region
\begin{subequations}
\begin{eqnarray}
\ol{\xi}_{red}^{\, u}(\bs{x}) \coloneqq \ol{\xi}_1^{\, u}(\bs{x}) \, (1 - \eta_{\gamma}(x-1/8)) + \ol{\xi}_2^{\, u}(\bs{x}) \, \eta_{\gamma}(x-1/8) \quad \text{for} \quad \bs{x} \in \mathbb{R}^3, \nonumber \\
\ol{\xi}_{red}^{\, d}(\bs{x}) \coloneqq \ol{\xi}_1^{\, d}(\bs{x}) \, (1 - \eta_{\gamma}(x-1/8)) + \ol{\xi}_2^{\, d}(\bs{x}) \, \eta_{\gamma}(x-1/8) \quad \text{for} \quad \bs{x} \in \mathbb{R}^3. \nonumber
\end{eqnarray}
\end{subequations}
We can use them to define
\begin{subequations}
\begin{eqnarray}
\ol{\xi}^{\, u} \coloneqq \ol{\xi}_1^{\, u} \bs{1}_{\{|x|, |y| \leq 1/8\}} + \ol{\xi}_2^{\, u} \bs{1}_{\{|x|, |y| \geq 1/8+\gamma\}} + \sum_{\theta \in \Theta} \rho_\theta \; \ol{\xi}_{red}^{\, u} \bs{1}_{\{1/8 < x < 1/8+\gamma\}}, \nonumber \\
\ol{\xi}^{\, d} \coloneqq \ol{\xi}_1^{\, d} \bs{1}_{\{|x|, |y| \leq 1/8\}} + \ol{\xi}_2^{\, d} \bs{1}_{\{|x|, |y| \geq 1/8+\gamma\}} + \sum_{\theta \in \Theta} \rho_\theta \; \ol{\xi}_{red}^{\, d} \bs{1}_{\{1/8 < x < 1/8+\gamma\}}. \nonumber
\end{eqnarray}
\end{subequations}
It can be easily verified that $\ol{\xi}^{\, u}$ is infinitely differentiable, that its support lies in $S^{u, \, \lambda + \gamma}$ and that its value is $1$ when $\bs{x} \in  S^{u, \, \frac{\lambda - \gamma}{2}}$. Similarly, $\ol{\xi}^{\, d}$ is infinitely differentiable has a support that lies in $S^{d, \, \lambda + \gamma}$ and its value is $-1$ when $\bs{x} \in  S^{d, \, \frac{\lambda - \gamma}{2}}$. We finally define the parent copies for the scalar field as
\addtocounter{equation}{1}
\begin{align}
\ol{\xi} \coloneqq \ol{\xi}^{\, u} + \ol{\xi}^{\, d}, \qquad \ol{\xi}_b \coloneqq \ol{\xi}^{\, u}_b + \ol{\xi}^{\, d}_b.
\tag{\theequation a-b}
\label{Hard work xi}
\end{align}
In summary, $\ol{\xi}, \ol{\xi}_b \in L^\infty(\mathbb{R}^3) \cap C^\infty(\mathbb{R}^3)$, $\supp \ol{\xi} \subseteq S^{u, \, \lambda + \gamma} \cup S^{d, \, \lambda + \gamma}$ and $\supp \ol{\xi}_b \subseteq \wt{S}_b^{u, \, \lambda + \gamma} \cup \wt{S}_b^{d, \, \lambda + \gamma}$

Similar to the case of velocity field, an important outcome of our construction is that scalar fields 
\begin{subequations}
\begin{eqnarray}
&& \widetilde{\xi}(\bs{x}) \coloneqq \overline{\xi}(\bs{x}) \, \bs{1}_{z<1/4} + \sum_{\theta \in \Theta} T^{\bs{\tau}_\theta}(\overline{\xi}(2 \bs{x})) \, \bs{1}_{z \geq 1/4},  \\
&& \widetilde{\xi}_b(\bs{x}) \coloneqq \overline{\xi}(\bs{x}) \, \bs{1}_{z<1/4} + \sum_{\theta \in \Theta} T^{\bs{\tau}_\theta}(\overline{\xi}_b(2 \bs{x})) \, \bs{1}_{z \geq 1/4}, \\
&& \widetilde{\xi}_r(\bs{x}) \coloneqq \overline{\xi}(\bs{x}) \, \bs{1}_{z \geq 0} + \overline{\xi}(x, y, -z) \, \bs{1}_{z < 0}, 
\end{eqnarray}
\label{imp. outcome scalar}
\end{subequations}
all belong to $C^\infty(\mathbb{R}^3)$. 

Let's now gather some of the important properties of the parent constructs of the velocity field and the scalar field in the following proposition.
\begin{proposition}
\label{Structure of the near optimal flow: properties of parent network}
In the definitions (\ref{Construction: vel v}, \ref{Construction: a few imp. vel. fields}) and (\ref{Hard work xi}) the two velocity fields $\overline{\bs{u}}, \overline{\bs{u}}_b \in C^\infty(\mathbb{R}^3, \mathbb{R}^3) \cap L^\infty(\mathbb{R}^3, \mathbb{R}^3)$ and the scalar fields $\overline{\xi}, \overline{\xi}_b \in C^\infty(\mathbb{R}^3) \cap L^\infty(\mathbb{R}^3)$ are such that the following statements are true.
\\
(i) $\nabla \cdot \overline{\bs{u}} \equiv 0$ and $\nabla \cdot \overline{\bs{u}}_b \equiv 0$, \\
(ii) $\supp \overline{\bs{u}} \cup \supp \overline{\xi} \subseteq (-1/3, 1/3) \times (-1/3, 1/3) \times \mathbb{R}$, \\
(iii) $\supp \overline{\bs{u}}_b \cup \supp \overline{\xi}_{b} \subseteq (-1/3, 1/3) \times (-1/3, 1/3) \times (-\infty, 1/4)$, \\
(iv) $\overline{\bs{u}} \cdot \nabla \overline{\xi} \equiv 0$, while $\supp_z (\overline{\bs{u}}_b \cdot \nabla \overline{\xi}_{b}) \Subset (1/32, 5/64),$ \\
(v) $\int_{\mathbb{R}^2} \int_{z=0}^{1/4} \ol{u}_z \ol{\xi} \, {\rm d} \bs{x} \geq c_3 > 0$ and $\int_{\mathbb{R}^2} \int_{z=0}^{1/8} \ol{u}_{b, z} \ol{\xi}_b \, {\rm d} \bs{x} \geq 0.$\\
Here, $c_3$ is a positive constant independent of any parameter. Furthermore, the velocity fields $\widetilde{\bs{u}}, \widetilde{\bs{u}}_b, \widetilde{\bs{u}}_r$ as defined in (\ref{imp. outcome velocity}) and the scalar fields $\widetilde{\xi}, \widetilde{\xi}_b, \widetilde{\xi}_r$ defined in (\ref{imp. outcome scalar}), respectively belong to $ C^\infty(\mathbb{R}^3, \mathbb{R}^3)$ and $ C^\infty(\mathbb{R}^3)$.
\end{proposition}
\begin{proof}[Proof of Proposition \ref{Structure of the near optimal flow: properties of parent network}] 
We already proved all the points except $(iv)$ and $(v)$, which we prove now.

We first focus on point $(iv)$. We note that: (1) $\supp{\ol{\bs{u}}} = \supp{\ol{\bs{u}}^u} \cup \supp{\ol{\bs{u}}^d}$, (2) $\supp{\ol{\bs{u}}^u} \subseteq S^{u, \gamma} \subset S^{u, \frac{\lambda - \gamma}{2}}$, (3) $\supp{\ol{\bs{u}}^d} \subseteq S^{d, \gamma} \subset S^{d, \frac{\lambda - \gamma}{2}}$. Furthermore, because of our choices of $\delta$, $\gamma$ and $\lambda$, we see that $S^{u, \frac{\lambda - \gamma}{2}} \cap S^{d, \frac{\lambda - \gamma}{2}} = \emptyset$. Now, if $\bs{x} \notin S^{u, \frac{\lambda - \gamma}{2}} \cup S^{d, \frac{\lambda - \gamma}{2}}$, it clear that $(\ol{\bs{u}} \cdot \nabla \ol{\xi})(\bs{x}) = 0$, as for this case $\bs{x} \notin \supp{\ol{\bs{u}}}$. In the next case, when $\bs{x} \in S^{u, \frac{\lambda - \gamma}{2}}$, we have $\ol{\xi} \equiv 1$ which implies $(\ol{\bs{u}} \cdot \nabla \ol{\xi})(\bs{x}) = 0$. In a similar manner, one can show $(\ol{\bs{u}} \cdot \nabla \ol{\xi})(\bs{x}) = 0$ when $\bs{x} \in S^{d, \frac{\lambda - \gamma}{2}}$.


Next, we show that $\supp_z \ol{\bs{u}}_b \cdot \nabla \ol{\xi}_b \Subset (1/32, 5/64)$. We first note that
\begin{eqnarray}
\{z < h\} \cap \supp{\ol{\bs{u}}}_b \subset \wt{S}^{u, \lambda - \gamma}_b \cup \wt{S}^{d, \lambda - \gamma}_b, \nonumber
\end{eqnarray}
and that $\wt{S}^{u, \lambda - \gamma}_b \cap \wt{S}^{d, \lambda - \gamma}_b = \emptyset$. For $z < h$, we proceed as in the last paragraph to show $(\ol{\bs{u}}_b \cdot \nabla \ol{\xi}_b)(\bs{x}) = 0.$ When $z > h + \lambda + \gamma$, from the definition of $\wt{S}^u_b$ and $\wt{S}^d_u$, we see that $\ol{\xi}_b \equiv 0$, therefore, $(\ol{\bs{u}}_b \cdot \nabla \ol{\xi}_b)(\bs{x}) = 0$ in this region as well. To summarize, $\supp_z \ol{\bs{u}}_b \cdot \nabla \ol{\xi}_b \subseteq [h, h+\gamma+\lambda] \Subset (1/32, 5/64).$

Now we move to point $(v)$ of Proposition \ref{Structure of the near optimal flow: properties of parent network}. From a straightforward calculation, we see that
\begin{eqnarray}
\int_{\mathbb{R}^3 \cap \{0 < z < 1/4\}} \ol{u}_z \ol{\xi} \, \text{d}  \bs{x} && = \int_{(\supp{\ol{\bs{u}}^u} \cup \supp{\ol{\bs{u}}^d}) \cap \{0 < z < 1/4\}} \ol{u}_z \ol{\xi} \, \text{d}  \bs{x} \nonumber \\
&& = \int_{\supp{\ol{\bs{u}}^u}  \cap \{0 < z < 1/4\}} \ol{u}_z \ol{\xi} \, \text{d} \bs{x} + \int_{\supp{\ol{\bs{u}}^d}  \cap \{0 < z < 1/4\}} \ol{u}_z \ol{\xi} \, \text{d}  \bs{x} \nonumber \\
&& = \int_{\supp{\ol{\bs{u}}^u}  \cap \{0 < z < 1/4\}} \ol{u}^u_z \ol{\xi} \, \text{d} \bs{x} + \int_{\supp{\ol{\bs{u}}^d}  \cap \{0 < z < 1/4\}} \ol{u}^d_z \ol{\xi} \, \text{d}  \bs{x} \nonumber \\
&& = \int_{\supp{\ol{\bs{u}}^u}  \cap \{0 < z < 1/4\}} \ol{u}^u_z \, \text{d}  \bs{x} - \int_{\supp{\ol{\bs{u}}^d}  \cap \{0 < z < 1/4\}} \ol{u}^d_z \, \text{d}  \bs{x} \nonumber \\
&& = \int_{\mathbb{R}^3 \cap \{0 < z < 1/4\}} \ol{u}^u_z \, \text{d}  \bs{x} - \int_{\mathbb{R}^3 \cap \{0 < z < 1/4\}} \ol{u}^d_z \, \text{d}  \bs{x} \nonumber \\
&& = \frac{1}{4}\int_{\mathbb{R}^2} \ol{u}^u_z (\cdot ,0) \, \text{d}  x \text{d}  y - \frac{1}{4}\int_{\mathbb{R}^2} \ol{u}^d_z (\cdot ,0) \, \text{d}  x \text{d}  y \nonumber \\
&& = c_3 > 0,
\end{eqnarray}
where $c_3$ is some constant. To obtain the fourth line, we used the fact that $\xi(\bs{x}) = 1$ when $\bs{x} \in \supp{\ol{\bs{u}}^u}$ and $\xi(\bs{x}) = -1$ when $\bs{x} \in \supp{\ol{\bs{u}}^d}$. To obtain the sixth line, we used the fact that $\ol{\bs{u}}^u$ and $\ol{\bs{u}}^d$ are divergence-free and that their support is bounded in the $xy$-plane, which in turn implies that the volume flux through any horizontal section is the same, i.e.,
\begin{eqnarray}
\int_{\mathbb{R}^2} \ol{u}^u_z (\cdot ,z) \, \text{d}  x \text{d}  y = \int_{\mathbb{R}^2} \ol{u}^u_z (\cdot ,0) \, \text{d}  x \text{d}  y \quad \text{and} \quad \int_{\mathbb{R}^2} \ol{u}^d_z (\cdot ,z) \, \text{d}  x \text{d}  y = \int_{\mathbb{R}^2} \ol{u}^d_z (\cdot ,0) \, \text{d}  x \text{d}  y \quad \text{for any} \quad z. \nonumber
\end{eqnarray}
To show $$\int_{\mathbb{R}^3 \cap \{0 < z < 1/4\}} \ol{u}_{b, z} \ol{\xi}_b \, \text{d}  \bs{x}  \geq 0,$$ we simply note that $\ol{\xi}_b \equiv 0$ when $z \geq h + \lambda + \gamma$ and for $z < h + \lambda + \gamma$, the velocity $\ol{\bs{u}}_b$ is unidirectional (only the $z$-component is non-zero). Furthermore, in this region, wherever $\ol{u}_{b, z} > 0$, we have $\ol{\xi}_b \geq 0$ and wherever $\ol{u}_{b, z} < 0$, we have $\ol{\xi}_b \leq 0$.
\end{proof}

\subsection{Main copies $\ol{\bs{u}}_N$  and $\ol{\xi}_N$: Proof of Proposition \ref{Structure of the near optimal flow: flow fields between stripes}}
Let's begin with a few useful definitions. First, let
\begin{eqnarray}
z_{i} \coloneqq \frac{1}{2} - \frac{1}{2^{i+1}} \quad \text{for} \quad i \in \mathbb{Z}_{\geq 0}, \nonumber
\end{eqnarray}
mark the vertical positions of the interfaces of different layers, while the intervals
\begin{eqnarray}
Z_{i} \coloneqq [z_{i-1}, z_{i}) \quad \text{for} \quad i \in \mathbb{N}, 
\label{Strategy: Zn}
\end{eqnarray}
denote the different layers. We define the set
\begin{eqnarray}
F \coloneqq \{(1, 0), (-1, 0), (0, 1), (0, -1)\} \nonumber
\end{eqnarray}
which we use to define sets of nodal points as
\begin{eqnarray}
\mathcal{N}_i = \left\{(x, y, z_i) \left| \; x = \sum_{j = 1}^{|i|} \frac{\alpha_j}{2^{j+1}}, \; y = \sum_{j = 1}^{|i|} \frac{\beta_j}{2^{j+1}}, \; (\alpha_j, \beta_j) \in F \right\} \right. \quad \text{for} \quad i \in \mathbb{N}.
\label{Structure of the near optimal flow: def F i}
\end{eqnarray}

\begin{proof}[Proof of Proposition \ref{Structure of the near optimal flow: flow fields between stripes}] 
For a given integer $N \geq 1$, we need to construct a velocity field $\overline{\bs{u}}_N$ and a scalar field $\overline{\xi}_N$ in the domain $D$ such that they satisfy the properties specified in Proposition \ref{Structure of the near optimal flow: flow fields between stripes}. To that end, we start by creating an intermediate flow field $\overline{\bs{u}}_{int, 1} : \mathbb{R}^3 \to \mathbb{R}^3$, whose support lies in $z \geq 0$ and is defined as follows:
\begin{eqnarray}
\overline{\bs{u}}_{int, 1}(\bs{x}) \coloneqq \overline{\bs{u}}(\bs{x}) \, \bs{1}_{Z_1} + \sum_{i = 1}^{N-1} \sum_{\bs{p} \in \mathcal{N}_i} T^{\, \bs{p}}(\overline{\bs{u}}(2^i \bs{x})) \bs{1}_{Z_{i+1}} + \sum_{\bs{p} \in \mathcal{N}_N} T^{\, \bs{p}}(\overline{\bs{u}}_b(2^N \bs{x})) \, \bs{1}_{Z_{N+1}} \quad \text{for} \quad \bs{x} \in \mathbb{R}^3.
\label{Strategy: u int 1}
\end{eqnarray}
To create $\overline{\bs{u}}_N$, we glue $\overline{\bs{u}}_{int, 1}$ and its mirror reflection about $z=0$. Let
\begin{eqnarray}
\overline{\bs{u}}_{int, 2}(\bs{x}) \coloneqq \left(-\overline{u}_{x, int, 1}(x, y, -z), -\overline{u}_{y, int, 1}(x, y, -z), \overline{u}_{z, int, 1}(x, y, -z)\right) \bs{1}_{z < 0} \quad \text{for} \quad \bs{x} \in \mathbb{R}^3.
\label{Strategy: u int 2}
\end{eqnarray}
Notice that the signs of $x$ and $y$ components are flipped to maintain the divergence-free condition. We finally define $\overline{\bs{u}}_N$ as
\begin{eqnarray}
\overline{\bs{u}}_N(\bs{x}) \coloneqq \overline{\bs{u}}_{int, 1}(\bs{x}) + \overline{\bs{u}}_{int, 2}(\bs{x}).  
\label{Strategy: uN}
\end{eqnarray}
Note that $\supp \overline{\bs{u}}_N \Subset D$ and it is really the restriction of $\overline{\bs{u}}_N$ to $D$, which we continue to call $\overline{\bs{u}}_N$, that we use in the proof of Proposition \ref{Structure of the near optimal flow: flow fields between stripes}. We then define $\overline{\xi}_N$ in a similar way. First, we define an intermediate scalar field $\overline{\xi}_{int, 1}: \mathbb{R}^3 \to \mathbb{R}$ as
\begin{eqnarray}
\overline{\xi}_{int, 1}(\bs{x}) \coloneqq \overline{\xi}(\bs{x}) \, \bs{1}_{Z_1} + \sum_{i = 1}^{N-1} \sum_{\bs{p} \in \mathcal{N}_i} T^{\, \bs{p}}(\overline{\xi}(2^i \bs{x})) \bs{1}_{Z_{i+1}} + \sum_{\bs{p} \in \mathcal{N}_N} T^{\, \bs{p}}(\overline{\xi}_b(2^N \bs{x})) \, \bs{1}_{Z_{N+1}} \quad \text{for} \quad \bs{x} \in \mathbb{R}^3, \nonumber
\end{eqnarray}
and its reflection about $z = 0$ as
\begin{eqnarray}
\overline{\xi}_{int, 2}(\bs{x}) = \overline{\xi}_{int, 1}(x, y, -z) \bs{1}_{z < 0} \quad \text{for} \quad \bs{x} \in \mathbb{R}^3, \nonumber
\end{eqnarray}
using which we define
\begin{eqnarray}
\overline{\xi}_{N}(\bs{x}) \coloneqq \overline{\xi}_{int, 1}(\bs{x}) + \overline{\xi}_{int, 2}(\bs{x}) \quad \text{for} \quad \bs{x} \in \mathbb{R}^3.
\end{eqnarray}
As before, $\supp \overline{\xi}_{N} \Subset D$ and  it is the restriction of $\overline{\xi}_{N}$ to $D$, which we continue to denote as $\overline{\xi}_{N}$, that we use in Proposition \ref{Structure of the near optimal flow: flow fields between stripes}. We claim that the velocity field $\overline{\bs{u}}_N$ and the scalar field $\overline{\xi}_N$ defined here satisfy all the requirements stated in Proposition \ref{Structure of the near optimal flow: flow fields between stripes}.

We first show that $\supp \overline{\bs{u}}_N \Subset (-1/2, 1/2) \times (-1/2, 1/2) \times (-z_{N+2}, z_{N+2}) \Subset D$. It is clear from the definition of $\overline{\bs{u}}_N$ given in (\ref{Strategy: uN}) along with (\ref{Strategy: u int 1}), (\ref{Strategy: u int 2}) and the definition of $Z_i$ in (\ref{Strategy: Zn}) that if $\wh{\bs{x}} \in \supp \overline{\bs{u}}_N$ then 
\begin{eqnarray}
\wh{z} \in [-z_{N+1}, z_{N+1}] \subset (-z_{N+2}, z_{N+2}).
\label{Strat: prop 1 inf 1}
\end{eqnarray}
Next from the statement (ii) in Proposition \ref{Structure of the near optimal flow: properties of parent network}, we note that if $\wh{\bs{x}} \in \supp \overline{\bs{u}}(2^i \bs{x})$ then
$$\wh{x}, \wh{y} \in \left(-\frac{1}{3 \cdot 2^i}, \, \frac{1}{3 \cdot 2^i}\right).$$ Also, note that if $\bs{p} \in \mathcal{N}_i$ for $i \in \mathbb{N}$, then $$|p_x|, |p_y| \leq \frac{1}{2} - \frac{1}{2^{i-1}},$$ Combining these two pieces of information tells us that if $\wh{\bs{x}} \in \supp \, T^{\, \bs{p}} \overline{\bs{u}}(2^i \bs{x})$ then
\begin{eqnarray}
\wh{x}, \wh{y} \in \left(-\frac{1}{2} + \frac{1}{3 \cdot 2^i}, \, \frac{1}{2} - \frac{1}{3 \cdot 2^i}\right) \subset \left(-\frac{1}{2}, \, \frac{1}{2} \right).
\label{Strat: prop 1 inf 2}
\end{eqnarray}
Finally, combining (\ref{Strat: prop 1 inf 1}) and (\ref{Strat: prop 1 inf 2}) with the definition (\ref{Strategy: uN}) gives 
\begin{eqnarray}
\supp \overline{\bs{u}}_N \Subset (-1/2, 1/2) \times (-1/2, 1/2) \times (-z_{N+2}, z_{N+2}) \Subset D.
\label{Strat: fin supp uN}
\end{eqnarray}

We now show that $\overline{\bs{u}}_N$ is infinitely differentiable, which together with (\ref{Strat: fin supp uN}) will imply $\overline{\bs{u}}_N \in C^\infty_{c}(D; \mathbb{R}^3)$. Let's first define two sets
\begin{eqnarray}
&& \Lambda \coloneqq \left(\bigcup_{0 \leq i \leq N-1} (z_i, z_{i+1})\right) \cup \left(\bigcup_{0 \leq i \leq N-1} (-z_{i+1}, -z_i) \right) \cup (z_N, 1/2) \cup (-1/2, -z_N), \nonumber \\
&& \Gamma \coloneqq \{z_0, z_1, -z_1, \dots  z_N, -z_N\}. \nonumber
\end{eqnarray}

It is easy to see from (\ref{Strategy: u int 1}), (\ref{Strategy: u int 2}), (\ref{Strategy: uN}) and from the infinite differentiability of $\ol{\bs{u}}$ and $\ol{\bs{u}}_b$ in Proposition \ref{Structure of the near optimal flow: properties of parent network} that $\overline{\bs{u}}_N(\bs{x})$ is infinitely differentiable when $z \in \Lambda$. Therefore, the only thing we still need to show  is that $\overline{\bs{u}}_N(\bs{x})$ is infinitely differentiable when $z \in \Gamma$, i.e., at the interfaces.

The fact $\widetilde{\bs{u}}_r$ in Proposition \ref{Structure of the near optimal flow: properties of parent network} belongs to $C^\infty(\mathbb{R}^3, \mathbb{R}^3)$ and $\overline{\bs{u}}_N$ coincides with $\widetilde{\bs{u}}_r$ when $z \in (-z_1, z_1)$, implies $\overline{\bs{u}}_N(\bs{x})$ is infinite differentiable when $z = z_0$. Now if $\bs{x}$ is such that $z \in (z_0, z_2)$ then $\overline{\bs{u}}_N(\bs{x})$ coincides with $\widetilde{\bs{u}}_b(\bs{x})$ when $N =1$ or it coincides with $\widetilde{\bs{u}}(\bs{x})$ when $N > 1$, which then concludes the infinite differentiability of $\overline{\bs{u}}_N$ at $z = z_1$. A similar argument can be applied to conclude the infinite differentiability at $z = -z_1$. In the last case, when $N > 1$ and $i \in \{2, \dots N-1\}$, then one can show
\begin{eqnarray}
\overline{\bs{u}}_N(\bs{x}) && = \sum_{\bs{p} \in \mathcal{N}_{i-1}} T^{\, \bs{p}}(\overline{\bs{u}}(2^{i-1} \bs{x})) \bs{1}_{Z_{i}} + \sum_{\bs{p} \in \mathcal{N}_i} T^{\, \bs{p}}(\overline{\bs{u}}(2^i \bs{x})) \bs{1}_{Z_{i+1}}  \nonumber \\ 
&& = \sum_{\bs{p} \in \mathcal{N}_{i-1}} T^{\, \bs{p}}\left( \left(\ol{\bs{u}} + \sum_{\bs{p}^\prime \in \mathcal{N}_1} T^{\, \bs{p}^\prime} \ol{\bs{u}}\right)(2^{i-1}\bs{x}) \right) \nonumber \\  && =  \sum_{\bs{p} \in \mathcal{N}_{i-1}} T^{\, \bs{p}}(\widetilde{\bs{u}}(2^{i-1} \bs{x})) \quad \text{when} \quad \bs{x} \in (z_{i-1}, z_{i+1}), \nonumber
\end{eqnarray}
or when $i = N$, then
\begin{eqnarray}
\overline{\bs{u}}_N(\bs{x}) =  \sum_{\bs{p} \in \mathcal{N}_{i-1}} T^{\, \bs{p}}(\widetilde{\bs{u}}_b(2^{i-1} \bs{x})) \quad \text{when} \quad \bs{x} \in (z_{i-1}, z_{i+1}), \nonumber
\end{eqnarray}
which then establishes that $\overline{\bs{u}}_N(\bs{x})$ is infinitely differentiable when $z = z_i$ for $2 \leq i \leq N$. A similar argument applies when  $z = -z_i$ for $2 \leq i \leq N$, which finishes the proof of $\overline{\bs{u}}_N \in C^\infty_{c}(D; \mathbb{R}^3)$. We note similar arguments will also work to show $\overline{\xi}_N \in C^\infty_{c}(D)$.

It is now fairly easy prove (i) in Proposition \ref{Structure of the near optimal flow: flow fields between stripes}. It is trivial to see that $\nabla \cdot \overline{\bs{u}}_N = 0$ when $z \in \Lambda$. As $\overline{\bs{u}}_N \in C^\infty_{c}(D; \mathbb{R}^3)$, the derivatives of $\overline{\bs{u}}_N$ are continuous in $D$, which leads us to conclude that $\nabla \cdot \overline{\bs{u}}_N = 0$ everywhere in $D$.

Next, we see that (\ref{Strat: fin supp uN}) and a similar conclusion derived for $\overline{\xi}_N$ proves (ii) in Proposition \ref{Structure of the near optimal flow: flow fields between stripes}.

To prove (iii) in Proposition \ref{Structure of the near optimal flow: flow fields between stripes}, we need the following simple lemma.

\begin{lemma}
\label{Structure of the near optimal flow: no intersection lemma}
For $i \in \mathbb{N}$, let $\bs{p}_1, \bs{p}_2 \in \mathcal{N}_i$ such that $\bs{p}_1 \neq \bs{p}_2$, then
\begin{subequations} 
\begin{eqnarray}
\left(\supp T^{\, \bs{p}_1} \overline{\bs{u}}(2^i \bs{x}) \cup \supp T^{\, \bs{p}_1} \overline{\xi}(2^i \bs{x}) \right)  \bigcap \left(\supp T^{\, \bs{p}_2} \overline{\bs{u}}(2^i \bs{x}) \cup \supp T^{\, \bs{p}_2} \overline{\xi}(2^i \bs{x}) \right) = \emptyset, \nonumber \\
\left(\supp T^{\, \bs{p}_1} \overline{\bs{u}}_b(2^i \bs{x}) \cup \supp T^{\, \bs{p}_1} \overline{\xi}_b(2^i \bs{x}) \right) \bigcap \left(\supp T^{\, \bs{p}_2} \overline{\bs{u}}_b(2^i \bs{x}) \cup \supp T^{\, \bs{p}_2} \overline{\xi}_b(2^i \bs{x}) \right) = \emptyset. \nonumber
\end{eqnarray}
\end{subequations}
\end{lemma}
\begin{proof}[Proof of Lemma \ref{Structure of the near optimal flow: no intersection lemma}]
As $\bs{p}_1 \neq \bs{p}_2$, from the definition (\ref{Structure of the near optimal flow: def F i}) of $\mathcal{N}_i$ , we note the following lower bound on the absolute difference of $x$ and $y$ coordinates of $\bs{p}_1$ and $\bs{p}_2$:
\begin{eqnarray}
|\bs{p}_{1, x} - \bs{p}_{2, x}| \geq \frac{1}{2^i} \quad \text{and} \quad |\bs{p}_{1, y} - \bs{p}_{2, y}| \geq \frac{1}{2^i}, \nonumber
\end{eqnarray}
which implies
\begin{eqnarray}
|\bs{p}_{1} - \bs{p}_{2}|_{\parallel} \geq \frac{\sqrt{2}}{2^i}. \nonumber
\end{eqnarray}
Now, if $\bs{x}_1 \in \left(\supp T^{\, \bs{p}_1} \overline{\bs{u}}(2^i \bs{x}) \cup \supp T^{\, \bs{p}_1} \overline{\xi}(2^i \bs{x}) \right)$ or $\bs{x}_1 \in \left(\supp T^{\, \bs{p}_1} \overline{\bs{u}}_b(2^i \bs{x}) \cup \supp T^{\, \bs{p}_1} \overline{\xi}_b(2^i \bs{x}) \right)$ and if $\bs{x}_2 \in \left(\supp T^{\, \bs{p}_2} \overline{\bs{u}}(2^i \bs{x}) \cup \supp T^{\, \bs{p}_2} \overline{\xi}(2^i \bs{x}) \right)$ or $\bs{x}_2 \in \left(\supp T^{\, \bs{p}_2} \overline{\bs{u}}_b(2^i \bs{x}) \cup \supp T^{\, \bs{p}_2} \overline{\xi}_b(2^i \bs{x}) \right)$, then using the statements (ii) and (iii) from Proposition \ref{Structure of the near optimal flow: properties of parent network}, we see that
\begin{eqnarray}
|\bs{x}_1 - \bs{p}_1|_{\parallel} \leq \frac{\sqrt{2}}{3 \cdot 2^i} \quad \text{and} \quad |\bs{x}_2 - \bs{p}_2|_{\parallel} \leq \frac{\sqrt{2}}{3 \cdot 2^i}. \nonumber
\end{eqnarray}
We can now finish the proof with a simple application of the triangle inequality as
\begin{eqnarray}
|\bs{x}_1 - \bs{x}_2| \geq  |\bs{p}_{1} - \bs{p}_{2}|_{\parallel} - |\bs{x}_1 - \bs{p}_1|_{\parallel} - |\bs{x}_2 - \bs{p}_2|_{\parallel} \geq \frac{\sqrt{2}}{3 \cdot 2^i}. \nonumber
\end{eqnarray}
\end{proof}
Using the lemma, we can write
\begin{eqnarray}
(\overline{\bs{u}}_{int, 1} \cdot \nabla \overline{\xi}_{int, 1})(\bs{x}) = \overline{\bs{u}}(\bs{x}) \cdot \nabla \overline{\xi} (\bs{x}) \, \bs{1}_{Z_1} + \sum_{i = 1}^{N-1} \sum_{\bs{p} \in \mathcal{N}_i} T^{\, \bs{p}}(\overline{\bs{u}}(2^i \bs{x})) \cdot \nabla T^{\, \bs{p}}(\overline{\xi}(2^i \bs{x})) \bs{1}_{Z_{i+1}} \nonumber \\ + \sum_{\bs{p} \in \mathcal{N}_N} T^{\, \bs{p}}(\overline{\bs{u}}_b(2^N \bs{x})) \cdot \nabla T^{\, \bs{p}}(\overline{\xi}_b(2^N \bs{x})) \, \bs{1}_{Z_{N+1}} \quad \text{for} \quad z \in \Lambda, \nonumber
\end{eqnarray}
which implies
\begin{eqnarray}
(\overline{\bs{u}}_{int, 1} \cdot \nabla \overline{\xi}_{int, 1})(\bs{x}) = \overline{\bs{u}}(\bs{x}) \cdot \nabla \overline{\xi} (\bs{x}) \, \bs{1}_{Z_1} + \sum_{i = 1}^{N-1} \sum_{\bs{p} \in \mathcal{N}_i} T^{\, \bs{p}}\left(\overline{\bs{u}}(2^i \bs{x}) \cdot \nabla \overline{\xi}(2^i \bs{x})\right) \bs{1}_{Z_{i+1}} \nonumber \\ + \sum_{\bs{p} \in \mathcal{N}_N} T^{\, \bs{p}}\left(\overline{\bs{u}}_b(2^N \bs{x}) \cdot \nabla \overline{\xi}_b(2^N \bs{x})\right) \, \bs{1}_{Z_{N+1}} \quad \text{for} \quad z \in \Lambda.
\label{Strat: u int 1 int res}
\end{eqnarray}
Using (\ref{Strat: u int 1 int res}) and point (iv) from Proposition \ref{Structure of the near optimal flow: properties of parent network}, we conclude
\begin{eqnarray}
(\overline{\bs{u}}_{int, 1} \cdot \nabla \overline{\xi}_{int, 1})(\bs{x}) = 0 \quad \text{when} \quad z \in \Lambda \setminus \left(\frac{1}{2} - \frac{15}{32 \cdot 2^N}, \frac{1}{2} - \frac{27}{64 \cdot 2^N}\right)
\label{Strat: two supp res}
\end{eqnarray}
A simple calculation then shows that
\begin{eqnarray}
(\overline{\bs{u}}_{N} \cdot \nabla \overline{\xi}_{N})(x, y, z) = - (\overline{\bs{u}}_{N} \cdot \nabla \overline{\xi}_{N})(x, y, -z) \quad \text{when} \quad z \in \Lambda
\label{Strat: a sim cal}
\end{eqnarray}
which, combined with the result (\ref{Strat: two supp res}) and the fact that $\overline{\bs{u}}_N$ and the derivatives of $\overline{\xi}_N$ are continuous when $z \in \Gamma$, help us conclude
\begin{eqnarray}
(\overline{\bs{u}}_{N} \cdot \nabla \overline{\xi}_{N})(\bs{x}) = 0 \quad \text{when} \quad \bs{x} \in \Gamma. \nonumber
\end{eqnarray}
In total, we then have
\begin{eqnarray}
\supp_z (\overline{\bs{u}}_{N} \cdot \nabla \overline{\xi}_{N}) \Subset \left(\frac{1}{2} - \frac{15}{32 \cdot 2^N}, \frac{1}{2} - \frac{27}{64 \cdot 2^N}\right) \cup \left(-\frac{1}{2} + \frac{27}{64 \cdot 2^N}, -\frac{1}{2} + \frac{15}{32 \cdot 2^N}\right). \nonumber
\end{eqnarray}
To prove (iv) in Proposition \ref{Structure of the near optimal flow: flow fields between stripes}, we note from (\ref{Strat: u int 1 int res}) and point (iv) in Proposition \ref{Structure of the near optimal flow: properties of parent network} that 
\begin{eqnarray}
(\overline{\bs{u}}_{int, 1} \cdot \nabla \overline{\xi}_{int, 1})(\bs{x}) = \sum_{\bs{p} \in \mathcal{N}_N} 2^N T^{\, \bs{p}}\left( (\overline{\bs{u}}_b \cdot \nabla \overline{\xi}_b) (2^N \bs{x})\right) \, \bs{1}_{Z_{N+1}} \quad \text{for} \quad z \in \Lambda, \nonumber
\end{eqnarray}
which when combined with Lemma \ref{Structure of the near optimal flow: no intersection lemma}, implies
\begin{eqnarray}
\norm{\overline{\bs{u}}_{int, 1} \cdot \nabla \overline{\xi}_{int, 1}}_{L^\infty(D)} \leq 2^N \norm{\overline{\bs{u}}_b \cdot \nabla \overline{\xi}_b}_{L^\infty(\mathbb{R}^3)}. \nonumber
\end{eqnarray}
Noting (\ref{Strat: a sim cal}) and that $\overline{\bs{u}}_N \cdot \nabla \overline{\xi}_{N}$ coincides with $\ol{\bs{u}}_{int, 1} \cdot \nabla \overline{\xi}_{int, 1}$ when $z > 0$, we have
\begin{eqnarray}
\norm{\overline{\bs{u}}_N \cdot \nabla \overline{\xi}_{N}}_{L^\infty(D)} \leq 2^N \norm{\overline{\bs{u}}_b \cdot \nabla \overline{\xi}_b}_{L^\infty(\mathbb{R}^3)}. \nonumber
\end{eqnarray}
Now $\overline{\bs{u}}_b \cdot \nabla \overline{\xi}_b$ is an infinite differentiable function and its support lies in a bounded set from (iii) and (iv) in Proposition \ref{Structure of the near optimal flow: properties of parent network}, therefore $\norm{\overline{\bs{u}}_b \cdot \nabla \overline{\xi}_b}_{L^\infty(\mathbb{R}^3)}$ is bounded and we can conclude that
\begin{eqnarray}
\norm{\overline{\bs{u}}_N \cdot \nabla \overline{\xi}_{N}}_{L^\infty(D)} \lesssim 2^N. \nonumber
\end{eqnarray}

Proof of (v) in Proposition \ref{Structure of the near optimal flow: flow fields between stripes} is a simple computation. Once again using Lemma \ref{Structure of the near optimal flow: no intersection lemma}, one can write the following
\begin{eqnarray}
\int_{D} |\nabla \overline{\bs{u}}_N|^2 \, {\rm d} \bs{x} && = 2 \int_{\{0 < z < 1/2\}} |\nabla \overline{\bs{u}}_{int, 1}|^2 \, {\rm d} \bs{x} \nonumber \\&& = 2 \int_{\{z \in Z_1\}} |\nabla \overline{\bs{u}}(\bs{x})|^2 \, {\rm d} \bs{x} + 2 \sum_{i = 1}^{N-1} \sum_{\bs{p} \in \mathcal{N}_i} \int_{\{ z \in Z_{i+1}\}} |\nabla  T^{\, \bs{p}}(\overline{\bs{u}}(2^i \bs{x}))|^2 \, {\rm d} \bs{x} \nonumber \\
&& \qquad \qquad \qquad + 2 \sum_{\bs{p} \in \mathcal{N}_N} \int_{\{ z\in Z_{N+1}\}} |\nabla T^{\, \bs{p}}(\overline{\bs{u}}_b(2^N \bs{x}))|^2  \, {\rm d} \bs{x}. \nonumber
\end{eqnarray} 
After an appropriate translation and dilation of the coordinate variables and noting that $| \mathcal{N}_i | = 4^i$, one can show that
\begin{eqnarray}
\int_{D} |\nabla \overline{\bs{u}}_N|^2 \, {\rm d} \bs{x} =  \left(\sum_{i = 0}^{N-1} 2^{i+1}\right)\int_{Z_1} |\nabla \overline{\bs{u}}(\bs{x})|^2 \, {\rm d} \bs{x} + 2^{N+1}\int_{\{z \in Z_1\}} |\nabla \overline{\bs{u}}_b(\bs{x})|^2 \, {\rm d} \bs{x} \nonumber \\
= 2^{N+1} \max\left\{\int_{\{z \in Z_1\}} |\nabla \overline{\bs{u}}(\bs{x})|^2 \, {\rm d} \bs{x}, \int_{\{z \in Z_1\}} |\nabla \overline{\bs{u}}_b(\bs{x})|^2 \, {\rm d} \bs{x}\right\} \lesssim 2^N. \nonumber
\end{eqnarray}
Similarly, one can also conclude 
\begin{eqnarray}
\int_{D} |\nabla \overline{\xi}_N|^2 \, {\rm d} \bs{x}  \lesssim 2^N, \nonumber
\end{eqnarray}
which then proves (v).

The proof of (vi) in Proposition \ref{Structure of the near optimal flow: flow fields between stripes} is also very similar to that of (v). We first write
\begin{eqnarray}
\int_{D} \ol{u}_{N, z} \, \ol{\xi}_{N} \, {\rm d} \bs{x} && = 2 \int_{\{0 < z < 1/2\} } \ol{u}_{int, 1, z} \, \ol{\xi}_{int, 1} \, {\rm d} \bs{x} \nonumber \\&& = 2 \int_{\{z \in Z_1\}} \ol{u}_z \, \ol{\xi} \, {\rm d} \bs{x} + 2 \sum_{i = 1}^{N-1} \sum_{\bs{p} \in \mathcal{N}_i} \int_{\{ z \in Z_{i+1}\}} T^{\, \bs{p}} \left(\ol{u}_z (2^i \bs{x}) \right) \, T^{\, \bs{p}} \left(\ol{\xi} (2^i \bs{x}) \right) \, {\rm d} \bs{x} \nonumber \\
&& \qquad \qquad \qquad + 2 \sum_{\bs{p} \in \mathcal{N}_N} \int_{\{ z\in Z_{N+1}\}} T^{\, \bs{p}} \left(\ol{u}_{b, z} (2^i \bs{x}) \right) \, T^{\, \bs{p}} \left(\ol{\xi}_b (2^i \bs{x}) \right) \, {\rm d} \bs{x}. \nonumber
\end{eqnarray} 
After an appropriate translation and dilation of the coordinate variables and noting that $| \mathcal{N}_i | = 4^i$, we obtain
\begin{eqnarray}
\int_{D} \ol{u}_{N, z} \, \ol{\xi}_{N} \, {\rm d} \bs{x} =  \left(\sum_{i = 0}^{N-1} 2^{-i+1}\right)\int_{\{z \in Z_1\}} \ol{u}_z \, \ol{\xi} \, {\rm d} \bs{x} + 2^{-N+1}\int_{\{z \in Z_1\}} \ol{u}_{b, z} \, \ol{\xi}_b \, {\rm d} \bs{x} \geq 2 c_3 > 0, \nonumber
\end{eqnarray}
where $c_3$ is a strictly positive constant independent of $N$.
\end{proof}

\section{A useful estimate for the solution of the Poisson's equation: Proof of Proposition \ref{Poisson's equation: inverse Laplace torus to D}}
\label{Section: An estimate solution of Poisson}

The aim of this section is to give an estimate on the solution of Poisson's equation $ \Delta \varphi = f $
solved between parallel boundaries with homogeneous Dirichlet boundary conditions imposed on $\varphi$. In particular, we are interested in obtaining bounds on the $L^2$ norm of $\nabla \varphi$ for a given specific form of the function $f$. The calculations done in this section will be helpful in establishing an upper bound on the nonlocal term $\dashint_{\Omega}|\nabla \Delta^{-1} \diverge (\bs{u} \xi)|^2$ from the section \ref{Strategy and proof} (see calculation (\ref{Strategy and proof of the main theorem: estimate nonlocal req. poisson})). The basic idea is to write down the solution of Poisson's equation using the Green's function method and then obtain estimates on the derivative of the Green's function to achieve our goal.  

The domain of interest for this section is
\begin{eqnarray}
D \coloneqq \mathbb{R} \times \mathbb{R} \times (-1/2, 1/2), \nonumber
\end{eqnarray}
as defined in section \ref{Notation} with boundary
\begin{eqnarray}
\partial D \coloneqq \partial D_+ \cup \partial D_- \coloneqq \mathbb{R} \times \mathbb{R} \times \{1/2\} \cup \mathbb{R} \times \mathbb{R} \times \{-1/2\}.  \nonumber
\end{eqnarray}
Now, suppose $\varphi$ solves Poisson's equation
\begin{eqnarray}
\Delta \varphi = f \quad \text{in } D, 
\label{The Poisson's equation: the Poisson's equation}
\end{eqnarray}
with boundary condition
\begin{eqnarray}
\varphi = 0 \quad \text{on} \quad \partial D. 
\label{The Poisson's equation: boundary conditions}
\end{eqnarray}
Then for a sufficiently smooth function $f$, we can write the solution of Poisson's equation using a Green's function
\begin{eqnarray}
\varphi(\bs{x}) = \int_D G(\bs{x}, \bs{x}^\prime) f(\bs{x}^\prime) \, {\rm d} \bs{x}, 
\label{The Poisson's equation: Green's function formula}
\end{eqnarray}
where $G : D \times D \to [-\infty, \infty]$ is given by
\begin{eqnarray}
G(\bs{x}, \bs{x}^\prime) \coloneqq  K(|\bs{x} - \bs{x}^\prime|_\parallel, z, z^\prime) 
\label{The Poisson's equation: Green's function}
\end{eqnarray}
and
\begin{eqnarray}
&& K(\sigma, z, z^\prime) \coloneqq  \int_{1}^\infty I(\sigma, z, z^\prime, \tau) \, \frac{{\rm d} \tau}{\sqrt{\tau^2 - 1}}, 
\label{The Poisson's equation: K}
\\
&& I(\sigma, z, z^\prime, \tau) \coloneqq 
\frac{\cos (\pi z) \cos (\pi z^\prime) \sinh (\pi \tau \sigma )}{ 2 \pi \left[ \cosh (\pi \tau \sigma ) + \cos \pi (z^\prime + z) \right] \left[ \cosh (\pi \tau \sigma ) - \cos \pi (z^\prime - z) \right]}.
\label{The Poisson's equation: I}
\end{eqnarray}
In particular, we have the following theorem
\begin{theorem}[Solution of the Poisson's equation]
\label{Poisson's equation: Solution of the Poisson's equation thm}
Let $f \in C^2(D) \cap L^\infty(D)$ whose is support lies a finite distance away from the boundary, i.e., $ \supp f \subseteq \mathbb{R} \times \mathbb{R} \times (-1/2 + \beta, 1/2 - \beta)$ for some $\beta \in (0, 1/2)$. Then $\varphi$ given by (\ref{The Poisson's equation: Green's function formula}) belongs to $C^2(D)$ and solves the Poisson's equation (\ref{The Poisson's equation: the Poisson's equation}) with boundary condition (\ref{The Poisson's equation: boundary conditions}).
\end{theorem}
\begin{proof}
The proof of the theorem is a standard one and is therefore omitted from the paper. The proof relies on the method of images to write the desired Green's function between parallel boundaries as a sum of appropriately translated Green's functions corresponding to the whole space $\mathbb{R}^3$, where the summation is then performed using Cauchy's residue theorem.
\end{proof}

Once we know that the solution $\varphi$ of the Poisson's equation is given by  (\ref{The Poisson's equation: Green's function formula}), we can use it to calculate $\nabla \varphi$. If $f \in L^\infty(D)$ then by an application of the mean value theorem and the dominated convergence theorem, we can perform differentiation under the integral sign in (\ref{The Poisson's equation: Green's function formula}), which leads to
\begin{eqnarray}
\nabla \varphi = \int_D \nabla_{\bs{x}} G(\bs{x}, \bs{x}^\prime) f(\bs{x}^\prime) \, {\rm d} \bs{x}.
\label{The Poisson's equation: grad phi}
\end{eqnarray}
From (\ref{The Poisson's equation: grad phi}), we see that estimates on $\nabla_{\bs{x}} G(\bs{x}, \bs{x}^\prime)$ can provide an upper bound on $|\nabla \varphi|$. Next, we state our result in that direction, but first, we note the following.

For clarity, we use $a$ and $b$ as placeholders for 
\begin{eqnarray}
\left|\frac{2}{\pi}\sin \left(\frac{\pi (z - z^\prime)}{2} \right)\right| \quad \text{and} \quad \left|\frac{2}{\pi}\cos \left( \frac{\pi (z + z^\prime)}{2} \right) \right|
\label{The Poisson's equation: a b placeholders}
\end{eqnarray}
respectively in the rest of this section and we will use the fact that
\begin{eqnarray}
b^2 - a^2 = \frac{4}{\pi^2} \cos \pi z \cos \pi z^\prime \geq 0 \quad \text{when} \quad z, z^\prime \in (-1/2, 1/2)
\label{The Poisson's equation: a b placeholders a simple equation}
\end{eqnarray}
in several places. We will use $c$ for a positive constant (not necessarily the same in all places) independent of any parameters.
\begin{proposition}
\label{Poisson's equation: uppper bound on grad phi prop}
Let $f \in L^\infty(D)$ and let $\varphi$ be defined by the formula (\ref{The Poisson's equation: Green's function formula}), then the following holds:
\begin{eqnarray}
|\nabla \varphi| (\bs{x}) \leq \norm{f}_{L^\infty(D)} \int_{\supp_z f} g(z, z^\prime) \, {\rm d} z^\prime \nonumber
\end{eqnarray}
where 
\begin{eqnarray}
g(z, z^\prime) = c \left(\log \left(1 + \frac{(b^2-a^2)}{a^2}\right) +  \frac{\cos \pi z^\prime}{b}\right)
\label{The Poisson's equation: the function g z z prime}
\end{eqnarray}
and $c > 0$ is a positive constant.
\end{proposition}
The functions $f$ that are of special interests to us are those which are supported in a ``thin layers'' close to the boundaries. From Proposition \ref{Poisson's equation: uppper bound on grad phi prop}, we can derive the following result for such functions.
\begin{corollary}
\label{Poisson's equation: uppper bound on grad phi prop for f with thin supp}
Let $f \in L^\infty(D)$ such that $\supp f \subseteq \mathbb{R} \times \mathbb{R} \times (1/2 - c_1 \varepsilon, 1/2 - c_2 \varepsilon) \cup \mathbb{R} \times \mathbb{R} \times (-1/2 + c_2 \varepsilon, -1/2 + c_1 \varepsilon)$, where $0 < c_2 < c_1 < 1$ and $\varepsilon < 1/4$ are three constants. If $\varphi$ is defined by the formula (\ref{The Poisson's equation: Green's function formula}), then the following holds
\begin{eqnarray}
\frac{1}{l_x l_y}\int_{-l_x/2}^{l_x/2} \int_{-l_y/2}^{l_y/2} \int_{-1/2}^{1/2} |\nabla \varphi|^2 \, {\rm d} z {\rm d} y {\rm d} x \lesssim \varepsilon^3 \norm{f}_{L^\infty(D)}^2.
\label{The Poisson's equation: f with thin supp grad phi estimate}
\end{eqnarray}
\end{corollary}

\begin{proof}[Proof of Proposition \ref{Poisson's equation: inverse Laplace torus to D}]
We identify a $l_x-l_y-$periodic function on $D$ with the function $f$. Then using Corollary \ref{Poisson's equation: uppper bound on grad phi prop for f with thin supp}, we can finish the proof.
\end{proof}
\begin{proof}[Proof of Corollary \ref{Poisson's equation: uppper bound on grad phi prop for f with thin supp}]
We note from Proposition \ref{Poisson's equation: uppper bound on grad phi prop}
\begin{eqnarray}
\frac{1}{l_x l_y} \int_{-l_y/2}^{l_y/2} \int_{-l_x/2}^{l_x/2} \int_{-1/2}^{1/2} |\nabla \varphi|^2 \, {\rm d} z {\rm d} x {\rm d} y \leq \norm{f}_{L^\infty(D)}^2 \int_{-1/2}^{1/2} \left(\int_{\supp_z f} g(z, z^\prime) \, {\rm d} z^\prime\right)^2 \, {\rm d} z \nonumber \\
 = \norm{f}_{L^\infty(D)}^2 \int_{0}^{1/2} \left(\int_{\supp_z f} g(z, z^\prime) \, {\rm d} z^\prime\right)^2 \, {\rm d} z + \norm{f}_{L^\infty(D)}^2 \int_{-1/2}^{0} \left(\int_{\supp_z f} g(z, z^\prime) \, {\rm d} z^\prime\right)^2 \, {\rm d} z.
\label{The Poisson's equation: main integral split}
\end{eqnarray}
We focus on obtaining a bound on the first term (where the integral is carried from $z = 0$ to $z = 1/2$) in (\ref{The Poisson's equation: main integral split}), as the calculation for the other integral is identical.

When $z^\prime \in (1/2 - c_1 \varepsilon, 1/2 - c_2 \varepsilon) \cup (-1/2 + c_2 \varepsilon, -1/2 + c_1 \varepsilon)$ and $z \geq 0$, the following simple succession of inequalities hold:
\begin{subequations}
\begin{eqnarray}
&& b \geq \frac{1}{\pi}\max\{\cos \pi z, \cos \pi z^\prime\} \\
&& b \geq \frac{1}{2 \pi}(\cos \pi z + \cos \pi z^\prime) \geq \frac{1}{4} \left(\frac{1}{2} - z + c_2 \varepsilon\right) \\
&& \frac{\pi c_2 \varepsilon}{2} \leq \cos \pi z^\prime \leq \pi c_1 \varepsilon \\
&& \frac{\pi}{4} - \frac{\pi z}{2} \leq \cos \pi z \leq \frac{\pi}{2} - \pi z \\
&& a \geq \frac{1}{2} \left(\frac{1}{2} - c_1 \varepsilon - z\right) \quad \text{when} \quad \frac{1}{2} - 2 c_1 \varepsilon \geq z \geq 0 \\
&& a \geq \frac{1}{2} |z - z^\prime| \quad \text{when} \quad \frac{1}{2} \geq z \geq \frac{1}{2} - 2 c_1 \varepsilon  \\
&& a \geq \frac{1}{5} \quad \text{when} \quad  z^\prime \in (-1/2 + c_2 \varepsilon, -1/2 + c_1 \varepsilon) \\
&& \log(1 + \alpha) \leq \alpha \quad \text{when} \quad \alpha \geq 0
\end{eqnarray}
\label{The Poisson's equation: basic trig inequalities}
\end{subequations}
Here, (\ref{The Poisson's equation: basic trig inequalities}{\color{blue}c}), (\ref{The Poisson's equation: basic trig inequalities}{\color{blue}d}), (\ref{The Poisson's equation: basic trig inequalities}{\color{blue}e}) and (\ref{The Poisson's equation: basic trig inequalities}{\color{blue}f}) are a simple consequence of the inequality $z/2 \leq \sin z \leq z$ when $z \in [0, \pi/2]$, whereas (\ref{The Poisson's equation: basic trig inequalities}{\color{blue}a}) is obtained by simple applications of trigonometric identities and (\ref{The Poisson's equation: basic trig inequalities}{\color{blue}b}) is a result of (\ref{The Poisson's equation: basic trig inequalities}{\color{blue}a}), (\ref{The Poisson's equation: basic trig inequalities}{\color{blue}c}) and (\ref{The Poisson's equation: basic trig inequalities}{\color{blue}d}). The result (\ref{The Poisson's equation: basic trig inequalities}{\color{blue}g}) is a consequence of the assumption $0 < c_2 < c_1 < 1$. Finally, (\ref{The Poisson's equation: basic trig inequalities}{\color{blue}h}) can be derived using a Taylor series expansion.

Next, using (\ref{The Poisson's equation: the function g z z prime}) and the Young's inequality, we can write
\begin{eqnarray}
\left(\int_{\supp_z f} g(z, z^\prime) \, {\rm d} z^\prime\right)^2 \lesssim \left(\int_{\supp_z f \cap \mathbb{R}_{+}} \log \left(1 + \frac{(b^2-a^2)}{a^2}\right) \, {\rm d} z^\prime\right)^2 \qquad \qquad  \qquad \qquad  \nonumber \\ + \left(\int_{\supp_z f \cap \mathbb{R}_{-}} \log \left(1 + \frac{(b^2-a^2)}{a^2}\right) \, {\rm d} z^\prime\right)^2   
+ \left(\int_{\supp_z f}  \frac{\cos \pi z^\prime}{b} \, {\rm d} z^\prime\right)^2
\label{The Poisson's equation: int in z prime split}
\end{eqnarray}

Using (\ref{The Poisson's equation: basic trig inequalities}{\color{blue}b}) and (\ref{The Poisson's equation: basic trig inequalities}{\color{blue}c}), the last term in (\ref{The Poisson's equation: int in z prime split}) can be bounded from above by
\begin{eqnarray}
\lesssim \frac{\varepsilon^4}{\left(\frac{1}{2} - z + c_2 \varepsilon\right)^2},
\label{The Poisson's equation: int in z prime split res 1}
\end{eqnarray}
Using (\ref{The Poisson's equation: a b placeholders a simple equation}), (\ref{The Poisson's equation: basic trig inequalities}{\color{blue}c}), (\ref{The Poisson's equation: basic trig inequalities}{\color{blue}g}) and (\ref{The Poisson's equation: basic trig inequalities}{\color{blue}h}), the  second term in (\ref{The Poisson's equation: int in z prime split}) satisfies the bound
\begin{eqnarray}
\lesssim \varepsilon^4.
\label{The Poisson's equation: int in z prime split res 2}
\end{eqnarray}
We divide the calculation of the first term in (\ref{The Poisson's equation: int in z prime split}) into two cases, when $0 \leq z \leq 1/2 - 2 c_1 \varepsilon$ and when $ 1/2 - 2 c_1 \varepsilon \leq z \leq 1/2$. In the first case, using (\ref{The Poisson's equation: a b placeholders a simple equation}), (\ref{The Poisson's equation: basic trig inequalities}{\color{blue}c}), (\ref{The Poisson's equation: basic trig inequalities}{\color{blue}d}), (\ref{The Poisson's equation: basic trig inequalities}{\color{blue}e}) and (\ref{The Poisson's equation: basic trig inequalities}{\color{blue}h}), we conclude the first term is 
\begin{eqnarray}
\lesssim \frac{\varepsilon^4}{\left(\frac{1}{2} - z - c_1 \varepsilon\right)^2}.
\label{The Poisson's equation: int in z prime split res 3}
\end{eqnarray}
In the second case, when $ 1/2 - 2 c_1 \varepsilon \leq z \leq 1/2$, we use (\ref{The Poisson's equation: a b placeholders a simple equation}), (\ref{The Poisson's equation: basic trig inequalities}{\color{blue}c}), (\ref{The Poisson's equation: basic trig inequalities}{\color{blue}d}), (\ref{The Poisson's equation: basic trig inequalities}{\color{blue}f}), which gives 
\begin{eqnarray}
\lesssim \varepsilon^2.
\label{The Poisson's equation: int in z prime split res 4}
\end{eqnarray}
Note that in this calculation we do not use the estimate (\ref{The Poisson's equation: basic trig inequalities}{\color{blue}h}). After using (\ref{The Poisson's equation: basic trig inequalities}{\color{blue}f}), we have a logarithmic singularity in the integrand but it is integrable.

Finally, collecting the results (\ref{The Poisson's equation: int in z prime split res 1}), (\ref{The Poisson's equation: int in z prime split res 2}), (\ref{The Poisson's equation: int in z prime split res 3}) and (\ref{The Poisson's equation: int in z prime split res 4}), and carrying out an integration in $z$ from $0$ to $1/2$, one can bound the first term in (\ref{The Poisson's equation: main integral split}) as
\begin{eqnarray}
\lesssim \varepsilon^3 \norm{f}_{L^\infty(D)}^2. \nonumber
\end{eqnarray} 
A similar calculation can be performed for the second term in (\ref{The Poisson's equation: main integral split}) and the same result can be derived which then finishes the proof.
\end{proof}

\subsection{Proof of Proposition \ref{Poisson's equation: uppper bound on grad phi prop}}
To prove Proposition \ref{Poisson's equation: uppper bound on grad phi prop}, we need to obtain estimates on $\nabla_{\bs{x}} G(\bs{x}, \bs{x}^\prime)$. From (\ref{The Poisson's equation: Green's function}), we notice that the derivative of $G(\bs{x}, \bs{x}^\prime)$ with respect to $x$ can be written as
\begin{eqnarray}
\frac{\partial}{\partial x} G(\bs{x}, \bs{x}^\prime) \; = \frac{\partial |\bs{x} - \bs{x}^\prime|_\parallel}{\partial x} \cdot \left. \frac{\partial}{\partial \sigma} K(\sigma, z, z^\prime) \right|_{\sigma = |\bs{x} - \bs{x}^\prime|_\parallel} \; = \; \frac{(x - x^\prime)}{|\bs{x} - \bs{x}^\prime|_\parallel} \cdot \left. \frac{\partial}{\partial \sigma} K(\sigma, z, z^\prime) \right|_{\sigma = |\bs{x} - \bs{x}^\prime|_\parallel}, \nonumber
\end{eqnarray}
which leads to the following estimate
\begin{eqnarray}
&& \left| \frac{\partial}{\partial x} G(\bs{x}, \bs{x}^\prime) \right| \leq \left. \left(\left|  \frac{\partial}{\partial \sigma} K(\sigma, z, z^\prime) \right| \right) \right|_{\sigma = |\bs{x} - \bs{x}^\prime|_\parallel}. 
\label{The Poisson's equation: G x derivative}
\end{eqnarray}
A similar calculation for the $y$-derivative of $G(\bs{x}, \bs{x}^\prime)$ leads to
\begin{eqnarray}
&& \left| \frac{\partial}{\partial y} G(\bs{x}, \bs{x}^\prime) \right| \leq \left. \left(\left|  \frac{\partial}{\partial \sigma} K(\sigma, z, z^\prime) \right| \right) \right|_{\sigma = |\bs{x} - \bs{x}^\prime|_\parallel}, 
\label{The Poisson's equation: G y derivative}
\end{eqnarray}
while the estimate for the $z$-derivative of $G(\bs{x}, \bs{x}^\prime)$ simply is
\begin{eqnarray}
&& \left| \frac{\partial}{\partial z} G(\bs{x}, \bs{x}^\prime) \right| \leq \left. \left(\left|  \frac{\partial}{\partial z} K(\sigma, z, z^\prime) \right| \right) \right|_{\sigma = |\bs{x} - \bs{x}^\prime|_\parallel}.
\label{The Poisson's equation: G z derivative}
\end{eqnarray}
Using (\ref{The Poisson's equation: G x derivative}), (\ref{The Poisson's equation: G y derivative}) and (\ref{The Poisson's equation: G z derivative}), we conclude that
\begin{eqnarray}
&& |\nabla_{\bs{x}} G(\bs{x}, \bs{x}^\prime)| \leq \left| \frac{\partial}{\partial x} G(\bs{x}, \bs{x}^\prime) \right| + \left| \frac{\partial}{\partial y} G(\bs{x}, \bs{x}^\prime) \right| + \left| \frac{\partial}{\partial z} G(\bs{x}, \bs{x}^\prime) \right| \nonumber \\ && \leq 2 \left. \left(\left|  \frac{\partial}{\partial \sigma} K(\sigma, z, z^\prime) \right| \right) \right|_{\sigma = |\bs{x} - \bs{x}^\prime|_\parallel} + \left. \left(\left|  \frac{\partial}{\partial z} K(\sigma, z, z^\prime) \right| \right) \right|_{\sigma = |\bs{x} - \bs{x}^\prime|_\parallel} \leq  H(|\bs{x} - \bs{x}^\prime|_\parallel, z, z^\prime),
\label{The Poisson's equation: the function H and a few consequences}
\end{eqnarray} 
for some suitable $H:\mathbb{R}_{+} \times (-1/2, 1/2) \times (-1/2, 1/2) \to [0, +\infty]$. It then follows that
\begin{eqnarray}
|\nabla \varphi (\bs{x})| && \leq \norm{f}_{L^\infty(D)} \int_{-\infty}^\infty \int_{-\infty}^\infty \int_{\supp_z f} H(|\bs{x} - \bs{x}^\prime|_\parallel, z, z^\prime) \, {\rm d} z^\prime {\rm d} x^\prime {\rm d} y^\prime, \nonumber \\
&& = \norm{f}_{L^\infty(D)} \int_{-\infty}^\infty \int_{-\infty}^\infty \int_{\supp_z f} H(|\bs{x}^\prime|_\parallel, z, z^\prime) \, {\rm d} z^\prime {\rm d} x^\prime {\rm d} y^\prime. \nonumber
\end{eqnarray}
By considering a transformation from Cartesian coordinates  to cylindrical coordinates
$$(x^\prime, y^\prime, z^\prime) \mapsto (\sigma, \theta, z^\prime)$$ 
one obtains
\begin{eqnarray}
|\nabla \varphi (\bs{x})|  && \leq  \norm{f}_{L^\infty(D)} \int_{0}^\infty \int_{0}^{2 \pi} \int_{\supp_z f} \sigma H(\sigma, z, z^\prime) \, {\rm d} z^\prime {\rm d} \theta {\rm d} \sigma, \nonumber \\
&& \lesssim \norm{f}_{L^\infty(D)} \int_{0}^\infty  \int_{\supp_z f} \sigma H(\sigma, z, z^\prime) \, {\rm d} z^\prime {\rm d} \sigma. \nonumber
\end{eqnarray}
So, to prove Proposition \ref{Poisson's equation: uppper bound on grad phi prop}, we need to find an appropriate $H(\sigma, z, z^\prime)$ and then perform the integral 
\begin{eqnarray}
\int_{0}^\infty  \sigma H(\sigma, z, z^\prime) \, {\rm d} \sigma
\label{The Poisson's equation: integral of H statement},
\end{eqnarray}
which is our next goal.

To calculate (\ref{The Poisson's equation: G x derivative}), (\ref{The Poisson's equation: G y derivative}) and (\ref{The Poisson's equation: G z derivative}), we need the derivative of $K(\sigma, z, z^\prime)$ with $\sigma$ and $z$. Using (\ref{The Poisson's equation: K}) and (\ref{The Poisson's equation: I}) along with an application of the mean value theorem and the dominated convergence theorem leads to
\begin{eqnarray}
\frac{\partial K}{\partial \sigma} = \int_{1}^\infty I_{\sigma 1} \, \frac{{\rm d} \tau}{\sqrt{\tau^2 - 1}} + \int_{1}^\infty I_{\sigma 2} \, \frac{{\rm d} \tau}{\sqrt{\tau^2 - 1}} + \int_{1}^\infty I_{\sigma 3} \, \frac{{\rm d} \tau}{\sqrt{\tau^2 - 1}},
\label{The Poisson's equation: K sigma derivative}
\end{eqnarray}
and 
\begin{eqnarray}
\frac{\partial K}{\partial z} = \int_{1}^\infty I_{z 1} \, \frac{{\rm d} \tau}{\sqrt{\tau^2 - 1}} + \int_{1}^\infty I_{z 2} \, \frac{{\rm d} \tau}{\sqrt{\tau^2 - 1}} + \int_{1}^\infty I_{z 3} \, \frac{{\rm d} \tau}{\sqrt{\tau^2 - 1}},
\label{The Poisson's equation: K z derivative}
\end{eqnarray}
where
\begin{subequations}
\begin{eqnarray}
&& I_{\sigma 1} \coloneqq
\frac{\cos (\pi z) \, \cos (\pi z^\prime) \, \tau \cosh (\pi \tau \sigma )}{ 2 \left[ \cosh (\pi \tau \sigma ) + \cos \pi (z^\prime + z) \right] \left[ \cosh (\pi \tau \sigma ) - \cos \pi (z^\prime - z) \right]}, \\
&& I_{\sigma 2} \coloneqq
- \frac{\cos (\pi z) \cos (\pi z^\prime) \tau \sinh^2 (\pi \tau \sigma )}{ 2 \left[ \cosh (\pi \tau \sigma ) + \cos \pi (z^\prime + z) \right]^2 \left[ \cosh (\pi \tau \sigma ) - \cos \pi (z^\prime - z) \right]},
\\
&& I_{\sigma 3} \coloneqq
- \frac{\cos (\pi z) \cos (\pi z^\prime) \tau \sinh^2 (\pi \tau \sigma )}{ 2 \left[ \cosh (\pi \tau \sigma ) + \cos \pi (z^\prime + z) \right] \left[ \cosh (\pi \tau \sigma ) - \cos \pi (z^\prime - z) \right]^2}, \\
&& I_{z 1} =
- \frac{\sin (\pi z) \cos (\pi z^\prime) \sinh (\pi \tau \sigma )}{ 2 \left[ \cosh (\pi \tau \sigma ) + \cos \pi (z^\prime + z) \right] \left[ \cosh (\pi \tau \sigma ) - \cos \pi (z^\prime - z) \right]}, \\
&& I_{z 2} =
\frac{\cos (\pi z) \cos (\pi z^\prime) \sin \pi (z^\prime + z) \sinh (\pi \tau \sigma )}{ 2 \left[ \cosh (\pi \tau \sigma ) + \cos \pi (z^\prime + z) \right]^2 \left[ \cosh (\pi \tau \sigma ) - \cos \pi (z^\prime - z) \right]}, \\
&& I_{z 3} =
\frac{\cos (\pi z) \cos (\pi z^\prime) \sin \pi (z^\prime - z) \sinh (\pi \tau \sigma )}{ 2 \left[ \cosh (\pi \tau \sigma ) + \cos \pi (z^\prime + z) \right] \left[ \cosh (\pi \tau \sigma ) - \cos \pi (z^\prime - z) \right]^2}.
\end{eqnarray}
\label{The Poisson's equation: I sigma z derivative parts}
\end{subequations}

Next, we state a few important lemmas to bound the derivatives of $K(\sigma, z, z^\prime)$. We will always implicitly assume that $z, z^\prime \in (-1/2, 1/2)$.
\begin{lemma}
\label{Poisson's equation: bound on I derivative large sigma}
Let $\sigma \geq \frac{1}{2 \pi}$, then we have

(i)  
\begin{eqnarray}
\int_1^\infty (|I_{\sigma 1}| + |I_{\sigma 2}| + |I_{\sigma 3}|) \, \frac{{\rm d} \tau}{\sqrt{\tau^2 - 1}} \lesssim \cos(\pi z^\prime) \exp(-\pi \sigma). \nonumber
\end{eqnarray}

(ii)  
\begin{eqnarray}
\int_1^\infty (|I_{z 1}| + |I_{z 2}| + |I_{z 3}|) \, \frac{{\rm d} \tau}{\sqrt{\tau^2 - 1}} \lesssim \cos(\pi z^\prime) \exp(-\pi \sigma). \nonumber
\end{eqnarray}
\end{lemma}

\begin{lemma}
\label{Poisson's equation: bound on I derivative small sigma large tau}
Let $0 < \sigma < \frac{1}{2 \pi}$, then we have

(i)  
\begin{eqnarray}
\int_{\frac{1}{\pi \sigma}}^\infty (|I_{\sigma 1}| + |I_{\sigma 2}| + |I_{\sigma 3}|) \, \frac{{\rm d} \tau}{\sqrt{\tau^2 - 1}} \lesssim \frac{\cos(\pi z^\prime)}{\sigma}.  \nonumber
\end{eqnarray}

(ii)  
\begin{eqnarray}
\int_{\frac{1}{\pi \sigma}}^\infty (|I_{z 1}| + |I_{z 2}| + |I_{z 3}|) \, \frac{{\rm d} \tau}{\sqrt{\tau^2 - 1}} \lesssim \cos(\pi z^\prime). \nonumber
\end{eqnarray}

\end{lemma}

\begin{lemma}
\label{Poisson's equation: bound on I derivative small sigma small tau}
Let $0 < \sigma < \frac{1}{2 \pi}$ and $z \neq z^\prime$, then we have

(i)  

\begin{eqnarray}
\int_{1}^{\frac{1}{\pi \sigma}} (|I_{\sigma 1}| + |I_{\sigma 2}| + |I_{\sigma 3}|) \, \frac{{\rm d} \tau}{\sqrt{\tau^2 - 1}} 
\lesssim \frac{1}{\sigma} \left[\frac{1}{\sqrt{\sigma^2 + a^2}} - \frac{1}{\sqrt{\sigma^2 + b^2}}\right]. && \nonumber
\end{eqnarray}

(ii)  

\begin{eqnarray}
\int_{1}^{\frac{1}{\pi \sigma}} |I_{z 1}| \, \frac{{\rm d} \tau}{\sqrt{\tau^2 - 1}} \lesssim |\tan \pi z| \left[\frac{1}{\sqrt{\sigma^2 + a^2}} - \frac{1}{\sqrt{\sigma^2 + b^2}}\right]. \nonumber
\end{eqnarray}

(iii)  

\begin{eqnarray}
\int_{1}^{\frac{1}{\pi \sigma}} |I_{z 2} + I_{z 3}| \, \frac{{\rm d} \tau}{\sqrt{\tau^2 - 1}} \lesssim \cos^2 \pi z \cos \pi z^\prime \left[P_1(\sigma, z, z^\prime) + P_2(\sigma, z, z^\prime)\right]. \nonumber
\end{eqnarray}
Here,
\begin{subequations}
\begin{eqnarray}
P_1(\sigma, z, z^\prime) \coloneqq \frac{2 a b \pi}{4 (b^2-a^2)^3} \left[\frac{4}{\sqrt{\sigma^2 + b^2}} - \frac{4}{\sqrt{\sigma^2 + a^2}} + (b^2-a^2) \left( \frac{1}{(\sigma^2 + a^2)^{3/2}}  + \frac{1}{(\sigma^2 + b^2)^{3/2}} \right)\right], \nonumber \\ 
P_2(\sigma, z, z^\prime) \coloneqq \frac{\pi^3}{4 (b^2-a^2)^3} \left[\frac{4 b^2}{\sqrt{\sigma^2 + b^2}} - \frac{4 a^2}{\sqrt{\sigma^2 + a^2}} + (b^2-a^2) \left( \frac{2\sigma^2 + a^2}{(\sigma^2 + a^2)^{3/2}}  + \frac{2 \sigma^2 + b^2}{(\sigma^2 + b^2)^{3/2}} \right)\right]. \nonumber 
\end{eqnarray} 
\end{subequations}
\end{lemma}
\begin{proof}[Proof of Proposition \ref{Poisson's equation: uppper bound on grad phi prop}]
Using the results from the above lemmas, a suitable function $H(\sigma, z, z^\prime)$ that works in (\ref{The Poisson's equation: the function H and a few consequences}) is 
\begin{eqnarray}
H(\sigma, z, z^\prime) \coloneqq c \left( \frac{1}{\sigma} \left[\frac{1}{\sqrt{\sigma^2 + a^2}} - \frac{1}{\sqrt{\sigma^2 + b^2}}\right] + |\tan \pi z| \left[\frac{1}{\sqrt{\sigma^2 + a^2}} - \frac{1}{\sqrt{\sigma^2 + b^2}}\right] \right. \nonumber \\ \left. + \cos^2 \pi z \cos \pi z^\prime \left[P_1(\sigma, z, z^\prime) + P_2(\sigma, z, z^\prime)\right] + \frac{\cos \pi z^\prime }{\sigma} \right), \nonumber
\end{eqnarray}
when $(\sigma, z, z^\prime) \in (0, 1/2\pi) \times (-1/2, 1/2) \times (-1/2, 1/2)$ and
\begin{eqnarray}
H(\sigma, z, z^\prime) \coloneqq c\cos \pi z^\prime \exp(-\pi \sigma), \nonumber
\end{eqnarray}
when $(\sigma, z, z^\prime) \in [1/2\pi, \infty) \times (-1/2, 1/2) \times (-1/2, 1/2)$. Here, $c > 0$ is some positive constant. With this definition of the function $H$ and Lemma \ref{Poisson's equation: a few important basic integrals} from appendix \ref{Appendix: Bounds on a few integrals}, we can obtain a bound on the integral (\ref{The Poisson's equation: integral of H statement}) as 
\begin{eqnarray}
\int_{0}^\infty  \sigma H(\sigma, z, z^\prime) \, {\rm d} \sigma && \lesssim  \log \left(1 + \frac{4 (b^2-a^2)}{3 a^2}\right) + \frac{\cos \pi z^\prime}{b} +  \frac{\cos^2 \pi z \cos \pi z^\prime}{b^3} + \cos \pi z^\prime \nonumber \\
&& \lesssim \log \left(1 + \frac{(b^2-a^2)}{a^2}\right) +  \frac{\cos \pi z^\prime}{b}.
\label{The Poisson's equation: integral of H bound}
\end{eqnarray}
Here, we used (\ref{The Poisson's equation: basic trig inequalities}{\color{blue}a}) to obtain the last line.
\end{proof}

\begin{proof}[Proof of Lemma \ref{Poisson's equation: bound on I derivative large sigma}]
We first note that 
\begin{eqnarray}
\int_1^\infty (|I_{\sigma 1}| + |I_{\sigma 2}| + |I_{\sigma 3}|) \, \frac{{\rm d} \tau}{\sqrt{\tau^2 - 1}} \leq \int_1^2 (|I_{\sigma 1}| + |I_{\sigma 2}| + |I_{\sigma 3}|) \, \frac{{\rm d} \tau}{\sqrt{\tau^2 - 1}} \qquad \qquad \qquad \nonumber \\  + 2 \int_2^\infty (|I_{\sigma 1}| + |I_{\sigma 2}| + |I_{\sigma 3}|) \, \frac{{\rm d} \tau}{\tau}.
\label{Poisson's equation: bound on I sigma sum simplified}
\end{eqnarray}
We also have
\begin{eqnarray}
\int_1^\infty (|I_{z 1}| + |I_{z 2}| + |I_{z 3}|) \, \frac{{\rm d} \tau}{\sqrt{\tau^2 - 1}} \leq \int_1^2 (|I_{z 1}| + |I_{z 2}| + |I_{z 3}|) \, \frac{{\rm d} \tau}{\sqrt{\tau^2 - 1}} \qquad \qquad \qquad \nonumber \\ + \int_2^\infty (|I_{z 1}| + |I_{z 2}| + |I_{z 3}|) \, {\rm d} \tau. 
\label{Poisson's equation: bound on I z sum simplified}
\end{eqnarray}
Now the assumption in the lemma is $\sigma \geq 1/2 \pi$. So, if $\tau \geq 1$, then 
$$\cosh(\pi \tau \sigma) - 1 \geq \frac{\cosh(\pi \tau \sigma)}{8},$$
and we always have 
$$\sinh(\pi \tau \sigma) \leq \cosh(\pi \tau \sigma) \quad \text{and} \quad \frac{\exp(\pi \tau \sigma)}{2} \leq \cosh(\pi \tau \sigma).$$
Using these relations in (\ref{The Poisson's equation: I sigma z derivative parts}{\color{blue}a-f}), one can show
\begin{eqnarray}
|I_{\sigma 1}| \lesssim \cos(\pi z^\prime) \tau \exp(-\pi \tau \sigma), \quad |I_{\sigma 2}| \lesssim \cos(\pi z^\prime) \tau \exp(-\pi \tau \sigma), \quad |I_{\sigma 3}| \lesssim \cos(\pi z^\prime) \tau \exp(-\pi \tau \sigma), \nonumber \\
|I_{z 1}| \lesssim \cos(\pi z^\prime) \exp(-\pi \tau \sigma), \quad |I_{z 2}| \lesssim \cos(\pi z^\prime) \exp(-2 \pi \tau \sigma), \quad |I_{z 3}| \lesssim \cos(\pi z^\prime) \exp(-2 \pi \tau \sigma). \nonumber
\end{eqnarray}
In total, we obtain
\begin{subequations}
\begin{eqnarray}
&& |I_{\sigma 1}| + |I_{\sigma 2}| + |I_{\sigma 3}| \lesssim \cos(\pi z^\prime) \tau \exp(-\pi \tau \sigma), \\
&& |I_{z 1}| + |I_{z 2}| + |I_{z 3}| \lesssim \cos(\pi z^\prime) \exp(-\pi \tau \sigma).
\end{eqnarray}
\label{Poisson's equation: bound on I sum bound}
\end{subequations}
Next, we substitute (\ref{Poisson's equation: bound on I sum bound}{\color{blue}a}) in (\ref{Poisson's equation: bound on I sigma sum simplified}) and (\ref{Poisson's equation: bound on I sum bound}{\color{blue}b}) in (\ref{Poisson's equation: bound on I z sum simplified}). We also use the fact $\exp(-\pi \tau \sigma) \leq \exp(-\pi \sigma)$ for the integrals carried from $\tau = 1$ to $\tau = 2$ in (\ref{Poisson's equation: bound on I sigma sum simplified}) and (\ref{Poisson's equation: bound on I z sum simplified}), which leads to
\begin{eqnarray}
\int_1^\infty (|I_{\sigma 1}| + |I_{\sigma 2}| + |I_{\sigma 3}|) \, \frac{{\rm d} \tau}{\sqrt{\tau^2 - 1}} && \lesssim \cos(\pi z^\prime) \exp(-\pi \sigma) + \cos(\pi z^\prime) \frac{\exp(-2\pi\sigma)}{\pi \sigma} \nonumber \\
&& \lesssim \cos(\pi z^\prime) \exp(-\pi \sigma),
\label{Poisson's equation: bound on I sigma sum lemma proved}
\end{eqnarray}
and 
\begin{eqnarray}
\int_1^\infty (|I_{z 1}| + |I_{z 2}| + |I_{z 3}|) \, \frac{{\rm d} \tau}{\sqrt{\tau^2 - 1}} && \lesssim \cos(\pi z^\prime) \exp(-\pi \sigma) + \cos(\pi z^\prime) \frac{\exp(-2\pi\sigma)}{\pi \sigma} \nonumber \\
&& \lesssim \cos(\pi z^\prime) \exp(-\pi \sigma).
\label{Poisson's equation: bound on I z sum lemma proved}
\end{eqnarray}
\end{proof}

\begin{proof}[Proof of Lemma \ref{Poisson's equation: bound on I derivative small sigma large tau}]
The proof is similar to the proof of lemma \ref{Poisson's equation: bound on I derivative large sigma}.
\end{proof}

\begin{proof}[Proof of Lemma \ref{Poisson's equation: bound on I derivative small sigma small tau}]
First, we establish a few simple relations. The assumption in the lemma is $\sigma < \frac{1}{2 \pi}$. So, if $1 \leq \tau \leq \frac{1}{\pi \sigma}$, then
$$\sinh(\pi \sigma \tau) \leq (\pi \sigma \tau) \sinh(1), \quad \cosh(\pi \sigma \tau) \leq \cosh(1),  \quad \text{and} \quad \cosh(\pi \sigma \tau) \leq 1 + (\cosh(1)-1) \pi^2 \sigma^2 \tau^2,$$
and we always have 
$$\cosh(\pi \sigma \tau) \geq 1 + \frac{(\pi \sigma \tau)^2}{2}.$$

(i) We can then use the relations above to derive a simple bound on $|I_{\sigma 1}|$:
\begin{eqnarray}
|I_{\sigma 1}| \leq I_{\sigma 1}^b \coloneqq
\frac{2 \cosh(1) \cos (\pi z) \, \cos (\pi z^\prime) \, \tau }{ \pi^4 \left[ \sigma^2 \tau^2 + a^2 \right] \left[ \sigma^2 \tau^2 + b^2 \right]}. \nonumber
\end{eqnarray}
We can also obtain a simple bound on $|I_{\sigma 2}|$ as follows
\begin{eqnarray}
|I_{\sigma 2}| && \leq
\frac{4 \sinh^2(1) \cos (\pi z) \, \cos (\pi z^\prime) \,  \sigma^2 \tau^3 }{ \pi^4 \left[ \sigma^2 \tau^2 + a^2 \right]^2 \left[ \sigma^2 \tau^2 + b^2 \right]} \nonumber \\ && \leq \frac{4 \sinh^2(1) \cos (\pi z) \, \cos (\pi z^\prime) \,  \tau }{ \pi^4 \left[ \sigma^2 \tau^2 + a^2 \right] \left[ \sigma^2 \tau^2 + b^2 \right]} = \frac{2 \sinh^2(1)}{\cosh(1)} I_{\sigma 1}^b. \nonumber
\end{eqnarray}
With a similar calculation, we prove that the same bound also holds for $|I_{\sigma 3}|$. We can now finish the proof as given below
\begin{eqnarray}
\int_1^{\frac{1}{\pi \sigma}} |I_{\sigma 1}| + |I_{\sigma 2}| + |I_{\sigma 3}| \, \frac{{\rm d} \tau}{\sqrt{\tau^2 - 1}} \lesssim \int_1^\infty |I_{\sigma 1}^b| \, \frac{{\rm d} \tau}{\sqrt{\tau^2 - 1}} \lesssim \frac{1}{\sigma} \left[\frac{1}{\sqrt{\sigma^2 + a^2}}   -   \frac{1}{\sqrt{\sigma^2 + b^2}}\right]. \nonumber
\end{eqnarray}

(ii) We can obtain a following simple bound on $|I_{z 1}|$ as
\begin{eqnarray}
|I_{z 1}| \leq \frac{2 \pi \sigma \sinh(1) |\sin \pi z| \cos \pi z^\prime \tau }{\pi^4 \left[ \sigma^2 \tau^2 + a^2 \right] \left[ \sigma^2 \tau^2 + b^2 \right]} = (\pi \sigma) \tanh(1) |\tan(\pi z)| I_{\sigma 1}^b.  \nonumber
\end{eqnarray}
Performing an integration in $\tau$ as in part (i) leads to the desired result.

(iii) We first obtain a simple bound on the sum $I_{z 2} + I_{z 3}$ given as follows
\begin{eqnarray}
&& I_{z 2} + I_{z 3} =
\frac{\cos^2 (\pi z) \cos (\pi z^\prime) \sinh (\pi \tau \sigma ) \left[\sin \pi z^\prime \cosh(\pi \sigma \tau) - \sin \pi z \right]}{ \left[ \cosh (\pi \tau \sigma ) + \cos \pi (z^\prime + z) \right]^2 \left[ \cosh (\pi \tau \sigma ) - \cos \pi (z^\prime - z) \right]^2}, \nonumber \\
\implies && |I_{z 2} + I_{z 3}| \leq \frac{16}{\pi^7} \cos^2 (\pi z) \cos (\pi z^\prime)
\frac{ \sigma \tau \left[|\sin \pi z^\prime - \sin \pi z \right| + \pi^2 \sigma^2 \tau^2 ]}{ \left[ \sigma^2 \tau^2 + a^2 \right]^2 \left[ \sigma^2 \tau^2 + b^2 \right]^2}. \nonumber
\end{eqnarray}
This result, combined with the following integrals
\begin{subequations}
\begin{eqnarray}
&& \int_{1}^{\infty} \frac{\sigma \tau}{ \left[ \sigma^2 \tau^2 + a^2 \right]^2 \left[ \sigma^2 \tau^2 + b^2 \right]^2} \, \frac{{\rm d} \tau}{\sqrt{\tau^2 - 1}} = \nonumber \\ && \frac{\pi}{4 (b^2-a^2)^3} \left[\frac{4}{\sqrt{\sigma^2 + b^2}} - \frac{4}{\sqrt{\sigma^2 + a^2}} + (b^2-a^2) \left( \frac{1}{(\sigma^2 + a^2)^{3/2}}  + \frac{1}{(\sigma^2 + b^2)^{3/2}} \right)\right], \nonumber \\
&& \int_{1}^{\infty} \frac{\sigma^3 \tau^3}{ \left[ \sigma^2 \tau^2 + a^2 \right]^2 \left[ \sigma^2 \tau^2 + b^2 \right]^2} \, \frac{{\rm d} \tau}{\sqrt{\tau^2 - 1}} = \nonumber \\ && \frac{\pi}{4 (b^2-a^2)^3} \left[\frac{4 b^2}{\sqrt{\sigma^2 + b^2}} - \frac{4 a^2}{\sqrt{\sigma^2 + a^2}} + (b^2-a^2) \left( \frac{2\sigma^2 + a^2}{(\sigma^2 + a^2)^{3/2}}  + \frac{2 \sigma^2 + b^2}{(\sigma^2 + b^2)^{3/2}} \right)\right], \nonumber
\end{eqnarray}
\end{subequations}
leads to the desired result.
\end{proof}

\section{Discussion}
\label{Discussion}
In this paper, we studied the problem of optimizing the heat transfer between two differentially heated parallel plates by incompressible flows that satisfy an enstrophy constraint ($\langle |\nabla \bs{u}|^2 \rangle \leq \mathscr{P}$) and no-slip boundary conditions. The main result of this paper was to show that the previously derived upper bound on the heat transfer are sharp in the scaling with $\mathscr{P}$, which we demonstrated by constructing an explicit example of three-dimensional branching pipe flows. In this section, we discuss the implications of our result in the context of (1) anomalous dissipation in a passive scalar and (2) Rayleigh--B\'enard convection. 
\subsection{Anomalous dissipation in a passive scalar}
The initial motivation for our study was a result by Drivas \textit{et al.} (\cite{drivas22anomdissp}), regarding the anomalous dissipation in a passive scalar transport. They constructed a velocity field $\bs{u} \in C^\infty([0, \tau) \times \mathbb{T}^d) \cap L^1([0, \tau]; C^\alpha(\mathbb{T}^d))$, where $d \geq 2$, $\tau$ is a fixed time and $\alpha < 1$, such that the solution of the advection-diffusion equation $$\partial_t T^\kappa + \bs{u} \cdot \nabla T^\kappa = \kappa \Delta T^\kappa$$
follows
\begin{eqnarray}
\liminf_{\kappa \to 0} \kappa \int_{0}^{\tau} \int_{\mathbb{T}^d} |\nabla T^\kappa|^2 \, d \bs{x} \; d t  \geq \chi > 0, \nonumber
\end{eqnarray}
where $\chi$ may depend on the initial data. While this result was obtained for a periodic domain, we were inspired by the possibility of proving such a result in a domain with boundaries. After appropriately rescaling the velocity fields that we created to prove Theorem \ref{Main theorem steady case}, we can state a weak result in this direction.
\begin{corollary}
For a constant $\kappa_0 > 0$, there exist velocity fields $\bs{u}^\kappa$, for every $0 < \kappa < \kappa_0$, such that $\norm{\bs{u}^\kappa}_{H^1_0(\Omega)} \leq 1$ and the solution of the steady advection diffusion equation: $\bs{u}^\kappa \cdot \nabla T^\kappa = \kappa \Delta T^\kappa$ in $\Omega$ with boundary conditions $T^\kappa=1$ at $z=-1/2$ and $T^\kappa=0$ at $z=1/2$ obeys 
\begin{eqnarray}
\liminf_{\kappa \to 0} \kappa^{2/3} \dashint_{\Omega} |\nabla T^\kappa|^2 \, d \bs{x}  \geq \chi_0 > 0.
\label{anom dissp}
\end{eqnarray}
for a constant $\chi_0$.
\label{anom dissp miss}
\end{corollary}
We see from (\ref{anom dissp}) that the exponent for $\kappa$ is $2/3$, which is less than one. Therefore this corrolary is not as strong as the statement we would have hoped to prove.  However, the a priori upper bound (\ref{upper bound Qu}) also shows that this is the best result one can achieve in the setting considered in Corollary \ref{anom dissp miss}. However, if we allow the velocity field to be less smooth, in particular, we allow $\bs{u}^\kappa$ to be only uniformly bounded in the $L^2$ norm, then we can indeed prove
\begin{eqnarray}
\liminf_{\kappa \to 0} \kappa \dashint_{\Omega} |\nabla T^\kappa|^2 \, d \bs{x}  \geq \chi_0 > 0.
\label{anom dissp rough}
\end{eqnarray}
This can be shown after appropriately rescaling the velocity fields of \cite{doering2019optimal} used to prove Theorem 1.1 in their paper. Another possibility is to allow the walls to be rough. This has not yet,  to our knowledge, been investigated, which raises the following question: if we allow the boundary of the domain, which locally is the graph of functions that are $\alpha$--H\"older continuous with exponent $\alpha < 1$, can one also prove (\ref{anom dissp rough}) in that case? Physically, it would mean that we are increasing the heat transfer by letting the area of walls go to infinity. Indeed, it is known in the literature that fractal boundaries tend to enhance heat transfer (\cite{toppaladoddi21fractal}). Answer to such a question will, therefore, help in understanding the role played by rough boundaries in increasing the heat transfer. Along the same line, it would also be interesting to investigate the role played by a slip boundary condition for the velocity field (see \cite{drivas2022bounds} and a recent review by \cite{nobili2021role}).

\subsection{Rayleigh--B\'enard convection }
Rayleigh--B\'enard convection is the flow of fluid between two differentially heated parallel plates driven by buoyancy force. The flow is traditionally modeled by the Navier--Stokes equations under the Boussinesq approximation, written here in nondimensional form as 
\begin{subequations}
\begin{eqnarray}
\partial_t \bs{u} + \bs{u} \cdot \nabla \bs{u} = - \nabla p + Pr \Delta \bs{u} + Pr Ra T \bs{e}_z, \\
\partial_t T + \bs{u} \cdot \nabla T = \Delta T, \qquad \qquad \quad
\end{eqnarray}
\label{RB eqn}
\end{subequations}
where $Ra$ is the Rayleigh number and $Pr$ is the Prandtl number, respectively given by
\begin{eqnarray}
Ra = \frac{g \alpha H^3 (T_B - T_T)}{\kappa \nu}, \qquad Pr = \frac{\nu}{\kappa}. \nonumber
\end{eqnarray}
In these above expressions, $\nu$ is the kinematic viscosity, $\kappa$ is the thermal diffusivity, $\alpha$ is the coefficient of thermal expansion, $H$ is the height of the domain, $T_B - T_T$ is the temperature difference and $g$ is the magnitude of the gravitational acceleration acting in $-\bs{e}_z$ direction.

We solve the nondimensional governing equations (\ref{RB eqn}{\color{blue}a-b}) in domain $\Omega$ with boundary conditions
\begin{eqnarray}
\bs{u} = \bs{0}, \quad T = 1 \quad \text{ at } z = -1/2 \qquad \text{and} \qquad \bs{u} = \bs{0}, \quad T = 0 \quad \text{ at } z = 1/2. \nonumber
\end{eqnarray}

The quantity of interest is the nondimensional heat transfer known as the Nusselt number $Nu$ given by $$Nu = 1 + \langle u_z T \rangle.$$ The angle brackets denote the long-time volume average and $u_z$ is the component of the velocity in the $z$ direction. Of course, $Nu$ depends on the initial condition. However, Doering and Constantin (\cite{doering1996variational}) using the background method (see \cite{fantuzzi2022background} for a survey), proved the following a priori bound for any initial condition when $Ra \gg 1$:
$$Nu \lesssim Ra^{\frac{1}{2}}.$$
This bound is uniform in the Prandtl number $Pr$. To date, the best known upper bound, namely $Nu \leq 0.02634 Ra^{\frac{1}{2}}$, was obtained by Plasting and Kerswell (\cite{Plasting03Couette}). 

An important question is whether the scaling of this bound with respect to the Rayleigh number is sharp. Our result, in this context, proves that this scaling is indeed sharp if one replaces the momentum equation with a simple enstrophy condition.
\begin{eqnarray}
\langle |\nabla \bs{u}|^2 \rangle = Ra (Nu - 1).
\label{RB enstrophy rel}
\end{eqnarray}
In other words Theorem \ref{Main theorem steady case} proves that, for large enough Rayleigh number, there exists velocity fields (depending on $Ra$) such that the solution of the advection-diffusion equation (\ref{RB eqn}{\color{blue}b}) satisfies the relation (\ref{RB enstrophy rel}) and for which $$Nu \sim Ra^{\frac{1}{2}}.$$

\appendix
    \begin{center}
      {\bf APPENDIX}
    \end{center}

\section{Derivation of the variational principle for heat transfer (\ref{steady var prin})}
\label{Derivation var prin}
In this appendix, we derive the variational principle given in (\ref{steady var prin}). The proof is taken from the paper of Doering \& Tobasco (\cite{doering2019optimal}) and provided here for completeness. We begin by recalling 
\begin{eqnarray}
Q(\bs{u}) = \dashint_{\Omega} |\nabla T|^2 \, {\rm d} \bs{x} = 1 + \dashint_{\Omega} u_z T \, {\rm d} \bs{x},
\end{eqnarray}
where $T$ solves the steady convection diffusion equation and the velocity is assumed to be in $L^\infty(\Omega; \mathbb{R}^3)$. After substituting $$T = \theta + \frac{1}{2} - z,$$ we obtain
\begin{eqnarray}
Q(\bs{u}) = 1 + \dashint_{\Omega} |\nabla \theta|^2 \, {\rm d} \bs{x} = 1 + \dashint_{\Omega} u_z \theta \, {\rm d} \bs{x},
\label{heat theta}
\end{eqnarray}
where $\theta \in H^1_0(\Omega)$ solves $$\bs{u} \cdot \nabla \theta = \Delta \theta + u_z \qquad \text{in} \qquad \Omega,$$
and satisfies the homogeneous Dirichlet boundary conditions. Next, we consider a system of two PDEs.
\addtocounter{equation}{1}
\begin{align}
\begin{rcases}
\bs{u} \cdot \nabla \eta_0 = \Delta \xi_0 + u_z, \\
\bs{u} \cdot \nabla \xi_0 = \Delta \eta_0,
\end{rcases} \qquad \text{in} \quad \Omega,
\tag{\theequation a-b}
\end{align}
where both $\eta_0$ and $\xi_0$ satisfy the homogeneous Dirichlet boundary conditions. Clearly then 
\begin{eqnarray}
\theta = \eta_0 + \xi_0.
\label{theta eta xi}
\end{eqnarray}
By multiplying the equation ({\color{blue}A.3b}) with $\xi_0$ and integrating, one obtains 
\begin{eqnarray}
\int_{\Omega} \nabla \eta_0 \cdot \nabla \xi_0 \, {\rm d}\bs{x} = 0.
\label{grad eta grad xi}
\end{eqnarray}
Next, we multiply the equation ({\color{blue}A.3a}) with $\eta_0$ and integrate and use (\ref{grad eta grad xi}) to obtain
\begin{eqnarray}
\int_{\Omega} u_z \eta_0 \, {\rm d}\bs{x} = 0.
\label{uz eta}
\end{eqnarray}
Finally, we substitute (\ref{theta eta xi}) in (\ref{heat theta}) and use (\ref{grad eta grad xi}) and (\ref{uz eta}) to get
\begin{eqnarray}
Q(\bs{u}) - 1 = \dashint_{\Omega} |\nabla \xi_0|^2 + |\nabla \eta_0|^2 \, {\rm d} \bs{x} = \dashint_{\Omega} u_z \xi_0 \, {\rm d} \bs{x},
\label{heat eta xi}
\end{eqnarray}
which after using ({\color{blue}A.3b}) can be rewritten as 
\begin{eqnarray}
Q(\bs{u}) - 1 = 2 \dashint_{\Omega} u_z \xi_0 \, {\rm d} \bs{x} - \dashint_{\Omega} |\nabla \xi_0|^2  \, {\rm d} \bs{x} - \dashint_{\Omega}   |\nabla \Delta^{-1} \bs{u} \cdot \nabla \xi_0|^2 \, {\rm d} \bs{x},
\label{heat eta xi rewrite}
\end{eqnarray}
where $\Delta^{-1}$ denotes the inverse Laplacian operator in $\Omega$ corresponding to the homogeneous Dirichlet boundary conditions. We now consider the following maximization problem
\begin{eqnarray}
\sup_{\xi \in H^1_0(\Omega)} 2 \dashint_{\Omega} u_z \xi \, {\rm d} \bs{x} - \dashint_{\Omega} |\nabla \xi|^2  \, {\rm d} \bs{x} - \dashint_{\Omega}   |\nabla \Delta^{-1} \bs{u} \cdot \nabla \xi|^2 \, {\rm d} \bs{x},
\label{a variational problem}
\end{eqnarray}
which is strictly concave, therefore, the only maximizer satisfies the Euler--Lagrange equation $$\bs{u} \cdot \nabla \Delta^{-1} (\bs{u} \cdot \nabla \xi) = \Delta \xi + u_z.$$
We see that $\xi_0$ is the solution of above equation, therefore, $\xi_0$ maximizes (\ref{a variational problem}), which then combined with (\ref{heat eta xi rewrite}) gives us
\begin{eqnarray}
Q(\bs{u}) - 1 = \max_{\xi \in H^1_0(\Omega)} 2 \dashint_{\Omega} u_z \xi \, {\rm d} \bs{x} - \dashint_{\Omega} |\nabla \xi|^2  \, {\rm d} \bs{x} - \dashint_{\Omega}   |\nabla \Delta^{-1} \bs{u} \cdot \nabla \xi|^2 \, {\rm d} \bs{x}.
\end{eqnarray}
To make this formulation homogeneous in the variable $\xi$, we consider the transformation $\xi \to s \xi$ and optimize in the scaling $s$ to obtain 
\begin{eqnarray}
Q(\bs{u}) - 1 = \sup_{\substack{ \xi \in H^1_0(\Omega) \\ \xi \not\equiv 0 }} \; \frac{\left( \dashint_{\Omega} u_z \xi \, {\rm d} \bs{x}\right)^2}{\dashint_{\Omega} |\nabla \xi|^2 \, {\rm d} \bs{x} + \dashint_{\Omega} |\nabla \Delta^{-1} \diverge (\bs{u} \xi)|^2 \, {\rm d} \bs{x}}.
\label{The variational principle: main corollary proof int}
\end{eqnarray}
Next, from the definition of $Q^s_{\max}(\mathscr{P})$ given in (\ref{Qmax steady advec-diff}), we simply obtain
\begin{eqnarray}
Q^s_{\max}(\bs{u}) - 1 = \sup_{\substack{\bs{u} \in L^\infty(\Omega) \\ \nabla \cdot \bs{u} = 0, \; \left. \bs{u} \right|_{\partial \Omega} = \bs{0} \\ \dashint_{\Omega} |\nabla \bs{u}|^2 \, {\rm d} \bs{x} \leq \mathscr{P}}}  \sup_{\substack{ \xi \in H^1_0(\Omega) \\ \xi \not\equiv 0 }} \; \frac{\left( \dashint_{\Omega} u_z \xi \, {\rm d} \bs{x}\right)^2}{\dashint_{\Omega} |\nabla \xi|^2 \, {\rm d} \bs{x} + \dashint_{\Omega} |\nabla \Delta^{-1} \diverge (\bs{u} \xi)|^2 \, {\rm d} \bs{x}}.
\label{The variational principle: main corollary proof int}
\end{eqnarray}
Finally, considering the following transformation
\begin{eqnarray}
\bs{u} \to \, \frac{\mathscr{P}^{\frac{1}{2}}} {\dashint_{\Omega} |\nabla \bs{u}|^2 \, {\rm d} \bs{x}} \, \bs{u}, \qquad \xi \to \frac{\dashint_{\Omega} |\nabla \bs{u}|^2 \, {\rm d} \bs{x}}{\mathscr{P}^{\frac{1}{2}}} \, \xi, 
\end{eqnarray}
in (\ref{The variational principle: main corollary proof int}) leads to the desired result stated in Proposition \ref{steady var prin}.

\section{Bounds on a few integrals}
\label{Appendix: Bounds on a few integrals}

\begin{lemma}
\label{Poisson's equation: a few important basic integrals}
Let $a$ and $b$ be as given in (\ref{The Poisson's equation: a b placeholders}). Let $z, z^\prime \in (-1/2, 1/2)$ and $z \neq z^\prime$, then the following estimates hold

(i) 
\begin{eqnarray}
\int_0^1 \left[\frac{1}{\sqrt{\sigma^2 + a^2}} - \frac{1}{\sqrt{\sigma^2 + b^2}}\right] {\rm d} \sigma \leq \frac{1}{2} \log \left(1 + \frac{4 (b^2-a^2)}{3 a^2}\right). 
\end{eqnarray}

(ii)

\begin{eqnarray}
\int_0^1 \left[\frac{\sigma}{\sqrt{\sigma^2 + a^2}} - \frac{\sigma}{\sqrt{\sigma^2 + b^2}}\right] {\rm d} \sigma \leq \frac{b^2-a^2}{b}.
\end{eqnarray}

(iii)

\begin{eqnarray}
\int_0^1 \sigma P_1(\sigma, z, z^\prime) \, {\rm d} \sigma \leq \frac{\pi}{2 b^3}. 
\end{eqnarray}

(iv)

\begin{eqnarray}
\int_0^1 \sigma P_2(\sigma, z, z^\prime) \, {\rm d} \sigma \leq \frac{\pi^3}{4 b^3}. 
\end{eqnarray}

\end{lemma}

\begin{proof}[Proof of Lemma \ref{Poisson's equation: a few important basic integrals}]
\noindent
\\
Recall from the definition (\ref{The Poisson's equation: a b placeholders}) of $a$ and $b$ that $a \leq b$, a fact which we will use in the proofs below.\\
(i)
\begin{eqnarray}
\int_0^1 \left[\frac{1}{\sqrt{\sigma^2 + a^2}} - \frac{1}{\sqrt{\sigma^2 + b^2}}\right] {\rm d} \sigma && = \frac{1}{2} \log \frac{1+\sqrt{1+a^2}}{1+\sqrt{1+b^2}}+\frac{1}{2} \log \frac{\sqrt{1+b^2}-1}{\sqrt{1+a^2}-1} \nonumber \\
&& \leq \frac{1}{2} \log \left(1 + \frac{\sqrt{1+b^2}-\sqrt{1+a^2}}{\sqrt{1+a^2}-1}\right) \nonumber \\
&& \leq \frac{1}{2} \log \left(1 + \frac{8 (b^2-a^2)}{3 a^2 \left[\sqrt{1+b^2}+\sqrt{1+a^2}\right]}\right) \nonumber \\
&& \leq \frac{1}{2} \log \left(1 + \frac{4 (b^2-a^2)}{3 a^2}\right).
\end{eqnarray}
(ii)
\begin{eqnarray}
\int_0^1 \left[\frac{\sigma}{\sqrt{\sigma^2 + a^2}} - \frac{\sigma}{\sqrt{\sigma^2 + b^2}}\right] {\rm d} \sigma && = \sqrt{1+a^2} - \sqrt{1+b^2} + b - a \leq \frac{b^2-a^2}{b+a} \leq \frac{b^2-a^2}{b}.
\end{eqnarray}
(iii)
\begin{eqnarray}
\int_0^1 \sigma P_1(\sigma, z, z^\prime) \, {\rm d} \sigma = \frac{\pi a b}{2} \left[\frac{1}{a b (b+a)^3} - \frac{1}{\sqrt{1+a^2} \sqrt{1+b^2} (\sqrt{1+b^2} + \sqrt{1+a^2})^3}\right] \leq \frac{\pi}{2 b^3}.
\end{eqnarray}
(iv)
\begin{eqnarray}
\int_0^1 \sigma P_2(\sigma, z, z^\prime) \, {\rm d} \sigma = \frac{\pi^3}{4} \left[\frac{1}{(b+a)^3} - \frac{1}{(\sqrt{1+b^2} + \sqrt{1+a^2})^3} \right. \qquad \qquad \qquad \qquad \qquad \qquad \nonumber \\ \left. - \frac{1}{\sqrt{1+a^2} \sqrt{1+b^2} (\sqrt{1+b^2} + \sqrt{1+a^2})^3}\right] \leq \frac{\pi^3}{4 b^3}.
\end{eqnarray}

\end{proof}

\section{A few basic lemmas}
\label{Appendix: A few basic lemmas}
\begin{lemma} \label{A few basic lemmas: missing diff at one point}
Let $f \in C^0(\mathbb{R})$ such that $f^\prime(r)$ exists for all $r \neq 0$. Now assume that $f^\prime(r) \to \mathfrak{f}^\prime_0$ as $r \to 0$ for some finite $ \mathfrak{f}^\prime_0$, then $f^\prime(0)$ exists and its value is $ \mathfrak{f}^\prime_0$.
\end{lemma}
\begin{proof}[Proof of Lemma \ref{A few basic lemmas: missing diff at one point}]
Using the mean value theorem, we have
\begin{eqnarray}
 \frac{f(h) - f(0)}{h} =  f^\prime(\eta) \qquad \text{for some } \eta \in (0, h),
\end{eqnarray}
The proof of lemma follows by taking $h \to 0$.
\end{proof}
\begin{definition}
A function $f : \mathbb{R} \to \mathbb{R}$ is said to have the property (N) if
\begin{eqnarray}
f \in C^\infty(\mathbb{R}) \qquad \text{and} \qquad
\left.\frac{d^{2n-1}f(r)}{d \, r^{2n-1}} \right|_{r = 0} = 0 \quad \forall \; n \in \mathbb{N}.
\end{eqnarray}
\end{definition}
It is clear that if $f(r)$ has the property (N) then so does the $f(\alpha r)$ for any $\alpha \neq 0$. Furthermore, we have the following lemma.
\begin{lemma} \label{A few basic lemmas: property N}
If a function $f : \mathbb{R} \to \mathbb{R}$ has the property \text{(N)} then so does the function $g : \mathbb{R} \to \mathbb{R}$ defined as
\begin{eqnarray}
g(r) \coloneqq 
\begin{cases}
\frac{1}{r} \frac{d f(r)}{d r} \quad \text{if} \quad r \neq 0, \\
f^{\prime \prime}(0) \quad \text{if} \quad r = 0,
\end{cases}
\end{eqnarray}
has the property \text{(N)}.
\end{lemma}
\begin{proof}[Proof of Lemma \ref{A few basic lemmas: property N}]
It is clear that $g(r)$ is continuous when $r \neq 0$. Now from L'Hospital's rule we obtain
\begin{eqnarray}
\lim_{r \to 0} \frac{1}{r} \frac{d f(r)}{d r} = \lim_{r \to 0} \frac{d^2 f(r)}{d r^2} = f^{\prime \prime}(0),
\label{A few basic lemmas: cont. of g}
\end{eqnarray} 
therefore, using Lemma \ref{A few basic lemmas: missing diff at one point}, $g(r)$ is continuous at $r = 0$ as well. Now, for $n \geq 1$, we have
\begin{eqnarray}
\frac{d^n g(r)}{d r^n} = \frac{n !}{r^{n+1}} \sum_{i = 0}^{n} \frac{(-1)^{n-i} r^i}{i!} \frac{d^{i+1} f(r)}{d r^{i+1}} \qquad \text{if} \quad r \neq 0.
\end{eqnarray}
Taking the limit $r \to 0$ and using the L'Hospital's rule, we obtain
\begin{eqnarray}
\lim_{r \to 0} \frac{d^n g(r)}{d r^n} && = \lim_{r \to 0} \frac{n !}{r^{n+1}} \sum_{i = 0}^{n} \frac{(-1)^{n-i} r^i}{i!} \frac{d^{i+1} f(r)}{d r^{i+1}} \nonumber \\
&& = \lim_{r \to 0} \frac{n !}{(n+1) r^{n}} \left[\sum_{i = 0}^{n} \frac{(-1)^{n-i} r^i}{i!} \frac{d^{i+2} f(r)}{d r^{i+2}} - \sum_{i = 1}^{n} \frac{(-1)^{n-i-1} r^{i-1}}{(i-1)!} \frac{d^{i+1} f(r)}{d r^{i+1}}  \right] \nonumber \\
&& = \left. \frac{1}{(n+1)}\frac{d^{i+2} f(r)}{d r^{i+2}} \right|_{r=0}. 
\label{A few basic lemmas: the main formula n deri}
\end{eqnarray}
Noting that $g$ is continuous and using Lemma \ref{A few basic lemmas: missing diff at one point} and the formula (\ref{A few basic lemmas: the main formula n deri}), for $n=1$, we see that $g(r)$ is differentiable at $r = 0$, furthermore, $g^\prime(r)$ is continuous everywhere. Proceeding in a similar manner, an induction argument then shows that $g$ is infinitely differentiable. Once again, noting from the formula (\ref{A few basic lemmas: the main formula n deri}) that all the odd derivatives of $g$ are zero at $r = 0$, proves the lemma. 
\end{proof}

\begin{lemma} \label{A few basic lemmas: C inf vel fields}
Let $g: \mathbb{R}^2 \to \mathbb{R}$ be given by $g(y, z) = y^{\alpha_y} z^{\alpha_z} f(\varrho)$, where $\alpha_y$ and $\alpha_z$ are nonnegative integers and $\varrho$ is a placeholder for $\sqrt{y^2+z^2}$. Furthermore, the function $f:\mathbb{R} \to \mathbb{R}$ has the property (N). Then the function $g$ is infinitely differentiable.
\end{lemma}
\begin{proof}
We can prove this lemma using an induction argument combined with Lemma \ref{A few basic lemmas: property N}. 
\end{proof}

Now it is a standard exercise in classical real analysis to show that the functions $\ol{\varphi}(r)$, $h(r)$, $\Psi_s(r)$, $\Psi_e(r)$, $\frac{\Psi_s(r)}{r^2}$ and $\frac{\Psi_e(r)}{r^2}$, with relevant definitions given in (\ref{Construction: over bar phi}), (\ref{Construction: the function h})and (\ref{Construction: Psis Psie}), have the property (N). Using Lemma \ref{A few basic lemmas: C inf vel fields}, one can then conclude that the velocity field, as defined in (\ref{Construction of the flow: Red comp vel field}), is infinitely smooth.

\bibliographystyle{halpha-abbrv}
\bibliography{references} 

\end{document}